\documentclass[11pt,reqno]{amsart}

\usepackage[T1]{fontenc}
\usepackage[utf8]{inputenc}
\usepackage{amsmath,amssymb,amsfonts,amsthm,mathtools}
\numberwithin{equation}{section}
\usepackage[a4paper,margin=2.5cm]{geometry}
\usepackage{microtype}
\usepackage{libertinus}
\usepackage{xcolor}
\usepackage{tikz}
\usetikzlibrary{positioning}
\usepackage{hyperref}

\hypersetup{colorlinks=true,linkcolor=black,citecolor=black,urlcolor=black}

\newtheorem{maintheorem}{Theorem}

\newtheorem{maintheoremb}{Theorem}

\newtheorem{maintheoremc}{Theorem}

\newtheorem{theorem}{Theorem}[section]
\newtheorem{lemma}[theorem]{Lemma}
\newtheorem{proposition}[theorem]{Proposition}
\newtheorem{corollary}[theorem]{Corollary}
\theoremstyle{definition}
\newtheorem{definition}[theorem]{Definition}
\newtheorem{example}[theorem]{Example}
\theoremstyle{remark}
\newtheorem{remark}[theorem]{Remark}

\newcommand{\T}{\mathbb T}
\newcommand{\N}{\mathbb N}
\newcommand{\E}{\mathbb E}
\newcommand{\Pp}{\mathbb P}

\newcommand{\proofstep}[2]{%
  \par\medskip
  \noindent\textbf{Step #1: #2.}%
}

\usepackage{tabularx}
\begin{document}

\title{Sharp Summability on Supports of Prescribed Combinatorial Dimension}

\author{Anderson Barbosa}
\address{Universidade Federal Rural do Semi-\'Arido (UFERSA), Rio Grande do Norte, Brazil}
\email{anderson.barbosa@ufersa.edu.br}

\author{Daniel Pellegrino}
\address{Departamento de Matem\'atica, Universidade Federal da Para\'iba, Jo\~ao Pessoa, PB, Brazil}
\email{dmpellegrino@gmail.com}

\author{Anselmo Raposo}
\address{Departamento de Matem\'atica, Universidade Federal do Maranh\~ao, S\~ao Lu\'is, MA, Brazil}
\email{anselmo.junior@ufma.br}

\author{Eduardo Teixeira}
\address{Department of Mathematics, Oklahoma State University, Stillwater, OK, USA}
\email{eduardo.teixeira@okstate.edu}

\date{}

\begin{abstract}
We solve four questions raised by Bayart concerning coefficient summability for multilinear forms with prescribed supports. For every $m\ge 2$ and $d\in[1,m]$, we determine the product-summability exponent and the multilinear summability invariant:
\[
 \mathrm{prod}(m,d)
 =\min\left\{\frac md,\,m-\lceil d\rceil+1\right\},
 \qquad
 \gamma_{\mathrm{mult}}(m,d)
 =\min\left\{m-\lceil d\rceil+1,\frac{2m}{d+1}\right\}.
\]
In particular, $\gamma_{\mathrm{mult}}(4,2)=8/3$, showing that the multilinear invariant need not be an integer. We also prove that, for every $d\in[1,m]$, there is a single infinite support of exact combinatorial dimension $d$ on which the dimensional Hardy--Littlewood bound is attained over both scalar fields for every anisotropic parameter $\mathbf p=(p_1,\ldots,p_m)$ with $\sum_j1/p_j<1$, simultaneously across the two regimes separated by $\sum_j1/p_j=1/2$.
\end{abstract}

\maketitle
\setcounter{tocdepth}{2}
\section{Introduction}\label{sec:introduction}

Classical coefficient inequalities measure a multilinear form through the
summability of its coefficient array. The framework introduced by Bayart in
\cite{BayartJEMS} asks a more refined question: what survives when the
summation is restricted to an arbitrary support
\[
 \Lambda\subset\mathbb{N}^m,
\]
and how does the geometry of that support enter the optimal exponent? The
relevant invariant is its combinatorial dimension, which records the largest
possible growth of $\Lambda$ inside finite Cartesian boxes.

Bayart's formulation separates the analytic size of a multilinear form from
the combinatorial concentration of its coefficients and asks how much
summability can be recovered from geometric information alone. The resulting questions are simple to state, but their
resolution requires one to distinguish features of a support that are visible
to its global dimension from finer concentrations along coordinate sections.

The underlying geometry originates in Blei's theory of fractional Cartesian
products
\cite{Blei1979,Blei1984,BleiMemoir,BleiBook}. Related constructions and
dimension formulas appear in
\cite{BleiKorner,BleiSchmerl,BleiGao,BleiPeresSchmerl}. This perspective
naturally brings together fractional-cover inequalities of
Bollob\'as--Thomason and Finner
\cite{BollobasThomason,Finner1992}, the join bounds of
Atserias--Grohe--Marx \cite{AGM}, and H\"older--Brascamp--Lieb inequalities
\cite{BrascampLieb1976,BCCT,CDKTY}. These tools play complementary roles
below. Algebraic H\"older--Brascamp--Lieb data produce exact model supports,
while sparse selection allows their geometry to be transferred to dimensions
not directly represented by the algebraic constructions.

We begin with the central geometric notion.

\begin{definition}[\cite{BayartJEMS}, p.~1162]\label{def:comb-dim}
For a nonempty set $\Lambda\subset\mathbb{N}^m$, define its counting function by
\[
 \Lambda(n):=
 \max_{\substack{A_1,\ldots,A_m\subset\mathbb{N}\\
 |A_j|\le n\ (1\le j\le m)}}
 |\Lambda\cap(A_1\times\cdots\times A_m)|.
\]
Thus $\Lambda(n)$ is the largest number of points of $\Lambda$ that can lie
in a Cartesian product whose factors contain at most $n$ elements.

The \emph{combinatorial dimension} of $\Lambda$ is
\[
 \dim(\Lambda):=
 \inf\bigl\{s>0:\Lambda(n)=O(n^s)\bigr\}.
\]
Equivalently,
\[
 \dim(\Lambda)
 =
 \limsup_{n\to\infty}\frac{\log\Lambda(n)}{\log n},
\]
where the quotient is considered for $n\ge2$. If $d=\dim(\Lambda)$, we say
that $\Lambda$ has \emph{exact combinatorial dimension $d$} when the critical
growth is attained, namely, when there exists $C>0$ such that
\[
 \Lambda(n)\le Cn^d
 \qquad(n\in\mathbb{N}).
\]
\end{definition}

{
Four questions from Bayart's theory guide the paper. Problem~5.2 asks
for the largest product-summability exponent compatible with a prescribed
combinatorial dimension. Problem~5.1 asks for the corresponding multilinear
summability invariant when cancellation by unimodular coefficients is allowed.
Question~5.7 asks whether this invariant must always be an integer. A fourth
question, raised in the discussion preceding \cite[Theorem~1.2]{BayartJEMS},
asks whether the dimensional Hardy--Littlewood estimate of
\cite[Theorem~1.1]{BayartJEMS} can be attained by a support of every
prescribed dimension.

These problems probe complementary aspects of the same geometry. Product
summability detects positive mass and coordinate concentration; the
multilinear invariant measures the additional room created by cancellation;
and the Hardy--Littlewood problem asks whether the universal dimensional bound
can be realized sharply on one support. We determine both extremal invariants
for every $d\in[1,m]$ and construct, at each prescribed dimension, one support
that is simultaneously sharp for every anisotropic parameter in the full
Hardy--Littlewood range.
}

\subsection{The product-summability problem}

Intermediate dimensions arise naturally from fractional Cartesian products,
developed below. The associated
extremal exponent is the quantity introduced by Bayart in Problem~5.2.

\begin{definition}[\cite{BayartJEMS}, Problem~5.2]
\label{def:product-summability}
Let $\Lambda\subset\mathbb{N}^m$ and $p\in[1,\infty)$. We say that $\Lambda$ is
\emph{$p$-product summable} if there exists $C>0$ such that
\begin{equation}\label{eq:product-summability-def}
 \sum_{(i_1,\ldots,i_m)\in\Lambda}
 \prod_{j=1}^m|x^{(j)}_{i_j}|
 \le C\prod_{j=1}^m\|x^{(j)}\|_{\ell_p}
 \qquad
 \bigl(x^{(1)},\ldots,x^{(m)}\in c_{00}\bigr).
\end{equation}
For $d\in[1,m]$, let
\[
 \mathcal C_d:=
 \{\Lambda\subset\mathbb{N}^m:
 |\Lambda|=\infty,\ \dim(\Lambda)=d\},
\]
and define
\[
 \mathrm{prod}(m,d)
 :=
 \sup\bigl\{
 p\ge1:
 \text{$\Lambda$ is $p$-product summable for some }
 \Lambda\in\mathcal C_d
 \bigr\}.
\]
\end{definition}

Two tests reveal the structure of the answer. The first is global. Applying
\eqref{eq:product-summability-def} to indicator functions of finite coordinate
sets gives
\begin{equation}\label{eq:intro-dimensional-upper}
 \mathrm{prod}(m,d)\le\frac md.
\end{equation}
This is the obstruction detected by the growth of the support inside
Cartesian boxes. When $m/d$ is an integer, equality was obtained in
\cite[Lemma~4.7 and Problem~5.2]{BayartJEMS} through uniformly incident
fractional Cartesian products.

The second obstruction is sectional rather than global. If a product
inequality holds beyond a certain exponent, suitable coordinate sections of
the support must be uniformly finite. For $d>1$, this yields
\[
 \mathrm{prod}(m,d)\le m-\lceil d\rceil+1.
\]
The mechanism is already implicit in \cite[Proposition~5.3]{BayartJEMS}; we
derive the required form directly. At $d=1$, the same numerical bound equals
$m$ and is already contained in \eqref{eq:intro-dimensional-upper}.

The two restrictions arise from different concentrations of the support. Combinatorial dimension measures its global box growth, but does not
by itself prevent a large portion of the support from lying on a
lower-dimensional coordinate section. Our first theorem shows that the optimal exponent is the lower envelope of
these two bounds.

\begin{maintheorem}\label{thm:A}
Let $m\ge2$ and $d\in[1,m]$. Then
\begin{equation}\label{eq:main-product-formula}
 \mathrm{prod}(m,d)
 =\min\left\{\frac md,\,m-\lceil d\rceil+1\right\}.
\end{equation}
Equivalently,
\begin{equation}\label{eq:main-product-piecewise}
\mathrm{prod}(m,d)=
\begin{cases}
m, & d=1,\\[1mm]
m-1, & 1<d\le\dfrac{m}{m-1},\\[3mm]
\dfrac{m}{d}, & \dfrac{m}{m-1}<d\le m-1,\\[3mm]
1, & m\ge3\text{ and }m-1<d\le m.
\end{cases}
\end{equation}
\end{maintheorem}

The formula displays three distinct behaviors. Immediately to the right of
$d=1$, the sectional obstruction creates a plateau at $m-1$. Across the
central range, global dimensional growth becomes decisive and the exponent
follows the curve $m/d$. Once $d>m-1$, even one-dimensional coordinate
sections force the terminal value $1$. The discontinuities reflect changes in the geometry that controls product
summability.

{
\subsection{The multilinear summability invariant}

Unlike the product problem, which is formulated through a single infinite
support of prescribed combinatorial dimension, Problem~5.1 is finite-dimensional
and asymptotic. The parameter $d$ plays the same geometric role: it prescribes
the polynomial size $n^d$ of the active coefficient set. Throughout this
subsection, $\ell_p^n$ is considered over $\mathbb C$, and
$\mathbb T:=\{z\in\mathbb C:|z|=1\}$.

For $n\in\mathbb N$ write $[n]=\{1,\ldots,n\}$. If
$\Lambda\subset[n]^m$ and $\varepsilon_{\mathbf i}\in\mathbb T$ for
$\mathbf i=(i_1,\ldots,i_m)\in\Lambda$, set
\begin{equation}\label{mult:eq:intro-form}
 T(x^{(1)},\ldots,x^{(m)})
 :=\sum_{\mathbf i\in\Lambda}\varepsilon_{\mathbf i}
 x^{(1)}_{i_1}\cdots x^{(m)}_{i_m}.
\end{equation}
Bayart's Problem~5.1 isolates the largest exponent for which one can place
$n^d$ unimodular coefficients in such a form while keeping its norm
subpolynomial in $n$.

\begin{definition}\label{mult:def:gamma-mult}
Let $m\ge2$ and $d\in[1,m]$. Denote by
$\Gamma_{\mathrm{mult}}(m,d)$ the set of all $p\ge1$ with the following
property: for every $\eta>0$ there exists $C_\eta\ge1$ such that, for every
$n\in\mathbb N$, one can find $\Lambda_n\subset[n]^m$ and unimodular
coefficients $(\varepsilon_{\mathbf i})_{\mathbf i\in\Lambda_n}$ satisfying
\[
 |\Lambda_n|\ge C_\eta^{-1}n^d
 \qquad\text{and}\qquad
 \|T_n\|_{\mathcal L(^m\ell_p^n)}\le C_\eta n^\eta,
\]
where $T_n$ is given by \eqref{mult:eq:intro-form}. Set
\[
 \gamma_{\mathrm{mult}}(m,d):=\sup\Gamma_{\mathrm{mult}}(m,d).
\]
\end{definition}

{\hypersetup{citecolor=red}
The partial Steiner-system constructions of
\cite{GalicerMuroSevillaPeris} yield
$\gamma_{\mathrm{mult}}(m,m-1)\ge2$; see
\cite[Section~5.1]{BayartJEMS}.}

Bayart proved the general upper estimate
\begin{equation}\label{mult:eq:Bayart-upper}
 \gamma_{\mathrm{mult}}(m,d)
 \le
 \min\left\{m-\lceil d\rceil+1,\frac{2m}{d+1}\right\}
\end{equation}
and determined the value in several regimes
\cite[Proposition~5.3 and Corollary~5.4]{BayartJEMS}. Problem~5.1 asks for the exact value of this invariant. Our second main theorem gives the full
answer.

\begin{maintheoremb}\label{mult:thm:B}
Let $m\ge2$ and $d\in[1,m]$. Then
\begin{equation}\label{mult:eq:main-formula}
 \gamma_{\mathrm{mult}}(m,d)
 =\min\left\{m-\lceil d\rceil+1,\frac{2m}{d+1}\right\}.
\end{equation}
\end{maintheoremb}

The two terms come from different configurations inside the same random
support. Global dispersion yields the restriction $p<2m/(d+1)$, while
concentration in deep coordinate fibers yields
$p<m-\lceil d\rceil+1$. A hereditary local sparsity inequality records both
phenomena; an independent random-sign argument controls every rectangular
trace, and a dyadic decomposition turns this local control into the required
operator norm bound.

Bayart also asked explicitly whether $\gamma_{\mathrm{mult}}(m,d)$ must always
be an integer. This answers Bayart's question.

\begin{corollary}\label{mult:cor:noninteger}
For $m=4$ and $d=2$,
\[
 \gamma_{\mathrm{mult}}(4,2)=\frac83.
\]
Consequently $\gamma_{\mathrm{mult}}(m,d)$ need not be an integer, and
Bayart's Question~5.7 has a negative answer.
\end{corollary}

Equivalently, the product invariant is
\[
 \mathrm{prod}(m,d)=\min\left\{\frac md,m-\lceil d\rceil+1\right\},
\]
whereas the multilinear invariant replaces the global term $m/d$ by
$2m/(d+1)$. Thus cancellation enlarges the admissible exponent exactly when
the global dimensional obstruction is the active one.

}

\subsection{Sharp Hardy--Littlewood realization}

The prescribed-support Hardy--Littlewood problem belongs to the classical Hardy--Littlewood line. It begins
with Littlewood's $4/3$ inequality \cite{Littlewood1930} and the bilinear
inequalities of Hardy and Littlewood \cite{HardyLittlewood1934}, and continues
through the multilinear theory on $\ell_p$ spaces initiated by
Praciano-Pereira \cite{PracianoPereira}; see also
\cite{AronNunezPellegrinoSerrano,DefantSevilla}. In the prescribed-support
setting, however, the support is no longer passive. Its combinatorial geometry
enters the critical exponent and becomes part of the extremal problem.

For $\mathbb K\in\{\mathbb{R},\mathbb C\}$, set
\[
 Z_p(\mathbb K):=
 \begin{cases}
  \ell_p(\mathbb K),&1\le p<\infty,\\
  c_0(\mathbb K),&p=\infty,
 \end{cases}
\]
and denote by $(e_n)_{n\ge1}$ the canonical basis in each $Z_p(\mathbb K)$.
For
\[
 \mathbf p=(p_1,\ldots,p_m)\in[1,\infty]^m,
\]
write
\[
 \left|\frac1{\mathbf p}\right|
 :=
 \sum_{j=1}^m\frac1{p_j},
 \qquad
 \frac1\infty:=0.
\]

For a prescribed support, the complex Hardy--Littlewood exponent is defined
in \cite[pp.~1162--1163]{BayartJEMS}. We write
$\mathrm{HL}_{\mathbb K}(\Lambda;\mathbf p)$ when the scalar field is
relevant; Bayart's notation corresponds to $\mathbb K=\mathbb C$.

\begin{definition}[\cite{BayartJEMS}, pp.~1162--1163]
\label{def:HL-exponent}
Let $\Lambda\subset\mathbb{N}^m$ be infinite and let
$\mathbf p=(p_1,\ldots,p_m)\in[1,\infty]^m$. For a continuous $m$-linear form
\[
 T:
 Z_{p_1}(\mathbb K)\times\cdots\times Z_{p_m}(\mathbb K)
 \longrightarrow\mathbb K
\]
and $s\in[1,\infty)$, set
\[
 \|T\|_{\Lambda,s}
 :=
 \left(
 \sum_{\mathbf i\in\Lambda}
 |T(e_{i_1},\ldots,e_{i_m})|^s
 \right)^{1/s}.
\]
We say that $s$ is \emph{admissible for $(\Lambda,\mathbf p)$} if there exists
$C>0$ such that
\begin{equation}\label{eq:HL-admissible-def}
 \|T\|_{\Lambda,s}\le C\|T\|
 \qquad
 \bigl(
 T\in
 \mathcal L(
 Z_{p_1}(\mathbb K),\ldots,Z_{p_m}(\mathbb K);\mathbb K
 )
 \bigr),
\end{equation}
where $\|T\|$ denotes the usual multilinear operator norm. The
Hardy--Littlewood exponent of $\Lambda$ at $\mathbf p$ over $\mathbb K$ is
\[
 \mathrm{HL}_{\mathbb K}(\Lambda;\mathbf p)
 :=
 \inf\bigl\{
 s\ge1:
 \text{$s$ is admissible for $(\Lambda,\mathbf p)$}
 \bigr\},
\]
with the convention $\inf\varnothing=\infty$.
\end{definition}

Bayart's theorem exhibits a sharp transition at the Hilbertian boundary
\[
 \left|\frac1{\mathbf p}\right|=\frac12.
\]
When $|1/\mathbf p|<1/2$, the prescribed dimension remains visible:
\cite[Theorem~1.1(a)]{BayartJEMS} gives
\begin{equation}\label{eq:Bayart-HL-lower-unified}
 \mathrm{HL}_{\mathbb C}(\Lambda;\mathbf p)^{-1}
 \ge
 \frac{\dim(\Lambda)+1}{2\dim(\Lambda)}
 -
 \frac1{\dim(\Lambda)}
 \left|\frac1{\mathbf p}\right|.
\end{equation}
If the combinatorial dimension is exact, the reciprocal of the right-hand
side is admissible. Bayart constructed sharp realizations in several ranges
of the dimension in \cite[Theorem~1.2]{BayartJEMS}.

For
\[
 \frac12\le\left|\frac1{\mathbf p}\right|<1,
\]
the geometry changes. The dimensional term disappears, and
\cite[Theorem~1.1(b)]{BayartJEMS} gives
\[
 \mathrm{HL}_{\mathbb C}(\Lambda;\mathbf p)^{-1}
 \ge
 1-\left|\frac1{\mathbf p}\right|,
\]
with the reciprocal exponent admissible on every infinite support. The
critical exponent is unchanged over the real field, as shown in
Proposition~\ref{prop:HL-scalar-invariance}.

The sharp estimates in the two regimes are known at the level of universal
upper bounds. We realize them on supports of every prescribed dimension, with
a single support, fixed once $m$ and $d$ are given, that is sharp for every
$\mathbf p$ with $|1/\mathbf p|<1$.

\begin{maintheoremc}\label{thm:C}
Let $m\ge2$ and $d\in[1,m]$. There exists an infinite support
$\Lambda\subset\mathbb{N}^m$ of exact combinatorial dimension $d$ such that,
for each $\mathbb K\in\{\mathbb{R},\mathbb C\}$ and every
$\mathbf p\in[1,\infty]^m$ with $|1/\mathbf p|<1$,
\begin{equation}\label{eq:main-HL-unified}
 \mathrm{HL}_{\mathbb K}(\Lambda;\mathbf p)^{-1}
 =
 \begin{cases}
 \dfrac{d+1}{2d}
 -\dfrac1d\left|\dfrac1{\mathbf p}\right|,
 & \left|\dfrac1{\mathbf p}\right|\le\dfrac12,\\[8pt]
 1-\left|\dfrac1{\mathbf p}\right|,
 & \dfrac12<\left|\dfrac1{\mathbf p}\right|<1.
 \end{cases}
\end{equation}
\end{maintheoremc}

The same support works simultaneously over the real and complex scalar fields
and for every admissible anisotropic exponent. Thus Theorem~C gives a single
support realizing the sharp Hardy--Littlewood exponents throughout the range
$|1/\mathbf p|<1$.

\medskip
\noindent\textbf{Notation.}
Throughout, $m\ge2$ and $\mathbb K\in\{\mathbb{R},\mathbb C\}$.

\begin{center}
\small
\begin{tabularx}{0.94\textwidth}{@{}lX@{}}
\hline
Symbol & Meaning\\
\hline
\(\mathbb K\) & scalar field, \(\mathbb{R}\) or \(\mathbb C\)\\
\([N]\) & \(\{1,\ldots,N\}\)\\
\(\mathbb N_0\) & \(\{0,1,2,\ldots\}\)\\
\(|E|\) & cardinality of a set \(E\), possibly infinite\\
\(\Lambda\subset\mathbb{N}^m\) & support\\
\(\Lambda(n)\) & maximal number of support points in a Cartesian box with at most \(n\) points in each factor\\
\(\dim(\Lambda)\) & combinatorial dimension of \(\Lambda\)\\
\(\mathrm{prod}(m,d)\) & extremal product-summability quantity in dimension \(d\)\\
\(\Gamma_{\mathrm{mult}}(m,d)\) & admissible exponents in Bayart's multilinear summability problem\\
\(\gamma_{\mathrm{mult}}(m,d)\) & multilinear summability invariant in dimension \(d\)\\
\(\mathcal U=(S_1,\ldots,S_m)\) & incidence family defining a fractional Cartesian model\\
\(I_{\mathcal U}\) & incidence matrix of \(\mathcal U\)\\
\(\pi_S\) & coordinate projection onto the coordinates indexed by \(S\)\\
\(\Pi_{\mathcal U}\), \(\mathbb{N}^{\mathcal U}\) & projection map and associated support\\
\(\rho^*(\mathcal U)\) & fractional edge-cover number\\
\(J(A_1,\ldots,A_m)\) & join \(\{u\in\mathbb{N}^\ell:\pi_{S_j}(u)\in A_j\text{ for all }j\}\)\\
\(E_{j,r}\), \(\Delta(E)\) & one-coordinate fiber and maximal fiber cardinality\\
\(\mathbf p=(p_1,\ldots,p_m)\) & anisotropic exponent vector\\
\(\lvert1/\mathbf p\rvert\) & \(\sum_j1/p_j\), with \(1/\infty=0\)\\
\(Z_p(\mathbb K)\) & \(\ell_p(\mathbb K)\) for \(p<\infty\), and \(c_0(\mathbb K)\) for \(p=\infty\)\\
\(\mathrm{HL}_{\mathbb K}(\Lambda;\mathbf p)\) & Hardy--Littlewood exponent on \(\Lambda\)\\
\hline
\end{tabularx}
\end{center}

\section{Product summability at prescribed dimension}\label{sec:problem52-prep}

The two endpoint dimensions are realized by the full Cartesian product and by the diagonal.

\begin{example}\label{ex:endpoint-dimensions}
Let
\[
 \Lambda_1:=\mathbb{N}^m,
 \qquad
 \Lambda_2:=\{(r,\ldots,r):r\in\mathbb{N}\}.
\]
If $A_1,\ldots,A_m\subset\mathbb{N}$ and $|A_j|\le n$, then
\[
 |(A_1\times\cdots\times A_m)\cap\Lambda_1|
 =\prod_{j=1}^m|A_j|\le n^m,
\]
with equality when $A_1=\cdots=A_m=[n]$. Hence
$\Lambda_1(n)=n^m$ and $\dim(\Lambda_1)=m$. On the other hand,
\[
 (A_1\times\cdots\times A_m)\cap\Lambda_2
 =\{(r,\ldots,r):r\in A_1\cap\cdots\cap A_m\},
\]
so this intersection has at most $n$ elements. Again equality is attained
for $A_1=\cdots=A_m=[n]$. Thus $\Lambda_2(n)=n$ and
$\dim(\Lambda_2)=1$.
\end{example}

For $m\ge3$, the four regimes in Theorem~\ref{thm:A} can be visualized schematically as follows.

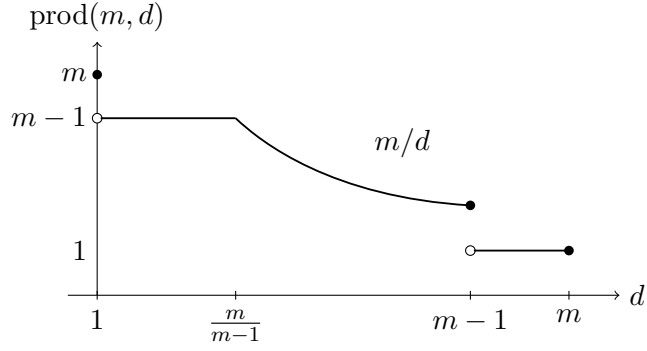
\begin{center}
\begin{tikzpicture}[x=1.55cm,y=0.72cm]
  \draw[->] (0.75,0) -- (5.45,0) node[right] {$d$};
  \draw[->] (1,0) -- (1,4.65) node[above] {$\mathrm{prod}(m,d)$};

  \fill[] (1,4.05) circle (1.8pt);
  \draw[fill=white] (1,3.25) circle (1.8pt);
  \draw[line width=0.7pt] (1.03,3.25) -- (2.18,3.25);

  \draw[line width=0.7pt]
    (2.18,3.25) .. controls (2.4,2.8) and (3,1.8) .. (4.18,1.65);
  \fill[] (4.18,1.65) circle (1.8pt);

  \draw[fill=white] (4.18,0.82) circle (1.8pt);
  \draw[line width=0.7pt] (4.21,0.82) -- (5.02,0.82);
  \fill[] (5.02,0.82) circle (1.8pt);

  \draw[] (1,0.08)--(1,-0.08) node[below] {$1$};
  \draw[] (2.18,0.08)--(2.18,-0.08) node[below] {$\frac{m}{m-1}$};
  \draw[] (4.18,0.08)--(4.18,-0.08) node[below] {$m-1$};
  \draw[] (5.02,0.08)--(5.02,-0.08) node[below] {$m$};

  \node[left] at (1,4.05) {$m$};
  \node[left] at (1,3.25) {$m-1$};
  \node[left] at (1,0.82) {$1$};
  \node[] at (3.6,2.8) {$m/d$};
\end{tikzpicture}
\end{center}

Fix \(\ell\ge1\). Let
\[
 \mathcal U=(S_1,\ldots,S_m),
 \qquad
 \varnothing\ne S_j\subset[\ell],
 \qquad
 \bigcup_{j=1}^mS_j=[\ell].
\]
For \(S\subset[\ell]\), define
\[
 \pi_S:\mathbb{N}^\ell\longrightarrow\mathbb{N}^S,
 \qquad
 \pi_S((u_t)_{t=1}^{\ell})=(u_t)_{t\in S},
\]
and
\[
 \Pi_{\mathcal U}(u)
 =
 \bigl(\pi_{S_1}(u),\ldots,\pi_{S_m}(u)\bigr).
\]
The associated fractional Cartesian support is
\[
 \mathbb{N}^{\mathcal U}:=\Pi_{\mathcal U}(\mathbb{N}^\ell).
\]
After fixed identifications \(\mathbb{N}^{S_j}\simeq\mathbb{N}\), it is regarded
as a subset of \(\mathbb{N}^m\).

The incidence matrix of \(\mathcal U\) is
\[
 I_{\mathcal U}=[\iota_{t,j}]_{\ell\times m},
 \qquad
 \iota_{t,j}=
 \begin{cases}
 1,&t\in S_j,\\
 0,&t\notin S_j.
 \end{cases}
\]
A \emph{fractional edge cover} is a vector
\(w=(w_1,\ldots,w_m)\in[0,\infty)^m\) satisfying
\[
 \sum_{j:\,t\in S_j}w_j\ge1
 \qquad(t\in[\ell]).
\]
Its fractional edge-cover number is
\[
 \rho^*(\mathcal U)
 =
 \min\left\{
 \sum_{j=1}^m w_j:
 w\text{ is a fractional edge cover of }\mathcal U
 \right\}.
\]

\subsection{Fractional Cartesian models and combinatorial dimension}

{For fractional Cartesian products, combinatorial dimension is given by the associated linear program. This is the Blei--Schmerl formula \cite[Theorem~1]{BleiSchmerl}; see also \cite{Blei1979,Blei1984,BleiMemoir,BleiBook,BleiGao}.}\par
The linear-programming dual of the fractional edge-cover problem is
\begin{equation}\label{eq:rhoU-dual}
 \max\left\{\sum_{t=1}^{\ell}y_t:
 y\in[0,\infty)^\ell,\ 
 \sum_{t\in S_j}y_t\le1\quad(1\le j\le m)\right\}.
\end{equation}
{By \cite[Theorem~1]{BleiSchmerl}, in the linear-programming form \eqref{eq:rhoU-dual},} the combinatorial dimension of
$\mathbb{N}^{\mathcal U}$ is the value of the dual program
\eqref{eq:rhoU-dual}. Hence, by linear programming duality,
\[
 \dim(\mathbb{N}^{\mathcal U})=\rho^*(\mathcal U).
\]

{The uniform-cover inequality below is a finite discrete form of Finner's generalized H\"older inequality \cite{Finner1992}; compare Bollob\'as--Thomason \cite{BollobasThomason}.}\par
\begin{lemma}
\label{lem:uniform-cover-holder}
Let $\ell\in\mathbb{N}$ and let $T_1,\ldots,T_R\subset[\ell]$ satisfy
\begin{equation}\label{eq:uniform-q-cover}
 |\{r:t\in T_r\}|=q
 \qquad(t\in[\ell]),
\end{equation}
where $1\le q\le R$. For each $1\le r\le R$, let
$f_r:\mathbb{N}^{T_r}\to[0,\infty)$ be finitely supported. Then
\begin{equation}\label{eq:uniform-cover-holder}
 \sum_{u\in\mathbb{N}^\ell}
 \prod_{r=1}^R f_r(\pi_{T_r}(u))^{1/q}
 \le
 \prod_{r=1}^R
 \left(\sum_{v\in\mathbb{N}^{T_r}}f_r(v)\right)^{1/q}.
\end{equation}
\end{lemma}

\begin{proof}
We argue by induction on $\ell$. For $\ell=1$, each $T_r$ is either
$\varnothing$ or $\{1\}$. Set
\[
 I_1:=\{r:1\in T_r\}.
\]
By \eqref{eq:uniform-q-cover}, $|I_1|=q$. If $r\notin I_1$, then
$f_r(\pi_{T_r}(u))=f_r(\varnothing)$ for every $u\in\mathbb{N}$. If
$r\in I_1$, then, under the natural identification
$\mathbb{N}^{\{1\}}\cong\mathbb{N}$, we have
$f_r(\pi_{T_r}(u))=f_r(u)$. H\"older's inequality therefore gives
\begin{align*}
 \sum_{u\in\mathbb{N}}\prod_{r=1}^R f_r(\pi_{T_r}(u))^{1/q}
 &=\left(\prod_{r\notin I_1}f_r(\varnothing)^{1/q}\right)
   \sum_{u\in\mathbb{N}}\prod_{r\in I_1}f_r(u)^{1/q}\\
 &\le
 \left(\prod_{r\notin I_1}f_r(\varnothing)^{1/q}\right)
 \prod_{r\in I_1}
 \left(\sum_{u\in\mathbb{N}}f_r(u)\right)^{1/q}\\
 &=\prod_{r=1}^R
 \left(\sum_{v\in\mathbb{N}^{T_r}}f_r(v)\right)^{1/q}.
\end{align*}

Assume now that the result holds for $\ell-1$ coordinates, and set
\[
 I_\ell:=\{r:\ell\in T_r\}.
\]
Again \eqref{eq:uniform-q-cover} gives $|I_\ell|=q$. Fix
$u'=(u_1,\ldots,u_{\ell-1})$. If $r\notin I_\ell$, then
$T_r\subset[\ell-1]$ and
\[
 f_r\bigl(\pi_{T_r}(u',u_\ell)\bigr)=f_r(\pi_{T_r}(u')).
\]
Applying H\"older's inequality to the $q$ factors indexed by $I_\ell$, we
obtain
\begin{align}
 &\sum_{u_\ell\in\mathbb{N}}
 \prod_{r=1}^R f_r\bigl(\pi_{T_r}(u',u_\ell)\bigr)^{1/q}
 \notag\\
 &\quad\le
 \prod_{r\notin I_\ell}f_r(\pi_{T_r}(u'))^{1/q}
 \prod_{r\in I_\ell}
 \left(\sum_{u_\ell\in\mathbb{N}}
 f_r\bigl(\pi_{T_r}(u',u_\ell)\bigr)\right)^{1/q}.
 \label{eq:last-coordinate-holder}
\end{align}
For each $r$, put $T'_r:=T_r\setminus\{\ell\}$ and define
$g_r:\mathbb{N}^{T'_r}\to[0,\infty)$ by
\begin{equation}\label{eq:def-gr}
 g_r(w):=
 \begin{cases}
  f_r(w),&\ell\notin T_r,\\[1mm]
  \displaystyle\sum_{s\in\mathbb{N}}f_r(w,s),&\ell\in T_r.
 \end{cases}
\end{equation}
Since $\ell$ is the largest coordinate, the notation $(w,s)$ respects the
ordering in $\mathbb{N}^{T_r}$. Each $g_r$ is finitely supported, and the
right-hand side of \eqref{eq:last-coordinate-holder} is
\[
 \prod_{r=1}^R g_r(\pi_{T'_r}(u'))^{1/q}.
\]
Summing over $u'\in\mathbb{N}^{\ell-1}$ gives
\[
 \sum_{u\in\mathbb{N}^\ell}\prod_{r=1}^R f_r(\pi_{T_r}(u))^{1/q}
 \le
 \sum_{u'\in\mathbb{N}^{\ell-1}}
 \prod_{r=1}^R g_r(\pi_{T'_r}(u'))^{1/q}.
\]
The family $T'_1,\ldots,T'_R$ covers each coordinate of $[\ell-1]$ exactly
$q$ times, since $t\in T'_r$ if and only if $t\in T_r$ for $t<\ell$. The
induction hypothesis therefore yields
\begin{equation}\label{eq:induction-gr}
 \sum_{u'\in\mathbb{N}^{\ell-1}}
 \prod_{r=1}^R g_r(\pi_{T'_r}(u'))^{1/q}
 \le
 \prod_{r=1}^R
 \left(\sum_{w\in\mathbb{N}^{T'_r}}g_r(w)\right)^{1/q}.
\end{equation}
Finally, for every $r$,
\begin{equation}\label{eq:mass-preservation}
 \sum_{w\in\mathbb{N}^{T'_r}}g_r(w)
 =\sum_{v\in\mathbb{N}^{T_r}}f_r(v).
\end{equation}
If $\ell\notin T_r$, this is immediate. If $\ell\in T_r$, it follows by
summing the second line of \eqref{eq:def-gr} over $w$. Combining
\eqref{eq:induction-gr} and \eqref{eq:mass-preservation} proves
\eqref{eq:uniform-cover-holder}.
\end{proof}

{The corresponding fractional-cover estimate is the AGM bound for joins \cite{AGM}; it also follows from Finner's generalized H\"older inequality \cite{Finner1992}. In the present notation it is an immediate consequence of Lemma~\ref{lem:uniform-cover-holder}.}\par
\begin{lemma}
\label{lem:fractional-edge-cover-finite-joins}
Let $\mathcal U=(S_1,\ldots,S_m)$ be a family of subsets of $[\ell]$ whose union is $[\ell]$. For each $1\le j\le m$, let
$A_j\subset\mathbb{N}^{S_j}$ be finite, and define
\[
 J(A_1,\ldots,A_m):=
 \{u\in\mathbb{N}^\ell:\pi_{S_j}(u)\in A_j\text{ for }1\le j\le m\}.
\]
If $w=(w_1,\ldots,w_m)\in[0,\infty)^m$ is a fractional edge cover of
$\mathcal U$, then
\begin{equation}\label{eq:fractional-edge-cover-join}
 |J(A_1,\ldots,A_m)|\le\prod_{j=1}^m|A_j|^{w_j}.
\end{equation}
Consequently,
\begin{equation}\label{eq:fractional-edge-cover-support}
 |(A_1\times\cdots\times A_m)\cap\mathbb{N}^{\mathcal U}|
 \le\prod_{j=1}^m|A_j|^{w_j}.
\end{equation}
In particular, if $|A_j|\le n$ for every $j$, then
\[
 |(A_1\times\cdots\times A_m)\cap\mathbb{N}^{\mathcal U}|
 \le n^{\sum_{j=1}^m w_j}.
\]
{Throughout, $0^0:=1$.}
\end{lemma}

\begin{proof}
Let
$T_1,\ldots,T_R\subset[\ell]$ cover every $t\in[\ell]$ exactly $q$ times,
and let $E\subset\mathbb{N}^\ell$ be finite. Taking
$f_r:=\mathbf1_{\pi_{T_r}(E)}$, we have
\[
 \prod_{r=1}^R
 \mathbf1_{\pi_{T_r}(E)}(\pi_{T_r}(u))=1
 \qquad(u\in E).
\]
Since this product takes only the values $0$ and $1$, we apply
Lemma~\ref{lem:uniform-cover-holder} with $f_r=\mathbf1_{\pi_{T_r}(E)}$. Its H\"older-type inequality gives
\begin{align*}
 |E|
 &\le
 \sum_{u\in\mathbb{N}^\ell}
 \prod_{r=1}^R
 \mathbf1_{\pi_{T_r}(E)}(\pi_{T_r}(u))^{1/q}\\
 &\le
 \prod_{r=1}^R
 \left(\sum_{v\in\mathbb{N}^{T_r}}
 \mathbf1_{\pi_{T_r}(E)}(v)\right)^{1/q}
 =\prod_{r=1}^R|\pi_{T_r}(E)|^{1/q}.
\end{align*}
Thus
\begin{equation}\label{eq:uniform-cover-cardinality}
 |E|^q\le\prod_{r=1}^R|\pi_{T_r}(E)|.
\end{equation}

If $J(A_1,\ldots,A_m)$ is empty, then
\eqref{eq:fractional-edge-cover-join} is immediate. We may therefore assume
that this join is nonempty. In particular, every $A_j$ is nonempty. Since the
sets $S_j$ cover $[\ell]$ and every $A_j$ is finite, the join is finite.

Suppose first that the weights $w_j$ are rational. Choose a common
denominator and write
\[
 w_j=\frac{a_j}{Q},
 \qquad
 a_j\in\mathbb{N}\cup\{0\},\quad Q\in\mathbb{N}.
\]
The fractional-cover condition becomes
\begin{equation}\label{eq:rational-cover-multiplicity}
 \sum_{j:\,t\in S_j}a_j\ge Q
 \qquad(t\in[\ell]).
\end{equation}
Repeat each edge $S_j$ exactly $a_j$ times, denoting its copies by
$S_{j,1},\ldots,S_{j,a_j}$. For each $t\in[\ell]$, condition
\eqref{eq:rational-cover-multiplicity} allows us to choose exactly $Q$ of the
copies containing $t$. For every pair $(j,s)$, let
$T_{j,s}\subset S_j$ be the set of vertices that selected the copy
$S_{j,s}$. By construction,
\[
 |\{(j,s):t\in T_{j,s}\}|=Q
 \qquad(t\in[\ell]).
\]
Thus $(T_{j,s})$ is a uniform cover of multiplicity $Q$. Applying
\eqref{eq:uniform-cover-cardinality} to $E=J(A_1,\ldots,A_m)$, we obtain
\[
 |J(A_1,\ldots,A_m)|^Q
 \le
 \prod_{j=1}^m\prod_{s=1}^{a_j}
 |\pi_{T_{j,s}}(J(A_1,\ldots,A_m))|.
\]
Since $T_{j,s}\subset S_j$ and
$\pi_{S_j}(J(A_1,\ldots,A_m))\subset A_j$, we have
\[
 |\pi_{T_{j,s}}(J(A_1,\ldots,A_m))|\le |A_j|.
\]
Consequently,
\[
 |J(A_1,\ldots,A_m)|^Q\le\prod_{j=1}^m|A_j|^{a_j},
\]
and taking the $Q$th root proves
\eqref{eq:fractional-edge-cover-join} for rational weights.

Let the weights now be nonnegative real numbers. To avoid confusing this approximation index with the box size used
elsewhere, let $M\in\mathbb{N}$ and set
\[
 w_j^{(M)}:=\frac{\lceil Mw_j\rceil}{M}.
\]
Then $w_j^{(M)}\ge w_j$, $w_j^{(M)}\to w_j$, and
\[
 \sum_{j:\,t\in S_j}w_j^{(M)}
 \ge\sum_{j:\,t\in S_j}w_j\ge1.
\]
Thus $w^{(M)}$ is a rational fractional cover. By the case already proved,
\[
 |J(A_1,\ldots,A_m)|
 \le\prod_{j=1}^m|A_j|^{w_j^{(M)}}.
\]
Letting $M\to\infty$ proves \eqref{eq:fractional-edge-cover-join}.
Finally, the injective map $\Pi_{\mathcal U}$ gives a bijection between
\[
 J(A_1,\ldots,A_m) \qquad \text{and} \qquad (A_1\times\cdots\times A_m)\cap\mathbb{N}^{\mathcal U},
\]
which proves \eqref{eq:fractional-edge-cover-support}. The final assertion
follows from $|A_j|\le n$.
\end{proof}

{The Blei--Schmerl formula gives $\dim(\mathbb{N}^{\mathcal U})=\rho^*(\mathcal U)$ \cite[Theorem~1]{BleiSchmerl}. Its finite-box form is the following.}
\begin{theorem}[Blei--Schmerl dimension formula]
\label{thm:bayart-general-dim}
Let $\mathcal U=(S_1,\ldots,S_m)$ be a family of subsets of $[\ell]$ whose union is $[\ell]$. Then
\[
 \dim(\mathbb{N}^{\mathcal U})=\rho^*(\mathcal U).
\]
More precisely, there exists $C\ge1$ such that
\[
 C^{-1}n^{\rho^*(\mathcal U)}
 \le \mathbb{N}^{\mathcal U}(n)
 \le Cn^{\rho^*(\mathcal U)}
 \qquad(n\ge2).
\]
\end{theorem}

\begin{proof}
Fix sets $A_j\subset\mathbb{N}^{S_j}$ with $|A_j|\le n$. For every fractional
edge cover $w=(w_1,\ldots,w_m)$ of $\mathcal U$,
Lemma~\ref{lem:fractional-edge-cover-finite-joins} gives
\[
 |(A_1\times\cdots\times A_m)\cap\mathbb{N}^{\mathcal U}|
 \le\prod_{j=1}^m|A_j|^{w_j}
 \le n^{\sum_jw_j}.
\]
Minimizing over $w$, we obtain
\[
 \mathbb{N}^{\mathcal U}(n)\le n^{\rho^*(\mathcal U)}.
\]

For the reverse estimate, let $y=(y_t)_{t=1}^\ell$ be an optimal solution of
\eqref{eq:rhoU-dual}. For $n\ge2$, choose finite sets
$B_t\subset\mathbb{N}$ with
\[
 |B_t|=\lfloor n^{y_t}\rfloor.
\]
Then $|B_t|\ge n^{y_t}/2$. Define
\[
 A_j:=\prod_{t\in S_j}B_t\subset\mathbb{N}^{S_j}.
\]
By dual feasibility,
\[
 |A_j|=\prod_{t\in S_j}|B_t|
 \le n^{\sum_{t\in S_j}y_t}\le n.
\]
Let $B:=\prod_{t=1}^\ell B_t\subset\mathbb{N}^\ell$. For every $u\in B$,
$\pi_{S_j}(u)\in A_j$, and therefore
\[
 \Pi_{\mathcal U}(B)
 \subset
 (A_1\times\cdots\times A_m)\cap\mathbb{N}^{\mathcal U}.
\]
The injectivity of $\Pi_{\mathcal U}$ now gives
\begin{align*}
 |(A_1\times\cdots\times A_m)\cap\mathbb{N}^{\mathcal U}|
 &\ge |\Pi_{\mathcal U}(B)|
 =|B|
 =\prod_{t=1}^\ell|B_t|\\
 &\ge 2^{-\ell}n^{\sum_ty_t}
 =2^{-\ell}n^{\rho^*(\mathcal U)}.
\end{align*}
Thus one may take $C=2^\ell$, and the two estimates prove the result.
\end{proof}

For a uniformly incident family, the formula takes the form used in \cite{BayartJEMS}.
Suppose that every $S_j$ has cardinality $k$ and that each $t\in[\ell]$ belongs to exactly $h$ of the sets $S_j$. Counting incidences gives
\begin{equation}\label{eq:incidence-count}
 mk=\ell h.
\end{equation}
The constant primal choice $w_j=1/h$ and constant dual choice $y_t=1/k$ are
feasible. Hence
\begin{equation}\label{eq:uniform-rho}
 \rho^*(\mathcal U)=\frac mh=\frac\ell k.
\end{equation}
This is the dimension formula in \cite[Section~4.3]{BayartJEMS}; in the
uniformly incident case, it also follows from \cite[Chapter~XIII, Theorem~14
and Corollary~16]{BleiBook}.

\subsection{Product estimates on fractional Cartesian supports}

The incidence matrix also determines the exact analytic region for the
associated product inequality. For $1\le j\le m$, let
\[
 a^{(j)}=(a^{(j)}_\nu)_{\nu\in\mathbb{N}^{S_j}}
\]
be a finitely supported nonnegative family. Given
$\mathbf p=(p_1,\ldots,p_m)\in[1,\infty]^m$, with $1/\infty:=0$,
the fractional-cover condition also characterizes the product inequality.
This is a discrete coordinate-projection form of Finner's generalized
H\"older inequality \cite{Finner1992}; compare
\cite[Chapter~XIII]{BleiBook} and \cite[Section~4.3]{BayartJEMS}.

\begin{proposition}
\label{prop:fractional-cartesian-product-region}
The estimate
\begin{equation}\label{eq:finner-product}
 \sum_{u\in\mathbb{N}^\ell}
 \prod_{j=1}^m a^{(j)}_{\pi_{S_j}(u)}
 \le
 \prod_{j=1}^m
 \|a^{(j)}\|_{\ell_{p_j}(\mathbb{N}^{S_j})}
\end{equation}
holds for all finitely supported nonnegative families $a^{(j)}$ if and only if
\begin{equation}\label{eq:product-region}
 \sum_{j:\,t\in S_j}\frac1{p_j}\ge1
 \qquad(t\in[\ell]).
\end{equation}
Equivalently, $(1/p_1,\ldots,1/p_m)$ is a fractional edge cover of
$\mathcal U$.
\end{proposition}

\begin{proof}
For sufficiency, put $\theta_j:=1/p_j$. If $\theta_j=0$, the corresponding
factor can be bounded pointwise by $\|a^{(j)}\|_\infty$ and removed from the
product. Since such a weight does not contribute to
\eqref{eq:product-region}, we may assume that $\theta_j>0$ for every $j$.

We prove by induction on $\ell$ the slightly more general assertion in which
empty sets $S_j$ are allowed. For $\ell=0$, the claim is immediate because
$\mathbb{N}^\varnothing$ has one element. Suppose $\ell\ge1$ and that the
result holds for $\ell-1$ coordinates. Set
\[
 I_\ell:=\{j:\ell\in S_j\}.
\]
Condition \eqref{eq:product-region} gives
$\sum_{j\in I_\ell}\theta_j\ge1$. Choose numbers
$0\le\lambda_j\le\theta_j$, $j\in I_\ell$, such that
$\sum_{j\in I_\ell}\lambda_j=1$. For fixed
$u'\in\mathbb{N}^{\ell-1}$, H\"older's inequality and monotonicity of sequence
norms yield
\begin{align*}
 &\sum_{u_\ell\in\mathbb{N}}
 \prod_{j\in I_\ell}a^{(j)}_{\pi_{S_j}(u',u_\ell)}\\
 &\quad\le
 \prod_{j\in I_\ell}
 \left\|
 \bigl(a^{(j)}_{\pi_{S_j}(u',s)}\bigr)_{s\in\mathbb{N}}
 \right\|_{\ell_{1/\lambda_j}}\\
 &\quad\le
 \prod_{j\in I_\ell}
 \left(\sum_{s\in\mathbb{N}}
 \bigl(a^{(j)}_{\pi_{S_j}(u',s)}\bigr)^{1/\theta_j}
 \right)^{\theta_j},
\end{align*}
where the factor corresponding to $\lambda_j=0$ is interpreted as a
supremum norm.

Put $S'_j:=S_j\setminus\{\ell\}$ and define
$g_j:\mathbb{N}^{S'_j}\to[0,\infty)$ by
\[
 g_j(w):=
 \begin{cases}
  a_w^{(j)},&j\notin I_\ell,\\[1mm]
  \displaystyle
  \left(\sum_{s\in\mathbb{N}}
  \bigl(a_{(w,s)}^{(j)}\bigr)^{1/\theta_j}\right)^{\theta_j},
  &j\in I_\ell.
 \end{cases}
\]
Summing first in $u_\ell$ and applying the one-variable H\"older estimate displayed immediately before the definition of $g_j$ gives
\begin{align*}
 \sum_{u\in\mathbb{N}^\ell}
 \prod_{j=1}^m a^{(j)}_{\pi_{S_j}(u)}
 &\le
 \sum_{u'\in\mathbb{N}^{\ell-1}}
 \left(\prod_{j\notin I_\ell}a^{(j)}_{\pi_{S_j}(u')}\right)
 \prod_{j\in I_\ell}
 \left(\sum_{s\in\mathbb{N}}
 \bigl(a^{(j)}_{\pi_{S_j}(u',s)}\bigr)^{1/\theta_j}
 \right)^{\theta_j}\\
 &=\sum_{u'\in\mathbb{N}^{\ell-1}}
 \prod_{j=1}^m g_j(\pi_{S'_j}(u')).
\end{align*}
For each $t<\ell$, the family $(S'_j)_{j=1}^m$ still satisfies
$\sum_{j:\,t\in S'_j}\theta_j\ge1$. Applying the induction hypothesis and
using
\[
 \|g_j\|_{\ell_{1/\theta_j}}
 =\|a^{(j)}\|_{\ell_{1/\theta_j}}
 \qquad(1\le j\le m)
\]
proves \eqref{eq:finner-product}.

For necessity, fix $t\in[\ell]$. Let $u_t$ range over
$[N]$ and fix all other coordinates at $1$. If $t\in S_j$, take
$a^{(j)}$ to be the indicator function of the resulting $N$ projected
points; if $t\notin S_j$, take the indicator function of the single fixed
point. The left-hand side of \eqref{eq:finner-product} equals $N$, whereas
the right-hand side equals
\[
 N^{\sum_{j:\,t\in S_j}1/p_j}.
\]
Letting $N\to\infty$ yields \eqref{eq:product-region}.
\end{proof}

For the isotropic specialization $p_1=\cdots=p_m=p$, define the incidence degree
\[
 {
\deg_{\mathcal U}:[\ell]\longrightarrow\{1,\ldots,m\},\qquad
\deg_{\mathcal U}(t):=|\{j:t\in S_j\}|.
}
\]
Proposition~\ref{prop:fractional-cartesian-product-region} shows that the
largest admissible isotropic exponent is
\begin{equation}\label{eq:pmax-U}
 p_{\max}(\mathcal U)=\min_{1\le t\le\ell}\deg_{\mathcal U}(t).
\end{equation}
In the uniformly $h$-incident case, this reduces to
$p_{\max}(\mathcal U)=h$, which is the product estimate used in \cite[Lemma~4.7]{BayartJEMS}.

\subsection{The two upper bounds}\label{sec:upper-bounds}

\subsubsection*{The dimensional obstruction}

The product estimate has two elementary consequences. The first depends only
on combinatorial dimension and follows by testing the product inequality on characteristic functions.

\begin{proposition}\label{prop:gamma-prod-upper}
Let $m\ge2$, let $d\in[1,m]$, and let $\Lambda\subset\mathbb{N}^m$ satisfy
$\dim(\Lambda)=d$. If, for some $p\ge1$ and $C>0$,
\begin{equation}\label{eq:bayart-product-general}
 \sum_{(i_1,\ldots,i_m)\in\Lambda}
 \prod_{j=1}^m |x^{(j)}_{i_j}|
 \le C\prod_{j=1}^m\|x^{(j)}\|_{\ell_p}
\end{equation}
for all $x^{(1)},\ldots,x^{(m)}\in c_{00}$, then $p\le m/d$. Consequently,
\[
\mathrm{prod}(m,d)\le \frac md.
\]
\end{proposition}

\begin{proof}
Fix finite sets $A_j\subset\mathbb{N}$ with $|A_j|\le n$ and set
\[
 x^{(j)}:=n^{-1/p}\mathbf 1_{A_j},\qquad 1\le j\le m.
\]
For every $j$,
\[
 \|x^{(j)}\|_{\ell_p}^p
 =\sum_{r\in A_j}n^{-1}
 =\frac{|A_j|}{n}
 \le1,
\]
so $\|x^{(j)}\|_{\ell_p}\le1$. Moreover, for
$\mathbf i=(i_1,\ldots,i_m)\in\Lambda$,
\[
 \prod_{j=1}^m|x^{(j)}_{i_j}|
 =n^{-m/p}\prod_{j=1}^m\mathbf 1_{A_j}(i_j).
\]
Hence the left-hand side of \eqref{eq:bayart-product-general} is exactly
\[
 n^{-m/p}|\Lambda\cap(A_1\times\cdots\times A_m)|.
\]
Applying \eqref{eq:bayart-product-general} therefore gives
\[
 n^{-m/p}|\Lambda\cap(A_1\times\cdots\times A_m)|\le C,
\]
and thus
\[
 |\Lambda\cap(A_1\times\cdots\times A_m)|\le Cn^{m/p}.
\]
Taking the maximum over all $A_1,\ldots,A_m$ with $|A_j|\le n$ yields
\[
 \Lambda(n)\le Cn^{m/p}.
\]
Consequently,
\[
 \frac{\log\Lambda(n)}{\log n}
 \le \frac{\log C}{\log n}+\frac mp.
\]
Passing to the limsup as $n\to\infty$ and using
$\log C/\log n\to0$ gives
\[
 d=\dim(\Lambda)\le\frac mp.
\]
\end{proof}

\subsubsection*{An obstruction from coordinate sections}

A second obstruction comes from coordinate sections. It is already implicit in \cite[Proposition~5.3]{BayartJEMS}; the direct form below makes explicit the section estimate that becomes decisive when the prescribed dimension is close to $m$.

\begin{proposition}\label{prop:general-section-obstruction}
Let $1\le r\le m-1$. If a support $\Lambda\subset\mathbb{N}^m$ satisfies
\eqref{eq:bayart-product-general} for some $p>r$, then
$\dim(\Lambda)\le m-r$. Hence, for every $d>1$,
\[
\mathrm{prod}(m,d)\le m-\lceil d\rceil+1.
\]
\end{proposition}

\begin{proof}
Fix $m-r$ coordinates. Consider a finite subset of the resulting
$r$-dimensional section, with $R$ points, and let $A_1,\ldots,A_r$ be its
coordinate projections. Testing \eqref{eq:bayart-product-general} with indicator functions in
the $r$ free coordinates and point masses in the $m-r$ fixed coordinates
gives
\[
 R\le C\prod_{j=1}^r |A_j|^{1/p}.
\]
Each coordinate projection of a set of $R$ points contains at most $R$
points, so $|A_j|\le R$ for $1\le j\le r$. Therefore
\[
 R\le C R^{r/p}.
\]
If $R>0$, division by $R^{r/p}$ gives
\[
 R^{1-r/p}\le C.
\]
Since $p>r$,
\[
 1-\frac rp=\frac{p-r}{p}>0,
\]
and raising both sides to the power $p/(p-r)$ yields
\[
 R\le C^{p/(p-r)}=:K.
\]
Thus every such section contains at most $K$ points. A box $A_1\times\cdots\times A_m$ with
$|A_j|\le n$ is partitioned into at most $n^{m-r}$ of these sections, so
\[
|\Lambda\cap(A_1\times\cdots\times A_m)|\le Kn^{m-r}.
\]
Thus $\dim(\Lambda)\le m-r$. For a support of dimension $d>1$, take
$r=m-\lceil d\rceil+1$. If an exponent $p>r$ were admissible, then the
preceding estimate would give
\[
d=\dim(\Lambda)\le m-r=\lceil d\rceil-1<d,
\]
a contradiction. Hence $p\le r=m-\lceil d\rceil+1$, which gives the stated bound.
\end{proof}

\subsection{Model supports}\label{sec:model-supports}

\subsubsection*{Vandermonde supports}

{For integer dimensions, the discrete H\"older--Brascamp--Lieb criterion of \cite[Theorem~1.4]{CDKTY} applies. In the torsion-free case its optimal constant is $1$.} See also \cite{BCCT,BrascampLieb1976}.

\begin{proposition}\label{prop:discrete-HBL}
Let $G,G_1,\ldots,G_m$ be finitely generated torsion-free abelian groups,
let $\phi_j:G\to G_j$ be homomorphisms, and let $s_j\in[0,1]$. The
inequality
\begin{equation}\label{eq:HBL-general}
 \sum_{u\in G}\prod_{j=1}^m f_j(\phi_j(u))
 \le \prod_{j=1}^m\|f_j\|_{\ell_{1/s_j}(G_j)},
 \qquad \frac1{0}:=\infty,
\end{equation}
holds for every family of nonnegative finitely supported functions
$f_j:G_j\to[0,\infty)$ if and only if
\begin{equation}\label{eq:HBL-rank}
 \operatorname{rank}H
 \le\sum_{j=1}^m s_j\operatorname{rank}\phi_j(H)
\end{equation}
for every subgroup $H\le G$. Under these conditions the optimal constant in
\eqref{eq:HBL-general} is $1$.
\end{proposition}

{Fix $m\ge2$ and an integer}
$d\in\{1,\ldots,m\}$. Choose distinct integers $a_{1},\ldots,a_{m}$ and define
homomorphisms
$\phi_{j}:\mathbb{Z}^{d}\to\mathbb{Z}$ by
\begin{equation}
\label{eq:Vandermonde-forms}\phi_{j}(u_{1},\ldots,u_{d}) :=u_{1}+a_{j}
u_{2}+\cdots+a_{j}^{d-1}u_{d}.
\end{equation}
The corresponding row vectors
\[
v_{j}=(1,a_{j},\ldots,a_{j}^{d-1})\in\mathbb{Q}^{d}
\]
have the usual Vandermonde property: every subfamily of at most $d$ vectors is
linearly independent.

{The Vandermonde homomorphisms satisfy the estimate below.}\par
\begin{lemma}\label{lem:Vandermonde-HBL}
For the homomorphisms \eqref{eq:Vandermonde-forms},
\begin{equation}
\label{eq:HBL-isotropic-vandermonde}\sum_{u\in\mathbb{Z}^{d}}\prod_{j=1}^{m}
f_{j}(\phi_{j}(u)) \le\prod_{j=1}^{m}\|f_{j}\|_{\ell_{m/d}(\mathbb{Z})}%
\end{equation}
for all nonnegative finitely supported functions
$f_j:\mathbb{Z}\to[0,\infty)$.
\end{lemma}

\begin{proof}
Recall that $v_j=(1,a_j,\ldots,a_j^{d-1})\in\mathbb Q^d$ and
$\phi_j(u)=\langle v_j,u\rangle$ for $u\in\mathbb{Z}^d$, where
$\langle\cdot,\cdot\rangle$ is the standard bilinear form on $\mathbb Q^d$.

\proofstep{1}{Vandermonde independence}
If
$j_1,\ldots,j_q$ are distinct indices and $q\le d$, then the vectors
$v_{j_1},\ldots,v_{j_q}$ are linearly independent over $\mathbb Q$.
Indeed, it is enough to consider the first $q$ coordinates of these
vectors. The resulting $q\times q$ matrix is
\[
\begin{pmatrix}
1&a_{j_1}&\cdots&a_{j_1}^{q-1}\\
1&a_{j_2}&\cdots&a_{j_2}^{q-1}\\
\vdots&\vdots&&\vdots\\
1&a_{j_q}&\cdots&a_{j_q}^{q-1}
\end{pmatrix},
\]
whose determinant is $\prod_{1\le \alpha<\beta\le q}
(a_{j_\beta}-a_{j_\alpha})$.
Since the integers $a_1,\ldots,a_m$ are pairwise distinct, this
determinant is nonzero.

\proofstep{2}{The HBL rank condition}
We verify the rank condition in Proposition~\ref{prop:discrete-HBL} with
the common weight $s_1=\cdots=s_m=\frac{d}{m}$.
Let $H\le\mathbb{Z}^d$ be an arbitrary subgroup and put $r:=\operatorname{rank}H$.
If $r=0$, the required inequality is immediate. Suppose that
$1\le r\le d$ and let $V:=\operatorname{span}_{\mathbb Q}H\le\mathbb Q^d$. Then $\dim_{\mathbb Q}V=r$ and $\dim_{\mathbb Q}V^\perp=d-r$.

Since $\phi_j(H)$ is a subgroup of $\mathbb{Z}$, its rank is either zero or one. Furthermore,
\begin{align*}
\operatorname{rank}\phi_j(H)=0
&\iff \phi_j(H)=\{0\}\\
&\iff \langle v_j,h\rangle=0
       \quad\text{for every }h\in H\\
&\iff \langle v_j,x\rangle=0
       \quad\text{for every }x\in V\\
&\iff v_j\in V^\perp.
\end{align*}
Consequently,
\[
\operatorname{rank}\phi_j(H)
=
\begin{cases}
0,&v_j\in V^\perp,\\
1,&v_j\notin V^\perp.
\end{cases}
\]

We claim that at most $d-r$ of the vectors $v_1,\ldots,v_m$ can belong
to $V^\perp$. Otherwise, $V^\perp$, which has dimension $d-r$, would
contain $d-r+1$ of the vectors $v_j$. Since $d-r+1\le d$, the
Vandermonde property would make those vectors linearly independent,
which is impossible in a space of dimension $d-r$. Therefore
\[
|\{j:v_j\in V^\perp\}|\le d-r.
\]
It follows that at least $m-(d-r)=m-d+r$ of the homomorphisms $\phi_j$ have a nonzero image on $H$. Hence
\begin{equation}
\label{eq:Vandermonde-rank-sum}
\sum_{j=1}^m\operatorname{rank}\phi_j(H)
\ge m-d+r.
\end{equation}

\proofstep{3}{Verification of the weighted inequality}
Multiplying \eqref{eq:Vandermonde-rank-sum} by $d/m$ gives
\[
\sum_{j=1}^m\frac{d}{m}\operatorname{rank}\phi_j(H)
\ge
\frac{d}{m}(m-d+r).
\]
Moreover,
\begin{align*}
\frac{d}{m}(m-d+r)-r
&=
\frac{d(m-d+r)-mr}{m}\\
&=
\frac{dm-d^2+dr-mr}{m}\\
&=
\frac{(m-d)(d-r)}{m}\\
&\ge0,
\end{align*}
because $d\le m$ and $r\le d$. We conclude that
\[
\operatorname{rank}H=r
\le
\sum_{j=1}^m\frac{d}{m}\operatorname{rank}\phi_j(H).
\]
Since $H\le\mathbb{Z}^d$ was arbitrary, the discrete HBL rank condition
holds with $s_j=d/m$ for every $j$.

\proofstep{4}{Conclusion}
Proposition~\ref{prop:discrete-HBL} now yields
\[
\sum_{u\in\mathbb{Z}^d}
\prod_{j=1}^m f_j(\phi_j(u))
\le
\prod_{j=1}^m
\|f_j\|_{\ell_{1/s_j}(\mathbb{Z})}.
\]
Finally,
\[
\frac{1}{s_j}=\frac{m}{d},
\]
and hence the last inequality is precisely
\eqref{eq:HBL-isotropic-vandermonde}. 

\end{proof}

Fix a bijection $b:\mathbb{Z}\to\mathbb{N}$ and define
\begin{equation}
\label{eq:Lambda-md}\Lambda_{m,d} :=\bigl\{(b(\phi_{1}(u)),\ldots,b(\phi
_{m}(u))):u\in\mathbb{Z}^{d}\bigr\} \subset\mathbb{N}^{m}.
\end{equation}
Coordinatewise reindexing preserves cardinalities and sequence norms, so
Lemma~\ref{lem:Vandermonde-HBL} gives the corresponding isotropic product
estimate on $\Lambda_{m,d}$.

\begin{theorem}\label{thm:bayart-problem52}
Let
$m\ge2$ and $d\in\{1,\ldots,m\}$. Then
\begin{equation}
\label{eq:gamma-prod-exact}\mathrm{prod}(m,d)=\frac{m}{d}.
\end{equation}
Thus the quantity in \cite[Problem~5.2]{BayartJEMS} has the value $m/d$ at every integer
dimension. When $m/d$ is an integer, this agrees with the uniformly incident
construction in \cite[Lemma~4.7]{BayartJEMS}.
\end{theorem}

\begin{proof}

The dimensional obstruction, Proposition~\ref{prop:gamma-prod-upper}, gives
\begin{equation}
\label{eq:gamma-upper-theorem33}
\mathrm{prod}(m,d)\le\frac{m}{d}.
\end{equation}

Now, choose distinct integers $a_1,\ldots,a_m$ and define
$\phi_j:\mathbb{Z}^d\longrightarrow\mathbb{Z}$, as in \eqref{eq:Vandermonde-forms}. Let $\Phi_d:\mathbb{Z}^d\longrightarrow\mathbb{Z}^m$ be defined by
\[
 \Phi_d(u):=(\phi_1(u),\ldots,\phi_m(u)).
\]
This map is injective. In fact, if $\Phi_d(u)=0$, then in particular $\phi_1(u)=\cdots=\phi_d(u)=0$. Thus
\[
\begin{pmatrix}
1&a_1&\cdots&a_1^{d-1}\\
1&a_2&\cdots&a_2^{d-1}\\
\vdots&\vdots&&\vdots\\
1&a_d&\cdots&a_d^{d-1}
\end{pmatrix}
\begin{pmatrix}
u_1\\u_2\\ \vdots\\u_d
\end{pmatrix}
=0.
\]
Since the Vandermonde matrix is invertible, $u=0$.

Fix a bijection $b:\mathbb{Z}\longrightarrow\mathbb{N}$ and define
\begin{equation}
\label{eq:Lambda-md-detailed}
\Lambda_{m,d}
:=
\left\{
\bigl(b(\phi_1(u)),\ldots,b(\phi_m(u))\bigr):
u\in\mathbb{Z}^d
\right\}
\subset\mathbb{N}^m.
\end{equation}
Because both $\Phi_d$ and $b$ are injective, the parametrization in
\eqref{eq:Lambda-md-detailed} is injective. In particular,
$\Lambda_{m,d}$ is infinite.

Let $x^{(1)},\ldots,x^{(m)}\in c_{00}$. For $1\le j\le m$, define
\[
 f_j:\mathbb{Z}\longrightarrow[0,\infty),\qquad
 f_j(t):=\bigl|x^{(j)}_{b(t)}\bigr|.
\]
Each $f_j$ is finitely supported. Since $b$ is a
bijection,
\begin{equation}
\label{eq:reindexing-norm}
\|f_j\|_{\ell_{m/d}(\mathbb{Z})}
=
\|x^{(j)}\|_{\ell_{m/d}(\mathbb{N})}.
\end{equation}
The injectivity of the parametrization of $\Lambda_{m,d}$ yields
\begin{align*}
\sum_{\mathbf i\in\Lambda_{m,d}}
\prod_{j=1}^m|x^{(j)}_{i_j}|
&=
\sum_{u\in\mathbb{Z}^d}
\prod_{j=1}^m
\bigl|x^{(j)}_{b(\phi_j(u))}\bigr|\\
&=
\sum_{u\in\mathbb{Z}^d}
\prod_{j=1}^m f_j(\phi_j(u)).
\end{align*}
By the Vandermonde HBL estimate,
Lemma~\ref{lem:Vandermonde-HBL}, and
\eqref{eq:reindexing-norm},
\begin{equation}
\label{eq:Lambda-md-product-estimate}
\sum_{\mathbf i\in\Lambda_{m,d}}
\prod_{j=1}^m|x^{(j)}_{i_j}|
\le
\prod_{j=1}^m
\|x^{(j)}\|_{\ell_{m/d}}.
\end{equation}
Thus $\Lambda_{m,d}$ is $(m/d)$-product summable, with constant one.

We verify that $\dim(\Lambda_{m,d})=d$. Let $n\in\mathbb{N}$ and let
$A_1,\ldots,A_m\subset\mathbb{N}$ be finite sets such that
$|A_j|\le n$. Applying \eqref{eq:Lambda-md-product-estimate} to
$x^{(j)}=\mathbf 1_{A_j}$ gives
\begin{align*}
\bigl|\Lambda_{m,d}\cap(A_1\times\cdots\times A_m)\bigr|
&=
\sum_{\mathbf i\in\Lambda_{m,d}}
\prod_{j=1}^m\mathbf 1_{A_j}(i_j)\\
&\le
\prod_{j=1}^m
\|\mathbf 1_{A_j}\|_{\ell_{m/d}}\\
&=
\prod_{j=1}^m|A_j|^{d/m}\\
&\le
\prod_{j=1}^m n^{d/m}
=n^d.
\end{align*}
Taking the supremum over all such $A_1,\ldots,A_m$, we obtain
\begin{equation}
\label{eq:Lambda-upper-growth}
\Lambda_{m,d}(n)\le n^d.
\end{equation}
Therefore
\[
\dim(\Lambda_{m,d})
=
\limsup_{n\to\infty}
\frac{\log\Lambda_{m,d}(n)}{\log n}
\le d.
\]
For $N\in\mathbb{N}$, consider the integer box $Q_N:=\{-N,\ldots,N\}^d\subset\mathbb{Z}^d$. It has cardinality $|Q_N|=(2N+1)^d$. Since $\Phi_d$ is injective,
\begin{equation}
\label{eq:Phi-box-cardinality}
|\Phi_d(Q_N)|=(2N+1)^d.
\end{equation}
For $u=(u_1,\ldots,u_d)\in Q_N$, we have
\begin{align*}
|\phi_j(u)|
&=
\left|u_1+a_ju_2+\cdots+a_j^{d-1}u_d\right|\\
&\le
N\left(1+|a_j|+\cdots+|a_j|^{d-1}\right).
\end{align*}
Put $M_j:=1+|a_j|+\cdots+|a_j|^{d-1}$. Then $\phi_j(Q_N)
\subset[-M_jN,M_jN]\cap\mathbb{Z}$, and consequently $|\phi_j(Q_N)|\le2M_jN+1$.
Hence there is a constant $C\ge1$ ($C=2\max_{j=1,\ldots,m}M_j+1$), independent of $N$, such that
\begin{equation}
\label{eq:coordinate-image-bound}
|\phi_j(Q_N)|\le CN
\qquad(j=1,\ldots,m,\ N\ge1).
\end{equation}
Define $A_{j,N}:=b\bigl(\phi_j(Q_N)\bigr)\subset\mathbb{N}$.
Since $b$ is bijective, \eqref{eq:coordinate-image-bound} implies $|A_{j,N}|\le CN$.
Moreover, the coordinatewise reindexed image of $\Phi_d(Q_N)$ is contained in $\Lambda_{m,d}\cap
(A_{1,N}\times\cdots\times A_{m,N})$. Using \eqref{eq:Phi-box-cardinality}, we obtain
\begin{equation}
\label{eq:Lambda-lower-growth}
\Lambda_{m,d}(CN)
\ge
\bigl|\Lambda_{m,d}\cap
(A_{1,N}\times\cdots\times A_{m,N})\bigr|
\ge(2N+1)^d.
\end{equation}
It follows that
\begin{align*}
\dim(\Lambda_{m,d})
&\ge
\limsup_{N\to\infty}
\frac{\log\Lambda_{m,d}(CN)}{\log(CN)}\\
&\ge
\lim_{N\to\infty}
\frac{d\log(2N+1)}{\log(CN)}
=d.
\end{align*}
Together with \eqref{eq:Lambda-upper-growth}, this proves
\begin{equation}
\label{eq:Lambda-exact-dimension}
\dim(\Lambda_{m,d})=d.
\end{equation}

Equations \eqref{eq:Lambda-md-product-estimate} and
\eqref{eq:Lambda-exact-dimension} exhibit an infinite support of
combinatorial dimension exactly $d$ that is $(m/d)$-product summable.
Therefore, by the definition of $\mathrm{prod}(m,d)$,
\[
\mathrm{prod}(m,d)\ge\frac{m}{d}.
\]
Combining this with \eqref{eq:gamma-upper-theorem33} gives
\[
\mathrm{prod}(m,d)=\frac{m}{d},
\]
as claimed.

\end{proof}

\subsubsection*{Higher-rank HBL supports}\label{sec:fusion-frame-HBL}

Equal-rank tight fusion frames extend the construction beyond the Vandermonde models. Their rank data fit the discrete H\"older--Brascamp--Lieb criterion; see \cite{CFMWZ}.

\begin{definition}[\cite{CFMWZ}]\label{def:TFF}
Let $1\le k\le D$ and let $L_1,\ldots,L_m\le\mathbb{R}^D$ satisfy
$\dim L_j=k$. For each $j$, let
\[
 P_{L_j}:\mathbb{R}^D\longrightarrow\mathbb{R}^D
\]
be the orthogonal projection onto $L_j$, and let
$I_D:\mathbb{R}^D\to\mathbb{R}^D$ be the identity operator. The family
$(L_j)_{j=1}^m$ is an \emph{equal-rank tight fusion frame} if
\begin{equation}\label{eq:TFF-tight}
 \sum_{j=1}^mP_{L_j}=\frac{mk}{D}I_D.
\end{equation}
\end{definition}

The rank condition is encoded by property $\mathcal S$, introduced by
Needham--Shonkwiler \cite[Section~4]{NeedhamShonkwiler}.

\begin{definition}[\cite{NeedhamShonkwiler}, Section~4]\label{def:property-S}
Let $1\le k\le D$ and let $K_1,\ldots,K_m\le\mathbb{R}^D$ satisfy
$\dim K_j=D-k$. The family
$(K_j)_{j=1}^m$ has \emph{property $\mathcal S$} if
\begin{equation}\label{eq:property-S-kernels}
 \sum_{j=1}^m\dim(V\cap K_j)
 \le\frac{m(D-k)}{D}\dim V
\end{equation}
for every proper linear subspace $V<\mathbb{R}^D$.
\end{definition}

{
The existence of the required equal-rank tight fusion frames is taken from
\cite{CFMWZ}, while the implication from tightness to property $\mathcal S$
is \cite[Proposition~4.6]{NeedhamShonkwiler}. {Lattice homomorphisms require rational subspaces. The next lemma provides them.}

\begin{lemma}\label{lem:rational-property-S}
Let $1\le r_0<D$. Suppose that $K_1,\ldots,K_m\le\mathbb{R}^D$ are
$r_0$-dimensional subspaces satisfying
\begin{equation}\label{eq:property-S-rational-lemma}
 \sum_{j=1}^m\dim(V\cap K_j)\le \frac{mr_0}{D}\dim V
\end{equation}
for every proper linear subspace $V<\mathbb{R}^D$. Then there are
$r_0$-dimensional subspaces $K'_1,\ldots,K'_m\le\mathbb{R}^D$, each admitting a
basis in $\mathbb Q^D$, for which the same inequalities hold.
\end{lemma}

\begin{proof}
For $1\le r<D$, let $\mathcal G_r$ denote the compact set of
$r$-dimensional subspaces of $\mathbb{R}^D$, endowed with the metric
$d(E,F)=\|P_E-P_F\|$, where $P_E$ denotes orthogonal projection onto $E$.
For a tuple $\mathbf K=(K_1,\ldots,K_m)$ of $r_0$-dimensional subspaces set
\[
 F_r(\mathbf K):=\max_{V\in\mathcal G_r}
       \sum_{j=1}^m\dim(V\cap K_j).
\]
The map
\[
 (V,K)\longmapsto \dim(V\cap K)
\]
is upper semicontinuous.  Indeed, in local matrix coordinates, if $B_V$ and
$B_K$ are full-column-rank matrices whose columns span $V$ and $K$, then
\[
 \dim(V\cap K)=r+r_0-\operatorname{rank}[B_V\;B_K],
\]
and matrix rank is lower semicontinuous.  Compactness of $\mathcal G_r$ then
implies that $F_r$ is upper semicontinuous: if
$\mathbf K^{(n)}\to\mathbf K$ and $V_n$ realizes
$F_r(\mathbf K^{(n)})$, a convergent subsequence of $(V_n)$ and the upper semicontinuity of $(V,K)\mapsto\dim(V\cap K)$ established in the previous sentence give
\[
 \limsup_{n\to\infty}F_r(\mathbf K^{(n)})\le F_r(\mathbf K).
\]

Since $F_r$ is integer-valued, the condition
\[
 F_r(\mathbf K)\le \left\lfloor\frac{mr_0}{D}r\right\rfloor
\]
is open.  Thus the tuples satisfying \eqref{eq:property-S-rational-lemma} for
all proper $V$ form an open subset of the product of the spaces
$\mathcal G_{r_0}$; there are only the finitely many dimensions
$r=1,\ldots,D-1$ to consider.

Finally, $r_0$-dimensional subspaces admitting a basis in $\mathbb Q^D$ are
dense in $\mathcal G_{r_0}$.  To see this, represent a given subspace by a
full-column-rank matrix $B\in M_{D\times r_0}(\mathbb{R})$ and approximate $B$
by rational matrices $B'\in M_{D\times r_0}(\mathbb Q)$.  For $B'$ sufficiently
close to $B$, it still has rank $r_0$, and its column space converges to the
column space of $B$.  Hence rational tuples are dense in the product
$\mathcal G_{r_0}^m$.  Approximating the given tuple inside the open set just
obtained yields $K'_1,\ldots,K'_m$ with the required properties.
\end{proof}}

\begin{theorem}\label{thm:fusion-frame-product}
Let $m,D,k$ be
positive integers with $1\le k\le D\le mk$. Suppose that there exists an
equal-rank tight fusion frame of $m$ $k$-dimensional subspaces of
$\mathbb{R}^{D}$. Then there exists a support $\Lambda\subset\mathbb{N}^{m}$
such that
\[
\dim(\Lambda)=\frac{D}{k}
\]
and
\begin{equation}
\label{eq:fusion-product-estimate}\sum_{(i_{1},\ldots,i_{m})\in\Lambda}
\prod_{j=1}^{m} |x^{(j)}_{i_{j}}| \le\prod_{j=1}^{m}\|x^{(j)}\|_{\ell_{mk/D}}%
\end{equation}
for all $x^{(1)},\ldots,x^{(m)}\in c_{00}$. Moreover, there is a constant
$c>0$ and arbitrarily large integers $N$ for which one can find finite sets
$B_N\subset\Lambda$ satisfying
\begin{equation}
\label{eq:fusion-host-blocks}
|B_N|\ge cN^{D/k},\qquad |\pi_j(B_N)|\le N\quad(j=1,\ldots,m),
\end{equation}
where $\pi_j$ denotes the $j$th coordinate projection.
Consequently,
\begin{equation}
\label{eq:fusion-gamma-exact}\mathrm{prod}\!\left(  m,\frac{D}%
{k}\right)  =\frac{mk}{D}.
\end{equation}

\end{theorem}

\begin{proof}
\proofstep{1}{Reduction to rational kernels}
If $k=D$, take the diagonal support. Its first $N$ diagonal points give
$|B_N|=N$ and $|\pi_j(B_N)|=N$ for every $j$, so all conclusions follow.
Assume $k<D$.
If $L_j^\perp=:K_j$, then $P_{K_j}=I_D-P_{L_j}$, and therefore
\[
 \sum_{j=1}^m P_{K_j}
 =mI_D-\frac{mk}{D}I_D
 =\frac{m(D-k)}{D}I_D.
\]
Thus taking orthogonal complements converts a tight fusion frame whose
subspaces have dimension $k$ into one whose subspaces have dimension $D-k$.
The converse follows from the same calculation. Hence the existence hypothesis
is equivalent to the existence of a tight fusion frame with subspace dimension
$D-k$.

By \cite[Proposition~4.6]{NeedhamShonkwiler}, the tight fusion frame
$(K_j)_{j=1}^m$ has property $\mathcal S$, i.e.
\[
 \sum_{j=1}^m\dim(V\cap K_j)
 \le \frac{m(D-k)}D\dim V
 \qquad(V<\mathbb{R}^D\text{ proper}).
\]
Applying Lemma~\ref{lem:rational-property-S} with $r_0=D-k$, we may replace it by a
property-$\mathcal S$ tuple $K_1,\ldots,K_m\subset\mathbb{R}^D$ such that every
$K_j$ admits a basis in $\mathbb Q^D$.
Put
\[
K_{j}^{\mathbb{Q}}:=K_{j}\cap\mathbb{Q}^{D}.
\]
Then $K_{j}^{\mathbb{Q}}$ is a $(D-k)$-dimensional $\mathbb{Q}$-subspace of
$\mathbb{Q}^{D}$. If $V\subset\mathbb Q^D$ is rational and
$V_{\mathbb{R}}:=\operatorname{span}_{\mathbb{R}}V$, rationality gives
\[
 \dim_{\mathbb Q}(V\cap K_j^{\mathbb Q})
 =\dim_{\mathbb{R}}(V_{\mathbb{R}}\cap K_j).
\]
Thus \eqref{eq:property-S-kernels} applies directly to the intersections
that occur in the rank computation below.
\proofstep{2}{Construction of the lattice homomorphisms}
For each $j$, choose a surjective rational linear map
\[
 \psi_j:\mathbb Q^D\longrightarrow\mathbb Q^k,
 \qquad \ker\psi_j=K_j^{\mathbb Q}.
\]
Represent $\psi_j$ by a $k\times D$ matrix with rational entries. Multiplying
that matrix by a common positive denominator does not change its kernel or its
rank and produces an integer matrix. Its restriction to $\mathbb{Z}^D$ is a
homomorphism
\[
\varphi_{j}:\mathbb{Z}^{D}\longrightarrow\mathbb{Z}^{k}
\]
whose $\mathbb Q$-linear extension has kernel $K_j^{\mathbb Q}$ and rank $k$.
\proofstep{3}{The discrete HBL rank inequality}
Let $H\le\mathbb{Z}^D$ and put
$V:=\operatorname{span}_{\mathbb Q}H$. Tensoring the image with $\mathbb Q$
gives
\[
 \operatorname{rank}\varphi_j(H)
 =\dim_{\mathbb Q}\psi_j(V)
 =\dim_{\mathbb Q}V-\dim_{\mathbb Q}(V\cap K_j^{\mathbb Q}).
\]
If $V=\mathbb{Q}^{D}$, then $\operatorname{rank}\varphi_{j}(H)=k$ for every $j$,
so
\[
\sum_{j=1}^{m}\operatorname{rank}\varphi_{j}(H)=mk =\frac{mk}{D}%
\operatorname{rank}H.
\]
If $V$ is proper, summing and using property $\mathcal{S}$ for the real span
of $V$ yields
\begin{align*}
\sum_{j=1}^{m}\operatorname{rank}\varphi_{j}(H)  &  =m\dim_{\mathbb{Q}}%
V-\sum_{j=1}^{m}\dim_{\mathbb{Q}}(V\cap K_{j}^{\mathbb{Q}})\\
&  \ge\frac{mk}{D}\dim_{\mathbb{Q}} V =\frac{mk}{D}\operatorname{rank}H.
\end{align*}
Thus the discrete H\"older--Brascamp--Lieb rank criterion holds with the
common weight
\[
 s=\frac{D}{mk}.
\]
\proofstep{4}{The product estimate and the support}
The existence of a fusion frame forces $mk\ge D$, so $0<s\le1$, as required
in Proposition~\ref{prop:discrete-HBL}. Applying that proposition with
$G=\mathbb{Z}^D$ and $G_j=\varphi_j(\mathbb{Z}^D)$ gives
\begin{equation}
\label{eq:fusion-HBL-lattice}\sum_{u\in\mathbb{Z}^{D}}\prod_{j=1}^{m}
f_{j}(\varphi_{j}(u)) \le\prod_{j=1}^{m}\|f_{j}\|_{\ell_{mk/D}(\varphi_j(\mathbb{Z}^D))}%
\end{equation}
for all nonnegative finitely supported functions on the image lattices
$\varphi_{j}(\mathbb{Z}^{D})$. Property $\mathcal{S}$ also forces
\[
\bigcap_{j=1}^{m} K_{j}^{\mathbb{Q}}=\{0\}.
\]
Indeed, if a nonzero rational subspace $V$ were contained in every
$K_j^{\mathbb Q}$, then $V_{\mathbb{R}}$ would be a nonzero proper subspace of
$\mathbb{R}^D$ and
\[
 \sum_{j=1}^m\dim(V_{\mathbb{R}}\cap K_j)
 =m\dim V
 >\frac{m(D-k)}{D}\dim V,
\]
contrary to property $\mathcal S$. Hence
$\bigcap_jK_j^{\mathbb Q}=\{0\}$. Since the kernel of the product map is this
intersection,
\[
\Psi_D=(\varphi_{1},\ldots,\varphi_{m}):\mathbb{Z}^{D}\longrightarrow(\mathbb{Z}%
^{k})^{m}
\]
is injective. Fix bijections from the countable lattices $\varphi_{j}%
(\mathbb{Z}^{D})$ onto subsets of $\mathbb{N}$ and let $\Lambda\subset
\mathbb{N}^{m}$ be the coordinatewise image of $\Psi_D(\mathbb{Z}^{D})$.
Reindexing \eqref{eq:fusion-HBL-lattice} gives
\eqref{eq:fusion-product-estimate}. Testing this estimate on indicator
functions gives $\dim(\Lambda)\le D/k$.

\proofstep{5}{Sharpness of the dimension}
For the reverse inequality, let
$Q_{R}=[-R,R]^{D}\cap\mathbb{Z}^{D}$. Since each $\varphi_j$ is represented by a fixed integer matrix of rank $k$,
its image lattice $L_j:=\varphi_j(\mathbb{Z}^D)$ has rank $k$. The set
$\varphi_j(Q_R)$ lies in $L_j$ inside a Euclidean ball of radius $C_jR$.
Choosing a lattice basis of $L_j$ therefore gives
\[
 |\varphi_j(Q_R)|\le C'_j(1+R)^k.
\]
After enlarging one constant $C\ge1$ and restricting to $R\ge1$, we have
\[
|\varphi_{j}(Q_{R})|\le C R^{k}\qquad(j=1,\ldots,m).
\]
Injectivity gives $|\Psi_D(Q_{R})|=(2R+1)^{D}$. Let $B_R$ be the coordinatewise
reindexed image of $\Psi_D(Q_R)$ in $\Lambda$ and set
\[
N_R:=\lceil C R^k\rceil.
\]
Then $N_R\to\infty$, $|\pi_j(B_R)|\le N_R$, and, after decreasing a
constant $c>0$ if necessary,
\[
|B_R|=(2R+1)^D\ge cN_R^{D/k}.
\]
Thus the sets $B_R$, indexed by the arbitrarily large integers $N_R$, satisfy
\eqref{eq:fusion-host-blocks}. They also show directly that
$\dim(\Lambda)\ge D/k$. Hence $\dim(\Lambda)=D/k$. The lower bound in
\eqref{eq:fusion-gamma-exact} follows from \eqref{eq:fusion-product-estimate},
and the reverse inequality is Proposition~\ref{prop:gamma-prod-upper}.
\end{proof}

\subsection{Dimension reduction}\label{sec:dimension-thinning}

Random constructions of sets with prescribed combinatorial dimension go back to Blei--K\"orner \cite{BleiKorner}; a multilinear adaptation appears in \cite[Lemma~4.3]{BayartJEMS}. The form needed here retains points independently on finite blocks while controlling both block cardinality and all rectangle counts. The restriction $d>1$ enters only through the uniform rectangle estimate.

\begin{lemma}\label{lem:dimension-thinning}
Let $m\ge2$ and let
$1<d_{0}\le m$. Suppose that a support $\Lambda_{0}\subset\mathbb{N}^{m}$
satisfies
\begin{equation}
\label{eq:thinning-host-product}\sum_{(i_{1},\ldots,i_{m})\in\Lambda_{0}}%
\prod_{j=1}^{m}|x^{(j)}_{i_{j}}| \le C\prod_{j=1}^{m}\|x^{(j)}\|_{\ell_{m/d_0}}%
\end{equation}
for all $x^{(1)},\ldots,x^{(m)}\in c_{00}$. Assume moreover that
there are arbitrarily large integers $N$ for which one can find a finite set
$B_{N}\subset\Lambda_{0}$ such that
\begin{equation}
\label{eq:thinning-host-blocks}|B_{N}|\ge cN^{d_{0}}, \qquad|\pi_{j}%
(B_{N})|\le N\quad(j=1,\ldots,m),
\end{equation}
where $c>0$ is independent of $N$. Then, for every $d$ with $1<d<d_{0}$, there
exists a support $\Lambda\subset\mathbb{N}^{m}$ such that
\[
\dim(\Lambda)=d
\]
and
\begin{equation}
\label{eq:thinning-conclusion-product}\sum_{(i_{1},\ldots,i_{m})\in\Lambda
}\prod_{j=1}^{m}|x^{(j)}_{i_{j}}| \le C\prod_{j=1}^{m}\|x^{(j)}\|_{\ell_{m/d_0}}%
\end{equation}
for all $x^{(1)},\ldots,x^{(m)}\in c_{00}$.
\end{lemma}

\begin{proof}
{
Fix $d$ with $1<d<d_0$. We select finite subsets of the sets $B_N$ by working on finite probability spaces.

\proofstep{1}{The Bernoulli probability space}
For each admissible $N$, consider
\[
\Omega_N:=2^{B_N}=\{H:H\subseteq B_N\},\qquad \mathcal F_N:=2^{\Omega_N}.
\]
For every $H\subseteq B_N$, define
\begin{equation}\label{eq:thinning-measure}
\mu_N(\{H\}):=
\bigl(N^{d-d_0}\bigr)^{|H|}
\bigl(1-N^{d-d_0}\bigr)^{|B_N|-|H|},
\end{equation}
and extend $\mu_N$ to every subset of $\Omega_N$ by finite summation. Since $d<d_0$, we have $0<N^{d-d_0}\le1$. Grouping subsets of $B_N$ according to their cardinalities and using the binomial theorem gives
\[
\mu_N(\Omega_N)
=\sum_{r=0}^{|B_N|}\binom{|B_N|}{r}
\bigl(N^{d-d_0}\bigr)^r
\bigl(1-N^{d-d_0}\bigr)^{|B_N|-r}
=1.
\]
Thus $(\Omega_N,\mathcal F_N,\mu_N)$ is a finite probability space.
For $u\in B_N$, let
{
\[
\chi_u:\Omega_N\longrightarrow\{0,1\},\qquad
\chi_u(H):=\mathbf 1_{\{u\in H\}}.
\]}
Every subset containing $u$ has a unique representation $H=\{u\}\cup L$, with $L\subseteq B_N\setminus\{u\}$. Consequently,
\begin{equation}\label{eq:thinning-point-integral}
\int_{\Omega_N}\chi_u\,d\mu_N=N^{d-d_0}.
\end{equation}

\proofstep{2}{Subsets having sufficiently large cardinality}
Since $|H|=\sum_{u\in B_N}\chi_u(H)$, linearity of the integral and \eqref{eq:thinning-point-integral} give
\begin{equation}\label{eq:thinning-mean-size-measure}
\int_{\Omega_N}|H|\,d\mu_N(H)
=N^{d-d_0}|B_N|
\ge cN^d.
\end{equation}
Consider
\[
\mathcal A_N:=\left\{H\subseteq B_N:|H|\le \frac12N^{d-d_0}|B_N|\right\}.
\]
For every $\lambda>0$ and $H\in\mathcal A_N$,
\[
e^{-\lambda|H|}\ge
\exp\!\left(-\frac{\lambda}{2}N^{d-d_0}|B_N|\right).
\]
Hence
\begin{equation}\label{eq:thinning-small-set-measure}
\mu_N(\mathcal A_N)
\exp\!\left(-\frac{\lambda}{2}N^{d-d_0}|B_N|\right)
\le \int_{\Omega_N}e^{-\lambda|H|}\,d\mu_N(H).
\end{equation}
The integral on the right is computed exactly:
\begin{align*}
\int_{\Omega_N}e^{-\lambda|H|}\,d\mu_N(H)
&=\sum_{H\subseteq B_N}e^{-\lambda|H|}
\bigl(N^{d-d_0}\bigr)^{|H|}
\bigl(1-N^{d-d_0}\bigr)^{|B_N|-|H|}\\
&=\bigl(1-N^{d-d_0}+N^{d-d_0}e^{-\lambda}\bigr)^{|B_N|}\\
&=\bigl(1+N^{d-d_0}(e^{-\lambda}-1)\bigr)^{|B_N|}\\
&\le \exp\!\left(N^{d-d_0}|B_N|(e^{-\lambda}-1)\right),
\end{align*}
where $1+t\le e^t$ was used. Taking $\lambda=\log2$ in \eqref{eq:thinning-small-set-measure} yields
\[
\mu_N(\mathcal A_N)
\le \exp\!\left[N^{d-d_0}|B_N|\left(\frac{\log2}{2}-\frac12\right)\right]
\le \exp\!\left(-\frac18N^{d-d_0}|B_N|\right)
\le \exp\!\left(-\frac c8N^d\right)\longrightarrow0.
\]
Thus every $H\notin\mathcal A_N$ satisfies
\begin{equation}\label{eq:thinning-lower-size}
|H|>\frac12N^{d-d_0}|B_N|\ge \frac c2N^d.
\end{equation}

\proofstep{3}{Points of $B_N$ in a rectangle}
Let $1\le n\le N$ and let $A_j\subseteq\pi_j(B_N)$ satisfy $|A_j|\le n$ for $j=1,\ldots,m$. Applying \eqref{eq:thinning-host-product} to the indicator functions of the $A_j$ gives
\begin{equation}\label{eq:thinning-host-rectangle}
|B_N\cap(A_1\times\cdots\times A_m)|
\le C\prod_{j=1}^m|A_j|^{d_0/m}
\le Cn^{d_0}.
\end{equation}
For the rectangle $A_1\times\cdots\times A_m$, define
{
\[
X_{A_1,\ldots,A_m}:\Omega_N\longrightarrow\mathbb N_0,\qquad
X_{A_1,\ldots,A_m}(H):=|H\cap(A_1\times\cdots\times A_m)|.
\]}
By \eqref{eq:thinning-point-integral} and \eqref{eq:thinning-host-rectangle},
\begin{equation}\label{eq:thinning-rectangle-integral}
\int_{\Omega_N}X_{A_1,\ldots,A_m}(H)\,d\mu_N(H)
\le CN^{d-d_0}n^{d_0}
=Cn^d\left(\frac nN\right)^{d_0-d}.
\end{equation}

\proofstep{4}{Measure of the sets violating the desired estimate}
Fix $\varepsilon>0$. Choose $K>eC$ so large that
\begin{equation}\label{eq:thinning-K-choice}
K(d_0-d)\ge4m\left(1+\frac1{\log2}\right),
\qquad \frac{K(d_0-d)}2>2.
\end{equation}
For fixed $n$ and fixed $A_j\subseteq\pi_j(B_N)$ with $|A_j|\le n$, put
\[
\mathcal E_{N,n,A_1,\ldots,A_m}
:=\left\{H\subseteq B_N:
X_{A_1,\ldots,A_m}(H)\ge \lceil Kn^{d+\varepsilon}\rceil\right\}.
\]
If $H$ belongs to this set, then it contains a subset of
$B_N\cap(A_1\times\cdots\times A_m)$ having exactly
$\lceil Kn^{d+\varepsilon}\rceil$ elements. For any fixed subset $F\subseteq B_N$, direct summation from \eqref{eq:thinning-measure} gives
\[
\mu_N(\{H\subseteq B_N:F\subseteq H\})
=\bigl(N^{d-d_0}\bigr)^{|F|}.
\]
Set $t:=\lceil Kn^{d+\varepsilon}\rceil$. If
$t>|B_N\cap(A_1\times\cdots\times A_m)|$, then the set
$\mathcal E_{N,n,A_1,\ldots,A_m}$ is empty. Otherwise, finite subadditivity,
\eqref{eq:thinning-host-rectangle}, and $\binom ab\le(ea/b)^b$ give
\[
\mu_N(\mathcal E_{N,n,A_1,\ldots,A_m})
\le
\left(
\frac{eC\,N^{d-d_0}n^{d_0}}{t}
\right)^t.
\]
Since $t\ge Kn^{d+\varepsilon}$,
\[
\frac{eC\,N^{d-d_0}n^{d_0}}{t}
\le
\frac{eC}{K}
\left(\frac nN\right)^{d_0-d}n^{-\varepsilon}
<1.
\]
The base is smaller than one and $t\ge Kn^{d+\varepsilon}$, hence
\begin{align}
\mu_N(\mathcal E_{N,n,A_1,\ldots,A_m})
&\le
\exp\!\left[-Kn^{d+\varepsilon}
\left(\log\frac K{eC}
+(d_0-d)\log\frac Nn
+\varepsilon\log n\right)\right].
\label{eq:thinning-bad-rectangle-measure}
\end{align}

\proofstep{5}{Counting all rectangles}
For $1\le n\le N/2$,
\[
\sum_{r=0}^n\binom Nr\le\left(\frac{eN}{n}\right)^n,
\]
whereas, for $N/2<n\le N$,
\[
\sum_{r=0}^n\binom Nr\le2^N\le2^{2n}
\le\exp\!\left(2n\log\frac{eN}{n}\right).
\]
Thus, for every $1\le n\le N$, the number of $m$-fold rectangles under consideration is at most
\begin{equation}\label{eq:thinning-rectangle-count}
\exp\!\left(2mn\log\frac{eN}{n}\right).
\end{equation}
Let $\mathcal E_{N,\varepsilon}$ be the union of all the sets
$\mathcal E_{N,n,A_1,\ldots,A_m}$ over $1\le n\le N$ and all admissible coordinate sets. From finite subadditivity, \eqref{eq:thinning-bad-rectangle-measure}, and \eqref{eq:thinning-rectangle-count},
\begin{align}
\mu_N(\mathcal E_{N,\varepsilon})
\le\sum_{n=1}^N\exp\!\Bigg[&2mn\log\frac{eN}{n}
-Kn^{d+\varepsilon}\log\frac K{eC}\\
&-K(d_0-d)n^{d+\varepsilon}\log\frac Nn
-K\varepsilon n^{d+\varepsilon}\log n\Bigg].
\label{eq:thinning-total-bad-measure}
\end{align}

\proofstep{6}{Vanishing of the bad-set measure}
We verify directly that the right-hand side of \eqref{eq:thinning-total-bad-measure} tends to zero. First let $1\le n\le N/2$. Since
\[
\log\frac{eN}{n}=1+\log\frac Nn
\le\left(1+\frac1{\log2}\right)\log\frac Nn,
\]
we may discard the second and fourth negative terms in the exponent and use \eqref{eq:thinning-K-choice} together with $n\le n^{d+\varepsilon}$ to obtain
\[
2mn\log\frac{eN}{n}
-K(d_0-d)n^{d+\varepsilon}\log\frac Nn
\le-\frac{K(d_0-d)}2n^{d+\varepsilon}\log\frac Nn.
\]
Hence it suffices to show
\begin{equation}\label{eq:thinning-half-sum}
\sum_{n=1}^{\lfloor N/2\rfloor}
\exp\!\left[-\frac{K(d_0-d)}2n^{d+\varepsilon}\log\frac Nn\right]
\longrightarrow0.
\end{equation}
Split this sum at $\sqrt N$. If $1\le n\le\sqrt N$, then
$\log(N/n)\ge\frac12\log N$ and $n^{d+\varepsilon}\ge1$, so this part is bounded by
\[
\sqrt N\,N^{-K(d_0-d)/4}
=N^{(2-K(d_0-d))/4}\longrightarrow0
\]
by \eqref{eq:thinning-K-choice}. If $\sqrt N<n\le N/2$, then
$n^{d+\varepsilon}>N^{(d+\varepsilon)/2}$ and
$\log(N/n)\ge\log2$, so this part is bounded by
\[
N\exp\!\left[-\frac{K(d_0-d)\log2}{2}
N^{(d+\varepsilon)/2}\right]\longrightarrow0.
\]
This proves \eqref{eq:thinning-half-sum}.

Now let $N/2<n\le N$. Since $\log(eN/n)\le\log(2e)$, the exponent in
\eqref{eq:thinning-total-bad-measure} is at most
\[
2m\log(2e)n-K\log\frac K{eC}\,n^{d+\varepsilon}.
\]
Because $d+\varepsilon>1$ and $\log(K/(eC))>0$, this is at most
\[
-\frac K2\log\frac K{eC}\,n^{d+\varepsilon}
\]
for all sufficiently large $N$, uniformly in $n>N/2$. The contribution of this range therefore tends to zero as well. Consequently,
\begin{equation}\label{eq:thinning-bad-measure-zero}
\mu_N(\mathcal E_{N,\varepsilon})\longrightarrow0.
\end{equation}

\proofstep{7}{Existence of a good subset of $B_N$}
By finite subadditivity,
\[
\mu_N\bigl(\Omega_N\setminus(\mathcal A_N\cup\mathcal E_{N,\varepsilon})\bigr)
\ge1-\mu_N(\mathcal A_N)-\mu_N(\mathcal E_{N,\varepsilon})\longrightarrow1.
\]
For every sufficiently large admissible $N$, this measure is positive, so the set being measured is nonempty. Choose
\[
H_N\in\Omega_N\setminus(\mathcal A_N\cup\mathcal E_{N,\varepsilon}).
\]
Then
\begin{equation}\label{eq:thinning-good-size-measure}
|H_N|\ge\frac c2N^d,
\end{equation}
and, whenever $|A_j|\le n\le N$,
\begin{equation}\label{eq:thinning-good-rectangle-measure}
|H_N\cap(A_1\times\cdots\times A_m)|
<(K+1)n^{d+\varepsilon}.
\end{equation}

\proofstep{8}{Simultaneous selection and construction of $\Lambda$}
For every integer $s\ge1$, consider simultaneously
$\varepsilon=1,1/2,\ldots,1/s$. For each fixed $s$, \eqref{eq:thinning-bad-measure-zero} and finite subadditivity give
\[
\mu_N\!\left(
\mathcal A_N\cup\bigcup_{r=1}^s\mathcal E_{N,1/r}
\right)
\le
\mu_N(\mathcal A_N)+\sum_{r=1}^s\mu_N(\mathcal E_{N,1/r})
\longrightarrow0.
\]
Hence, for all sufficiently large admissible $N$, the complement of this union is nonempty.
We may therefore choose increasing admissible integers $N_s$ and sets
$H_s\subseteq B_{N_s}$ such that
\begin{equation}\label{eq:thinning-simultaneous-size}
|H_s|\ge\frac c2N_s^d
\end{equation}
and, for every $r\in\{1,\ldots,s\}$, there is a constant $K_r$ such that
\begin{equation}\label{eq:thinning-simultaneous}
|H_s\cap(A_1\times\cdots\times A_m)|
\le K_r n^{d+1/r}
\end{equation}
whenever $|A_j|\le n\le N_s$.

For each $j$, choose finite sets $I_{j,s}\subset\mathbb{N}$, pairwise disjoint as $s$ varies, with
$|I_{j,s}|=|\pi_j(B_{N_s})|$. Reindex the $j$th coordinates of $H_s$ bijectively from
$\pi_j(B_{N_s})$ onto $I_{j,s}$, denote the resulting set by $\widetilde H_s$, and set
\[
\Lambda:=\bigcup_{s=1}^\infty\widetilde H_s.
\]

\proofstep{9}{Computation of the combinatorial dimension}
Since
\[
|\widetilde H_s|\ge\frac c2N_s^d,
\qquad |\pi_j(\widetilde H_s)|\le N_s,
\]
the definition of combinatorial dimension gives $\dim(\Lambda)\ge d$.

For the converse, fix $r\ge1$ and finite sets $A_j\subset\mathbb{N}$ with $|A_j|\le n$. Since the sets $I_{j,s}$ are pairwise disjoint in $s$ for each fixed $j$,
\begin{equation}\label{eq:thinning-budget}
\sum_s\max_{1\le j\le m}|A_j\cap I_{j,s}|
\le\sum_{j=1}^m\sum_s|A_j\cap I_{j,s}|
\le\sum_{j=1}^m|A_j|\le mn.
\end{equation}
The finitely many blocks with $s<r$ contribute at most
$\sum_{s<r}|\widetilde H_s|$. For $s\ge r$, the transported form of
\eqref{eq:thinning-simultaneous} gives
\[
|\widetilde H_s\cap(A_1\times\cdots\times A_m)|
\le K_r\left(\max_{1\le j\le m}|A_j\cap I_{j,s}|\right)^{d+1/r}.
\]
Because $d+1/r>1$, nonnegative numbers satisfy
$\sum_s b_s^{d+1/r}\le(\sum_s b_s)^{d+1/r}$. Hence, by
\eqref{eq:thinning-budget},
\[
|\Lambda\cap(A_1\times\cdots\times A_m)|
\le\sum_{s<r}|\widetilde H_s|+K_r(mn)^{d+1/r}.
\]
Thus $\dim(\Lambda)\le d+1/r$. Since this holds for every $r\ge1$,
$\dim(\Lambda)\le d$, and therefore $\dim(\Lambda)=d$.

\proofstep{10}{Preservation of the product estimate}
For each $s$, the set $\widetilde H_s$ is a coordinatewise reindexed copy of a subset of
$\Lambda_0$. Therefore
\[
\sum_{\mathbf i\in\widetilde H_s}\prod_{j=1}^m|x^{(j)}_{i_j}|
\le C\prod_{j=1}^m
\left(\sum_{i\in I_{j,s}}|x_i^{(j)}|^{m/d_0}\right)^{d_0/m}.
\]
Summing in $s$ and applying H\"older's inequality with exponent $m$ in the index $s$ gives
\[
\sum_{\mathbf i\in\Lambda}\prod_{j=1}^m|x^{(j)}_{i_j}|
\le C\prod_{j=1}^m
\left[\sum_s
\left(\sum_{i\in I_{j,s}}|x_i^{(j)}|^{m/d_0}\right)^{d_0}
\right]^{1/m}.
\]
Since $d_0>1$, the $\ell_{d_0}$-norm of a nonnegative sequence does not exceed its $\ell_1$-norm. Consequently,
\[
\left[\sum_s
\left(\sum_{i\in I_{j,s}}|x_i^{(j)}|^{m/d_0}\right)^{d_0}
\right]^{1/m}
\le
\left[\sum_s\sum_{i\in I_{j,s}}|x_i^{(j)}|^{m/d_0}\right]^{d_0/m}
\le\|x^{(j)}\|_{\ell_{m/d_0}}.
\]
It follows that
\[
\sum_{\mathbf i\in\Lambda}\prod_{j=1}^m|x^{(j)}_{i_j}|
\le C\prod_{j=1}^m\|x^{(j)}\|_{\ell_{m/d_0}},
\]
which is \eqref{eq:thinning-conclusion-product}.
}
\end{proof}

\begin{corollary}\label{cor:initial-plateau}
For every $m\ge2$,
\begin{equation}
\label{eq:initial-plateau}\mathrm{prod}(m,d)=m-1 \qquad\left(
1<d\le\frac{m}{m-1}\right).
\end{equation}

\end{corollary}

\begin{proof}
Let $e_{1},\ldots,e_{m}$ be the standard orthonormal basis of $\mathbb{R}^{m}$
and put $L_{j}=e_{j}^{\perp}$. Then
\[
\sum_{j=1}^{m}P_{L_{j}}=(m-1)I_{m},
\]
so $L_{1},\ldots,L_{m}$ form an equal-rank tight fusion frame with
$(D,k)=(m,m-1)$. Theorem~\ref{thm:fusion-frame-product} gives a support of
dimension
\[
d_{0}=\frac{m}{m-1}
\]
with product exponent $m-1$. The proof of that theorem supplies blocks
satisfying \eqref{eq:thinning-host-blocks}, so
Lemma~\ref{lem:dimension-thinning} gives the same exponent at every
$1<d<d_{0}$. At $d=d_{0}$, Theorem~\ref{thm:fusion-frame-product} applies
directly. Proposition~\ref{prop:general-section-obstruction}, with $r=m-1$,
gives the reverse inequality for every $d>1$.
\end{proof}

By the sufficient part of \cite[Theorem~1]{CFMWZ}, if $2k<D$, $k\nmid D$, and
\[
m\ge\left\lceil \frac Dk\right\rceil +2,
\]
then an equal-rank tight fusion frame with parameters $(m,k,D)$ exists; if
$k\mid D$, existence already follows from $m\ge D/k$. Their construction is formulated over complex Euclidean space. Realifying a
complex tight fusion frame replaces $(k,D)$ by $(2k,2D)$ and produces a real
tight fusion frame with the same ratio $D/k$.

\begin{theorem}\label{thm:curve-above-two}
Let $m\ge5$. Then
\begin{equation}
\label{eq:curve-above-two}
\mathrm{prod}(m,d)=\frac md
\qquad\left(2<d<\frac m2\right).
\end{equation}
\end{theorem}

\begin{proof}
Fix $2<d<m/2$ and $1\le q<m/d$. Choose a rational number
\[
d<d_{0}<\min\left\{  \frac m2,\frac mq\right\}  ,
\]
and write $d_{0}=D/k$ with positive integers $D,k$. Since $d_{0}>2$, we have
$2k<D$. If $k\mid D$, then $m>d_{0}=D/k$, and \cite[Theorem~1]{CFMWZ} gives
a complex tight fusion frame with parameters $(m,k,D)$. Realifying it gives a
real tight fusion frame with parameters $(m,2k,2D)$ and the same ratio
$D/k=d_0$. If $k\nmid D$, then $m>2d_{0}$ implies
\[
m\ge\lceil d_{0}\rceil+2,
\]
so \cite[Theorem~1]{CFMWZ} again gives an equal-rank tight fusion frame with
parameters $(m,k,D)$ over $\mathbb{C}$. Identify $\mathbb{C}^{D}$ with
$\mathbb{R}^{2D}$. Under this identification a complex $k$-dimensional
subspace becomes a real $2k$-dimensional subspace, and its complex orthogonal
projection becomes the corresponding real orthogonal projection. Thus the
tight-frame identity is preserved, giving a real tight fusion frame with
parameters $(m,2k,2D)$ and the same ratio $D/k=d_{0}$.
Theorem~\ref{thm:fusion-frame-product} therefore produces a support
$\Lambda_{0}$ of dimension $d_{0}$ satisfying the product estimate with
exponent
\[
\frac{m}{d_{0}}>q.
\]
By \eqref{eq:fusion-host-blocks}, this support satisfies the block hypothesis
of Lemma~\ref{lem:dimension-thinning}. Hence there is a support of dimension
exactly $d$ with the same product estimate. Since $q<m/d_{0}$, one has $\ell
_{q}\subset\ell_{m/d_{0}}$ and $\|x\|_{\ell_{m/d_0}}\le\|x\|_{\ell_q}$ for finitely
supported sequences. Hence the exponent $q$ is admissible. Letting $q\uparrow
m/d$ gives
\[
\mathrm{prod}(m,d)\ge\frac md.
\]
The reverse inequality is Proposition~\ref{prop:gamma-prod-upper}.
\end{proof}

\subsection{Dimensions below two}\label{sec:sparse-below-two}

For $m/(m-1)<d<2$, we combine the Vandermonde construction of Lemma~\ref{lem:Vandermonde-HBL} with a sparse
block selection. The summation step uses standard Lorentz-space machinery; see
Lorentz \cite{Lorentz1950} and, for multilinear H\"older inequalities on
Lorentz sequence spaces, Carando--Dimant--Sevilla-Peris
\cite{CarandoDimantSevilla}. The additional estimate required here is the
simultaneous rectangle control for the selected blocks.

\begin{definition}[\cite{Lorentz1950}]\label{def:lorentz}
Let $1<r<\infty$ and $x\in c_{00}$. If
$x^*=(x_n^*)_{n\ge1}$ is the nonincreasing rearrangement of $(|x_n|)$, set
\[
 \|x\|_{\ell_{r,1}}:=\sum_{n\ge1}n^{1/r-1}x_n^*.
\]
\end{definition}

For $1\le q<r$, the estimate
$x_n^*\le n^{-1/q}\|x\|_{\ell_q}$ gives
\begin{equation}\label{eq:lorentz-embedding}
 \|x\|_{\ell_{r,1}}
 \le C_{q,r}\|x\|_{\ell_q},\qquad
 C_{q,r}:=\sum_{n\ge1}n^{1/r-1-1/q}<\infty.
\end{equation}

The restricted-type estimate below is the standard dyadic Lorentz bound; compare \cite{Lorentz1950,CarandoDimantSevilla}.

\begin{lemma}\label{lem:restricted-lorentz}
Let $H\subset
\mathbb{N}^{m}$ be finite, let $0<\alpha<1$, and write
$r_\alpha:=1/\alpha$. Assume
that
\begin{equation}
\label{eq:restricted-rectangle}|H\cap(A_{1}\times\cdots\times A_{m})| \le
K\prod_{j=1}^{m} |A_{j}|^{\alpha}%
\end{equation}
for all finite sets $A_{j}\subset\mathbb{N}$. Then there is a constant
$C_{m,\alpha}$, independent of $H$ and of the ambient dimension, such that
\begin{equation}
\label{eq:restricted-lorentz-conclusion}\sum_{(i_{1},\ldots,i_{m})\in H}%
\prod_{j=1}^{m} |x^{(j)}_{i_{j}}| \le C_{m,\alpha}K\prod_{j=1}^{m}%
\|x^{(j)}\|_{\ell_{r_\alpha,1}}%
\end{equation}
for all $x^{(1)},\ldots,x^{(m)}\in c_{00}$.
\end{lemma}

\begin{proof}
Replacing each sequence by its modulus, assume that all entries are
nonnegative. For $k\in\mathbb{Z}$ and $j=1,\ldots,m$, set
\[
A_{j,k}:=\{i:2^{-(k+1)}<x^{(j)}_{i}\le2^{-k}\}.
\]
The sets $A_{j,k}$ are pairwise disjoint in $k$. Decomposing the sum according
to the $m$ dyadic levels and applying \eqref{eq:restricted-rectangle} gives
\begin{align}
\sum_{\mathbf{i}\in H}\prod_{j=1}^{m} x^{(j)}_{i_{j}}  &  \le\sum
_{k_{1},\ldots,k_{m}\in\mathbb{Z}} 2^{-k_{1}-\cdots-k_{m}} |H\cap(A_{1,k_{1}%
}\times\cdots\times A_{m,k_{m}})|\\
&  \le K\prod_{j=1}^{m} \left(  \sum_{k\in\mathbb{Z}}2^{-k}|A_{j,k}|^{\alpha
}\right).\label{eq:restricted-lorentz-factorized}
\end{align}
Fix $j$ and let
\[
 {
D_j:(0,\infty)\longrightarrow\mathbb N_0,\qquad
D_j(t):=\bigl|\{i:|x^{(j)}_i|>t\}\bigr|
}
\]
be the distribution function of $x^{(j)}$. Since
$A_{j,k}\subset\{i:|x^{(j)}_i|>2^{-(k+1)}\}$, we have
\[
 |A_{j,k}|^\alpha\le D_j(2^{-(k+1)})^\alpha.
\]
The distribution function $D_j$ is decreasing. Therefore, for
$2^{-(k+2)}<t\le2^{-(k+1)}$,
$D_j(t)\ge D_j(2^{-(k+1)})$, and hence
\[
 2^{-k}|A_{j,k}|^\alpha
 \le 4\int_{2^{-(k+2)}}^{2^{-(k+1)}}D_j(t)^\alpha\,dt.
\]
The intervals on the right are pairwise disjoint, so
\begin{equation}\label{eq:lorentz-distribution-bound}
 \sum_{k\in\mathbb{Z}}2^{-k}|A_{j,k}|^\alpha
 \le4\int_0^\infty D_j(t)^\alpha\,dt.
\end{equation}
The last integral is compared with the norm in Definition~\ref{def:lorentz}. Since $x^{(j)}$ is finitely supported and
$x^{(j)*}$ is decreasing, writing $x^{(j)*}_{n}=0$ beyond the support gives
\begin{align*}
 \int_0^\infty D_j(t)^\alpha\,dt
 &=\sum_{n\ge1} n^\alpha
   \bigl(x^{(j)*}_n-x^{(j)*}_{n+1}\bigr)\\
 &=\sum_{n\ge1}\bigl(n^\alpha-(n-1)^\alpha\bigr)x^{(j)*}_n.
\end{align*}
By the mean value theorem,
$n^\alpha-(n-1)^\alpha\le C_\alpha n^{\alpha-1}$ for $n\ge1$.
Consequently,
\[
 \int_0^\infty D_j(t)^\alpha\,dt
 \le C_\alpha\sum_{n\ge1}n^{\alpha-1}x^{(j)*}_n
 =C_\alpha\|x^{(j)}\|_{\ell_{r_\alpha,1}},
\]
because $\alpha=1/r_\alpha$. Combining this estimate with
\eqref{eq:lorentz-distribution-bound} for each $j$ in \eqref{eq:restricted-lorentz-factorized} proves
\eqref{eq:restricted-lorentz-conclusion}.
\end{proof}

\subsubsection*{Sparse Vandermonde blocks}

\begin{lemma}\label{lem:vandermonde-sparse-strip}
Let $m\ge3$
and
\[
\frac{m}{m-1}<d<2.
\]
Fix $\alpha$ with
\begin{equation}
\label{eq:alpha-strip-range}\frac dm<\alpha<\frac2m.
\end{equation}
Choose distinct integers $a_{1},\ldots,a_{m}$ and define
\[
\phi_j:\mathbb{Z}^2\to\mathbb{Z},\qquad \phi_j(u,v):=u+a_j v,
\]
and define
\[
\Theta:\mathbb{Z}^2\longrightarrow\mathbb{Z}^m,\qquad
\Theta(u,v):=(\phi_1(u,v),\ldots,\phi_m(u,v)).
\]
For $N\ge1$, put
\[
Q_{N}:=\{1,\ldots,N\}^{2}, \qquad V_{N}:=\Theta(Q_{N})\subset\mathbb{Z}^{m}.
\]
Then there are constants $c>0$ and $K_{\alpha}>0$ such that, for all
sufficiently large $N$, one can find $H_{N}\subset V_{N}$ satisfying
\begin{equation}
\label{eq:vander-strip-size}|H_{N}|\ge cN^{d}%
\end{equation}
and
\begin{equation}
\label{eq:vander-strip-restricted}|H_{N}\cap(A_{1}\times\cdots\times A_{m})|
\le K_{\alpha}\prod_{j=1}^{m}|A_{j}|^{\alpha}%
\end{equation}
for all finite sets $A_{j}\subset\mathbb{Z}$. Moreover, for every finite set
$E\subset(0,\infty)$, the same $H_{N}$ may be chosen so that, for each
$\varepsilon\in E$,
\begin{equation}
\label{eq:vander-strip-dimension-control}|H_{N}\cap(A_{1}\times\cdots\times
A_{m})| \le K_{\varepsilon}n^{d+\varepsilon}%
\end{equation}
whenever $|A_{j}|\le n$ for all $j$, where $K_{\varepsilon}$ is independent of
$N$.
\end{lemma}

\begin{proof}
{
Any two forms $\phi_j,\phi_k$ are linearly independent because $a_j\ne a_k$.
Hence
\[
(u,v)\longmapsto(\phi_j(u,v),\phi_k(u,v))
\]
is injective whenever $j\ne k$. In particular, $|V_N|=N^2$. There is a
constant $C_0\ge1$, depending only on $a_1,\ldots,a_m$, such that
\begin{equation}
\label{eq:vander-projection-size}
|\pi_j(V_N)|\le C_0N\qquad(j=1,\ldots,m).
\end{equation}

For each $N$, put $\rho_N:=N^{d-2}$ and consider the finite measure space
\[
\Omega_N:=2^{V_N},\qquad \mathcal F_N:=2^{\Omega_N},
\]
with
\[
\mu_N(\{H\})
=\rho_N^{|H|}(1-\rho_N)^{|V_N|-|H|}
\qquad(H\subseteq V_N).
\]
For a family of subsets of $V_N$, we estimate its $\mu_N$-measure;
all estimates below are obtained by summing this measure over explicitly
described exceptional families.

\proofstep{1}{Size}
Exactly as in Lemma~\ref{lem:dimension-thinning}, for every $\lambda>0$,
\[
\int_{\Omega_N}e^{-\lambda|H|}\,d\mu_N(H)
=\bigl(1+\rho_N(e^{-\lambda}-1)\bigr)^{N^2}
\le \exp\!\left(N^d(e^{-\lambda}-1)\right).
\]
Taking $\lambda=\log2$ shows that
\[
\mu_N\!\left(\left\{H:|H|\le\frac12N^d\right\}\right)
\le e^{-N^d/8}\longrightarrow0.
\]
Hence every $H$ outside this exceptional family satisfies
\eqref{eq:vander-strip-size}, after fixing for instance $c=1/2$.

\proofstep{2}{Restricted rectangles}
Fix nonempty sets $A_j\subset\pi_j(V_N)$ and put $b_j:=|A_j|$. Relabel the
coordinates so that
\[
1\le b_1\le b_2\le\cdots\le b_m,\qquad
P:=\prod_{j=1}^m b_j.
\]
A point of $V_N$ is determined by any two of its coordinates. Hence the
cylinder
\[
S(A_1,A_2):=\{z\in V_N:\pi_1(z)\in A_1,\ \pi_2(z)\in A_2\}
\]
has at most $b_1b_2$ points. Let
\[
t:=\lceil KP^\alpha\rceil.
\]
If $t>|S(A_1,A_2)|$, the corresponding exceptional family is empty.
Otherwise, if a subset $H\subseteq V_N$ contains at least $t$ points of this
cylinder, then it contains some fixed $t$-point subset of the cylinder.
Finite subadditivity therefore gives
\[
\mu_N\bigl(\{H:|H\cap S(A_1,A_2)|\ge t\}\bigr)
\le \binom{b_1b_2}{t}\rho_N^t
\le \left(\frac{e\rho_N b_1b_2}{t}\right)^t.
\]
Since $b_1b_2\le P^{2/m}$ and $P\le(C_0N)^m$,
\[
\frac{e\rho_N b_1b_2}{t}
\le \frac{eC_1}{K}N^{d-m\alpha}.
\]
Put $\delta:=m\alpha-d>0$. For all sufficiently large $N$ the last quantity
is at most $N^{-\delta/2}$, so
\begin{equation}
\label{eq:vander-cylinder-tail}
\mu_N\bigl(\{H:|H\cap S(A_1,A_2)|\ge t\}\bigr)
\le N^{-\delta t/2}.
\end{equation}

Because $d>m/(m-1)$ and $\alpha>d/m$, one has $\alpha(m-1)>1$. Also
\[
P\ge b_1b_2^{m-1},
\]
and therefore
\begin{equation}
\label{eq:vander-entropy-budget}
t\ge KP^\alpha\ge Kb_2,\qquad
b_1+b_2\le2b_2\le\frac{2t}{K}.
\end{equation}
For fixed cardinalities $(b_1,\ldots,b_m)$ and a fixed ordering of the
coordinates, only $A_1,A_2$ need to be counted: a violation by the full
rectangle forces the corresponding two-coordinate cylinder to contain at
least $t$ points of $H$. By \eqref{eq:vander-projection-size} and
\eqref{eq:vander-entropy-budget}, the number of choices is at most
\[
(C_0N)^{b_1+b_2}\le N^{3t/K}
\]
for large $N$. There are at most $m!$ coordinate orderings and at most
$(C_0N)^m$ cardinality vectors. Choose $K$ so large that
\[
\frac3K<\frac\delta4,\qquad \frac{\delta K}{4}>m+2.
\]
Finite subadditivity and \eqref{eq:vander-cylinder-tail} then show that the
$\mu_N$-measure of all subsets violating
\eqref{eq:vander-strip-restricted} is bounded by
\[
C_mN^mN^{-\delta K/4}\longrightarrow0.
\]

\proofstep{3}{Uniform dimension control}
Fix $\varepsilon>0$, let $|A_j|\le n$, and set
\[
R:=A_1\times\cdots\times A_m.
\]
Two-coordinate injectivity gives
\begin{equation}
\label{eq:vander-host-rectangle}
|V_N\cap R|\le n^2.
\end{equation}
Let $t_n:=\lceil K_\varepsilon n^{d+\varepsilon}\rceil$. If
$t_n>|V_N\cap R|$, no subset $H\subseteq V_N$ can satisfy $|H\cap R|\ge t_n$, so the exceptional family is empty; otherwise
\[
\mu_N\bigl(\{H:|H\cap R|\ge t_n\}\bigr)
\le
\binom{|V_N\cap R|}{t_n}\rho_N^{t_n}
\le
\left(\frac{eN^{d-2}n^2}{t_n}\right)^{t_n}.
\]
Since $t_n\ge K_\varepsilon n^{d+\varepsilon}$,
\begin{align}
\mu_N\bigl(\{H:|H\cap R|\ge t_n\}\bigr)
\le
\exp\!\left[-K_\varepsilon n^{d+\varepsilon}
\left(
\log\frac{K_\varepsilon}{e}
+(2-d)\log\frac Nn
+\varepsilon\log n
\right)\right].
\label{eq:vander-dimension-tail}
\end{align}
Only subsets of the coordinate projections of $V_N$ matter. Hence, for fixed
$n$, the number of relevant $m$-tuples is at most
\begin{equation}
\label{eq:vander-dimension-count}
\left(\sum_{a=0}^n\binom{\lceil C_0N\rceil}{a}\right)^m
\le
\exp\!\left(C_mn\log\frac{eC_0N}{n}\right).
\end{equation}
For $1\le n\le N/2$, put $L:=\log(N/n)$. Since $L\ge\log2$, there is a constant
$C'_m=C'_m(C_0)$ such that
\[
 C_m n\log\frac{eC_0N}{n}\le C'_m nL.
\]
Choose $K_\varepsilon$ so large that
\[
 K_\varepsilon(2-d)\ge4C'_m.
\]
Because $n^{d+\varepsilon}\ge n$, the contribution of the $(2-d)L$ term in
\eqref{eq:vander-dimension-tail} alone gives
\[
 C'_m nL-K_\varepsilon(2-d)n^{d+\varepsilon}L
 \le-\frac{K_\varepsilon(2-d)}2n^{d+\varepsilon}L.
\]
Thus the sum over $1\le n\le N/2$ is bounded by
\[
 \sum_{n\le N/2}
 \exp\left[-\frac{K_\varepsilon(2-d)}2
 n^{d+\varepsilon}\log\frac Nn\right],
\]
which tends to zero by splitting the sum into $1\le n\le\sqrt N$ and $\sqrt N<n\le N/2$, exactly as in the tail estimate in the proof of Lemma~\ref{lem:dimension-thinning}.

Now let $N/2<n\le C_0N$. In this range
$\log(eC_0N/n)=O_{C_0}(1)$ and
$\log(N/n)\ge-\log C_0$. Hence the bracket in
\eqref{eq:vander-dimension-tail} is bounded below by
\[
 \varepsilon\log n-C_{0,d,K_\varepsilon}
\]
for a constant independent of $N$ and $n$. For all sufficiently large $N$,
uniformly in $n>N/2$, this is at least $(\varepsilon/2)\log n$. The entropy
term is $O_{m,C_0}(n)$, whereas the negative term is then at most
\[
 -\frac{K_\varepsilon\varepsilon}{2}
 n^{d+\varepsilon}\log n.
\]
Since $d+\varepsilon>1$, the term $n^{d+\varepsilon}\log n$ dominates the entropy term $O_{m,C_0}(n)$ uniformly; multiplying by the
$O(N)$ possible values of $n$ still gives a quantity tending to zero. We have
therefore proved
\[
\mu_N(\mathcal D_{N,\varepsilon})\longrightarrow0,
\]
where $\mathcal D_{N,\varepsilon}$ denotes the family of subsets violating
\eqref{eq:vander-strip-dimension-control}. If $n>C_0N$, replace each $A_j$
by $A_j\cap\pi_j(V_N)$; its cardinality is at most $C_0N<n$, so the already
proved estimate at $\lceil C_0N\rceil$ is stronger than the required bound
with $n$.

Finally, let $E\subset(0,\infty)$ be finite. The exceptional family formed by
the size condition, the restricted-rectangle condition, and the families
$\mathcal D_{N,\varepsilon}$ for $\varepsilon\in E$ has $\mu_N$-measure
tending to zero. Its complement is therefore nonempty for every sufficiently
large $N$. Choosing any $H_N$ in this complement gives all the stated
properties simultaneously.
}
\end{proof}

\begin{theorem}\label{thm:remaining-strip}
Let $m\ge3$. If
\[
\frac{m}{m-1}<d<2,
\]
then
\begin{equation}
\label{eq:remaining-strip-value}\mathrm{prod}(m,d)=\frac md.
\end{equation}

\end{theorem}

\begin{proof}
The upper bound is Proposition~\ref{prop:gamma-prod-upper}. Fix $q$ with $1\le
q<m/d$. Since $d<2$, one can choose $\alpha$ such that
\[
\frac dm<\alpha<\min\left\{  \frac1q,\frac2m\right\}.
\]
Put $r_\alpha:=1/\alpha$. Then $q<r_\alpha$. For every $s\ge1$, apply
Lemma~\ref{lem:vandermonde-sparse-strip} with the finite set
\[
\mathcal E_s:=\{1,1/2,\ldots,1/s\}
\]
and choose $N_{s}\uparrow\infty$ and a corresponding block $H_{s}$. Reindex
the coordinate projections of the blocks into pairwise disjoint sets
$I_{j,s}\subset\mathbb{N}$, and let $\widetilde H_{s}$ denote the reindexed
block. Set
\[
\Lambda:=\bigcup_{s=1}^{\infty}\widetilde H_{s}.
\]
For the dimension, \eqref{eq:vander-strip-size} together with the projection
bound \eqref{eq:vander-projection-size} gives $\dim(\Lambda)\ge d$. Fix an
integer $R\ge1$ and sets $A_{j}\subset\mathbb{N}$
with $|A_{j}|\le n$. Define
\[
n_{s}:=\max_{1\le j\le m}|A_{j}\cap I_{j,s}|.
\]
As in \eqref{eq:thinning-budget}, disjointness of the coordinate blocks gives
\[
\sum_{s} n_{s}\le mn.
\]
Set
\[
 M_R:=\sum_{s<R}|\widetilde H_s|,\qquad K_R:=K_{1/R}.
\]
For $s\ge R$, property \eqref{eq:vander-strip-dimension-control} with
$\varepsilon=1/R$ yields
\[
|\widetilde H_{s}\cap(A_{1}\times\cdots\times A_{m})| \le K_{R} n_{s}%
^{d+1/R}.
\]
Therefore
\[
|\Lambda\cap(A_{1}\times\cdots\times A_{m})| \le M_{R}+K_{R}(mn)^{d+1/R}.
\]
It follows that $\dim(\Lambda)\le d+1/R$. Letting $R\to\infty$ gives
$\dim(\Lambda)=d$. For the product estimate,
\eqref{eq:vander-strip-restricted} and Lemma~\ref{lem:restricted-lorentz}
give, on each block,
\[
\sum_{\mathbf{i}\in\widetilde H_{s}}\prod_{j=1}^{m}|x^{(j)}_{i_{j}}| \le
C_{m,\alpha}K_{\alpha}\prod_{j=1}^{m}\|x^{(j)}|_{I_{j,s}}\|_{\ell_{r_\alpha,1}}.
\]
Since $q<r_\alpha$, \eqref{eq:lorentz-embedding} gives
\[
\|x^{(j)}|_{I_{j,s}}\|_{\ell_{r_\alpha,1}} \le C_{q,r_\alpha}\|x^{(j)}|_{I_{j,s}}\|_{\ell_q}.
\]
Put $a_{j,s}:=\|x^{(j)}|_{I_{j,s}}\|_{\ell_q}$. Summing over $s$ and applying
H\"older's inequality to the block index with exponent $m$ gives
\[
\sum_{s}\prod_{j=1}^{m} a_{j,s} \le\prod_{j=1}^{m}\left(  \sum_{s} a_{j,s}%
^{m}\right)  ^{1/m}.
\]
Because $q<m/d<m$, one has
\[
\left(  \sum_{s} a_{j,s}^{m}\right)  ^{1/m} \le\left(  \sum_{s} a_{j,s}%
^{q}\right)  ^{1/q} \le\|x^{(j)}\|_{\ell_q}.
\]
Thus $\Lambda$ satisfies the product estimate with exponent $q$. Since every
$q<m/d$ is admissible,
\[
\mathrm{prod}(m,d)\ge\frac md.
\]
Proposition~\ref{prop:gamma-prod-upper} gives the reverse inequality, and
hence \eqref{eq:remaining-strip-value}.
\end{proof}

\subsection{The high-dimensional range}\label{sec:sparse-high-d}

If
$E\subset[N]^m$, $1\le j\le m$, and $r\in[N]$, write
\begin{equation}\label{eq:def-Delta}
 E_{j,r}:=\{\mathbf i=(i_1,\ldots,i_m)\in E:i_j=r\},
 \qquad
 \Delta(E):=\max_{j,r}|E_{j,r}|.
\end{equation}
Thus $\Delta(E)$ is the largest number of points of $E$ lying in a fiber
obtained by fixing one coordinate.

The high-dimensional range is obtained by Bernoulli thinning of $[N]^m$ at density $N^{d-m}$. Such random constructions originate in Blei--K\"orner \cite{BleiKorner}; see also \cite{BleiPeresSchmerl,BleiGao} for fractional products and prescribed dimensions, and \cite[Section~4.2 and Lemma~4.3]{BayartJEMS} for a variant retaining analytic control after thinning. The finite-block statement below combines the three estimates needed later: cardinality, one-coordinate fibers, and rectangle counts.

\begin{lemma}\label{lem:sparse-random-blocks-general}
Let $m\ge2$ and $1<d<m$. There is a constant $K_0=K_0(m,d)$ such that, for
all sufficiently large $N$, one can find $H_N\subset[N]^m$ satisfying
\begin{align}
 \frac12N^d\le |H_N|&\le2N^d,\label{eq:random-block-size-general}\\
 |H_N\cap(A_1\times\cdots\times A_m)|&\le K_0 n^d
 \label{eq:random-block-dimension-general}
\end{align}
whenever $|A_j|\le n\le N$ for every $j$, and
\begin{equation}\label{eq:random-block-fibers}
 \Delta(H_N)\le K_0N^{d-1}.
\end{equation}
If, in addition, $m\ge3$, $1<d<m-1$, and
\[
 \max\left\{\frac dm,\frac12\right\}<\alpha<1,
\]
then the same realization may be chosen so that
\begin{equation}\label{eq:random-block-rectangle-general}
 |H_N\cap(A_1\times\cdots\times A_m)|
 \le K_\alpha\prod_{j=1}^m|A_j|^\alpha
\end{equation}
for all $A_j\subset[N]$, where $K_\alpha$ depends only on $m,d,\alpha$.
\end{lemma}

\begin{proof}
{
For each $N$, set $\rho_N:=N^{d-m}$ and consider the finite measure space
\[
\Omega_N:=2^{[N]^m},\qquad \mathcal F_N:=2^{\Omega_N},
\]
with
\[
\mu_N(\{H\})
=\rho_N^{|H|}(1-\rho_N)^{N^m-|H|}
\qquad(H\subseteq[N]^m).
\]

\proofstep{1}{Cardinality}
For every $\lambda>0$,
\[
\int_{\Omega_N}e^{-\lambda|H|}\,d\mu_N(H)
=
\bigl(1+\rho_N(e^{-\lambda}-1)\bigr)^{N^m}
\le
\exp\!\left(N^d(e^{-\lambda}-1)\right),
\]
whereas
\[
\int_{\Omega_N}e^{\lambda|H|}\,d\mu_N(H)
\le
\exp\!\left(N^d(e^\lambda-1)\right).
\]
Taking $\lambda=\log2$ in the first estimate and $\lambda=\log2$ in the
second gives
\[
\mu_N\!\left(\left\{H:|H|<\frac12N^d\right\}\right)\le e^{-N^d/8},
\]
and
\[
\mu_N\!\left(\left\{H:|H|>2N^d\right\}\right)
\le
\exp\!\left((1-2\log2)N^d\right).
\]
Thus the family violating \eqref{eq:random-block-size-general} has measure
tending to zero.

\proofstep{2}{Coordinate fibers}
Fix a coordinate fiber $F\subset[N]^m$. Then $|F|=N^{m-1}$. We first work with an auxiliary constant $K>e$ and put
$t:=\lceil KN^{d-1}\rceil$. If $t>|F|$, there is nothing to prove.
Otherwise, if $|H\cap F|\ge t$, then $H$ contains some $t$-point subset of
$F$. Hence
\[
\mu_N\bigl(\{H:|H\cap F|\ge t\}\bigr)
\le
\binom{N^{m-1}}{t}\rho_N^t
\le
\left(\frac{eN^{d-1}}{t}\right)^t
\le
\left(\frac eK\right)^t.
\]
There are exactly $mN$ coordinate fibers. Since $d>1$, finite
subadditivity gives
\[
mN\left(\frac eK\right)^{\lceil KN^{d-1}\rceil}\longrightarrow0.
\]
Moreover, outside this exceptional set,
\[
 |H\cap F|<\lceil KN^{d-1}\rceil
 \le KN^{d-1}+1
 \le (K+1)N^{d-1}.
\]
We henceforth enlarge constants, when necessary, to absorb the analogous
rounding terms arising from integer thresholds. Thus, taking $K_0\ge K+1$,
\eqref{eq:random-block-fibers} holds outside a set of $\mu_N$-measure tending
to zero.

\proofstep{3}{Uniform box estimate}
Fix $1\le n\le N$ and sets $A_j\subset[N]$ with $|A_j|\le n$. Enlarging them
if needed, assume $|A_j|=n$. Let
\[
R:=A_1\times\cdots\times A_m,\qquad t_n:=\lceil K_0n^d\rceil.
\]
If $t_n>|R|=n^m$, the exceptional family is empty. Otherwise
\[
\mu_N\bigl(\{H:|H\cap R|\ge t_n\}\bigr)
\le
\binom{n^m}{t_n}\rho_N^{t_n}
\le
\left[
\frac e{K_0}\left(\frac nN\right)^{m-d}
\right]^{t_n}.
\]
There are at most $(eN/n)^{mn}$ possible $m$-tuples of coordinate sets.
Write
\[
L:=\log(N/n),\qquad a:=\log(K_0/e).
\]
Since $t_n\ge K_0n^d$, finite subadditivity bounds the measure of all
violations at this scale by
\[
\exp\!\left(
mn(1+L)-K_0n^d\bigl(a+(m-d)L\bigr)
\right).
\]
In the nonempty range $K_0n^d\le n^m$, one has
$n^{m-d}\ge K_0$. Because $d>1$, $K_0=K_0(m,d)$ can be chosen so large that
\[
K_0n^{d-1}a\ge2m,\qquad
K_0(m-d)n^{d-1}\ge2m
\]
throughout this range. Consequently the exponent in the exponential bound just obtained,
\[
mn(1+L)-K_0n^d\bigl(a+(m-d)L\bigr),
\]
is at most
\[
-mn(1+L).
\]
Therefore the total measure of all subsets violating
\eqref{eq:random-block-dimension-general} is bounded by
\[
S_N:=\sum_{n=1}^N e^{-mn}\left(\frac nN\right)^{mn}.
\]
The elementary splitting
\[
[1,N]=[1,N^{1/2}]\cup(N^{1/2},N/2]\cup(N/2,N]
\]
gives respectively
\[
O(N^{-m/2}),\qquad
N e^{-m(1+\log2)N^{1/2}},\qquad
Ne^{-mN/2},
\]
so $S_N\to0$.

\proofstep{4}{Restricted rectangles}
Assume now that $m\ge3$, $1<d<m-1$, and
\[
\max\left\{\frac dm,\frac12\right\}<\alpha<1.
\]
Set
\[
\delta:=m\alpha-d>0,\qquad
\beta:=\frac{1-\alpha}{\alpha},\qquad
\eta:=2-\frac1\alpha>0.
\]
Choose an integer $L_0$ so large that
\[
L_0(m-1-d)>m.
\]
Fix values in any $m-1$ coordinate positions and let $F$ be the corresponding
\emph{completion fiber}, obtained by allowing only the remaining coordinate to
vary. Thus $|F|=N$. This is different from the one-coordinate fibers measured by $\Delta(H)$ in \eqref{eq:def-Delta}: here $m-1$ coordinates are fixed. The family
\[
\{H:|H\cap F|\ge L_0\}
\]
has measure at most
\[
\binom NL_0\rho_N^{L_0}
\le
\left(\frac{eN^{d-m+1}}{L_0}\right)^{L_0}.
\]
There are $mN^{m-1}$ such fibers. Define the good event
\[
 \mathcal E_N:=\{H\subset[N]^m:\ |H\cap F|<L_0
 \text{ for every completion fiber }F\}.
\]
Its complement is the union of the families for which some $(m-1)$-codegree is at least $L_0$, and hence
\[
\mu_N(\mathcal E_N^c)
\le
mN^{m-1}
\left(\frac{eN^{d-m+1}}{L_0}\right)^{L_0}.
\]
The exponent of $N$ on the right is
\[
 (m-1)+L_0(d-m+1)
 =(m-1)-L_0(m-1-d)<-1,
\]
by the choice of $L_0$. Therefore $\mu_N(\mathcal E_N^c)\to0$.

Fix now a nonempty rectangle with side cardinalities $a_j$, relabeled so that
$a_{\max}:=a_m=\max_j a_j$, and put
\[
Q:=\prod_{j<m}a_j.
\]
For $H\in\mathcal E_N$, the rectangle contains at most $L_0Q$ points of $H$.
Thus only
\[
L_0Q>K_\alpha(Qa_{\max})^\alpha
\]
needs further consideration. In this range
\[
a_{\max}\le\left(\frac{L_0}{K_\alpha}\right)^{1/\alpha}Q^\beta.
\]
Let
\[
k:=\lceil K_\alpha(Qa_{\max})^\alpha\rceil.
\]
Then
\[
\frac{k}{a_{\max}}
\ge
K_\alpha\left(\frac{K_\alpha}{L_0}\right)^\beta Q^\eta.
\]
Choose $K_\alpha$ so large that
\[
K_\alpha\left(\frac{K_\alpha}{L_0}\right)^\beta>\frac{8m}{\delta},
\qquad
\frac{\delta K_\alpha}{8}>m+1.
\]

For the fixed rectangle $R$, if $k>|R|$ the exceptional family is empty.
Otherwise
\[
\mu_N\bigl(\{H:|H\cap R|\ge k\}\bigr)
\le
\binom{Qa_{\max}}{k}\rho_N^k
\le
\left(\frac{eN^{d-m}Qa_{\max}}{k}\right)^k.
\]
Since $Qa_{\max}\le N^m$ and $\delta=m\alpha-d$,
\[
\frac{eN^{d-m}Qa_{\max}}{k}
\le
\frac e{K_\alpha}N^{d-m}(Qa_{\max})^{1-\alpha}
\le
\frac e{K_\alpha}N^{-\delta}.
\]
For large $N$ this is at most $N^{-\delta/2}$, so the exceptional family has
measure at most $N^{-\delta k/2}$. Moreover, in the nontrivial range,
\[
\sum_j a_j\le m a_{\max}
\le
\frac{mk}{K_\alpha(K_\alpha/L_0)^\beta}.
\]
For prescribed cardinalities, the number of rectangles is therefore at most
\[
N^{mk/[K_\alpha(K_\alpha/L_0)^\beta]}.
\]
By the choice of $K_\alpha$, multiplying this count by
$N^{-\delta k/2}$ gives at most $N^{-\delta k/4}$. Summing over the at most
$N^m$ cardinality vectors shows that the union of the families violating
\eqref{eq:random-block-rectangle-general} has $\mu_N$-measure tending to zero.

The union of the exceptional families from Steps 1--4 has measure tending to
zero. Its complement is therefore nonempty for every sufficiently large
$N$. Any $H_N$ in this complement satisfies all the stated properties
simultaneously.
}
\end{proof}

\begin{theorem}
\label{thm:high-dimensional-range}
Let $m\ge3$. For every $1<d<m-1$,
\begin{equation}
\label{eq:sparse-general-lower}\mathrm{prod}(m,d)\ge\min\left\{
2,\frac md\right\}.
\end{equation}
Consequently, if
\[
\frac m2\le d\le m-1,
\]
then
\begin{equation}
\label{eq:high-dimensional-curve}\mathrm{prod}(m,d)=\frac md.
\end{equation}
If $m-1<d\le m$, then
\begin{equation}
\label{eq:top-plateau}\mathrm{prod}(m,d)=1.
\end{equation}

\end{theorem}

\begin{proof}
Fix first $1<d<m-1$ and
\[
1\le q<\min\left\{2,\frac md\right\}.
\]
Then one can choose
\[
\max\left\{  \frac dm,\frac12\right\}  <\alpha<\frac1q.
\]
Choose integers $N_{s}\uparrow\infty$. For each $s$, apply Lemma~\ref{lem:sparse-random-blocks-general} and choose a
realization $H_s\subset[N_s]^m$ for which
\eqref{eq:random-block-size-general}, \eqref{eq:random-block-dimension-general}
and \eqref{eq:random-block-rectangle-general} hold. For each coordinate $j$, reindex the copy of $[N_{s}]$
used by the $s$-th block onto a finite set $I_{j,s}\subset\mathbb{N}$,
choosing the sets $I_{j,s}$ pairwise disjoint in $s$. Let $\widetilde
H_{s}\subset I_{1,s}\times\cdots\times I_{m,s}$ be the reindexed copy of
$H_{s}$, and define
\[
\Lambda:=\bigcup_{s=1}^{\infty}\widetilde H_{s}.
\]
For the combinatorial dimension, the estimate
$|\widetilde H_s|\ge N_s^d/2$, together with $|I_{j,s}|=N_s$, gives
$\dim(\Lambda)\ge d$. Conversely, fix finite sets
$A_j\subset\mathbb{N}$ with $|A_j|\le n$ and put
\[
 n_s:=\max_{1\le j\le m}|A_j\cap I_{j,s}|.
\]
For each fixed coordinate $j$, the sets $I_{j,s}$ are disjoint in $s$, so
$\sum_s|A_j\cap I_{j,s}|\le |A_j|\le n$. Since
$n_s\le\sum_{j=1}^m|A_j\cap I_{j,s}|$, summing first in $s$ and then in
$j$ gives the block budget
\[
 \sum_s n_s\le mn.
\]
By \eqref{eq:random-block-dimension-general},
\[
 |\Lambda\cap(A_1\times\cdots\times A_m)|
 \le K_0\sum_s n_s^d
 \le K_0\left(\sum_s n_s\right)^d
 \le K_0m^dn^d.
\]
Hence $\Lambda(n)\le K_0m^dn^d$ for all $n$ and $\dim(\Lambda)=d$; in
particular, $\Lambda$ has exact combinatorial dimension $d$.
Set $r_\alpha:=1/\alpha$. Since $\alpha<1/q$, we have $q<r_\alpha$.
On each block, \eqref{eq:random-block-rectangle-general} is a restricted-type
estimate with Lorentz exponent $r_\alpha$. Let $C_{q,r_\alpha}$ be the constant in \eqref{eq:lorentz-embedding} and set
\[
 C_*:=C_{m,\alpha}K_\alpha C_{q,r_\alpha}^{m}.
\]
Lemma~\ref{lem:restricted-lorentz} then gives, for every $s$,
\[
\sum_{\mathbf{i}\in\widetilde H_{s}}\prod_{j=1}^{m} |x^{(j)}_{i_{j}}|
\le C_*\prod_{j=1}^{m}\|x^{(j)}|_{I_{j,s}}\|_{\ell_q}.
\]
Define
\[
a_{j,s}:=\|x^{(j)}|_{I_{j,s}}\|_{\ell_q}.
\]
Summing over the blocks and applying H\"older's inequality in the index $s$
with exponent $m$ gives
\[
\sum_{\mathbf{i}\in\Lambda}\prod_{j=1}^{m}|x^{(j)}_{i_{j}}| \le C_*\prod_{j=1}^{m} \left(  \sum_{s} a_{j,s}^{m}\right)  ^{1/m}.
\]
Since $q<2\le m$,
\[
\left(  \sum_{s} a_{j,s}^{m}\right)  ^{1/m} \le\left(  \sum_{s} a_{j,s}%
^{q}\right)  ^{1/q} \le\|x^{(j)}\|_{\ell_q}.
\]
Hence the global product estimate holds with exponent $q$. Since every
$q<\min\{2,m/d\}$ can be chosen at the start of the proof,
\eqref{eq:sparse-general-lower} follows. If $m/2\le d<m-1$, then $m/d\le2$,
and Proposition~\ref{prop:gamma-prod-upper} gives
\eqref{eq:high-dimensional-curve}. The endpoint $d=m-1$ is
Theorem~\ref{thm:bayart-problem52}. If $m-1<d<m$,
Proposition~\ref{prop:general-section-obstruction} with $r=1$ gives the upper
bound $1$. For the reverse inequality, apply
Lemma~\ref{lem:dimension-thinning} to the full support $\Lambda_{0}%
=\mathbb{N}^{m}$, which has dimension $m$, satisfies the product estimate in
$\ell_{1}$ with constant one, and has admissible blocks $B_{N}=[N]^{m}$. This
gives a support of dimension $d$ with product exponent $1$. At $d=m$, the full
support itself gives the same exponent. Hence \eqref{eq:top-plateau} follows.
\end{proof}

\subsection{Solution of Problem~5.2}\label{sec:problem52-proof}

\begin{proof}[Proof of Theorem~\ref{thm:A}]
We verify the formula regime by regime, always matching the available lower
construction with the corresponding upper obstruction from
Proposition~\ref{prop:gamma-prod-upper} and the coordinate-section obstruction in Proposition~\ref{prop:general-section-obstruction}.

\proofstep{1}{The bilinear case}
If $m=2$, Theorem~\ref{thm:bayart-problem52} gives
$\mathrm{prod}(2,1)=2$, while Corollary~\ref{cor:initial-plateau} gives
$\mathrm{prod}(2,d)=1$ for $1<d\le2$. These are precisely the two values in
\eqref{eq:main-product-formula}. Assume henceforth that $m\ge3$.

\proofstep{2}{The endpoint $d=1$ and the initial plateau}
At $d=1$, Theorem~\ref{thm:bayart-problem52} gives
$\mathrm{prod}(m,1)=m$, which agrees with both the dimensional upper bound
$m/d=m$ and \eqref{eq:main-product-formula}. For
\[
 1<d\le \frac{m}{m-1},
\]
Corollary~\ref{cor:initial-plateau} gives the lower value $m-1$. In this
interval $\lceil d\rceil=2$, so the section obstruction is
$m-\lceil d\rceil+1=m-1$; hence the lower construction and the upper bound
coincide.

\proofstep{3}{The range between the initial plateau and dimension two}
For
\[
 \frac{m}{m-1}<d<2,
\]
Theorem~\ref{thm:remaining-strip} gives $\mathrm{prod}(m,d)=m/d$. Here the
matching upper bound is Proposition~\ref{prop:gamma-prod-upper}, namely
$\mathrm{prod}(m,d)\le m/d$. At $d=2$, the integer-dimensional result,
Theorem~\ref{thm:bayart-problem52}, gives the same value $m/2$, again matching
the dimensional obstruction.

\proofstep{4}{The remaining dimensions}
{
When $m\ge5$ and $2<d<m/2$, Theorem~\ref{thm:curve-above-two} gives the lower value $m/d$,
and Proposition~\ref{prop:gamma-prod-upper} gives the reverse inequality.
The remaining dimensions up to $m-1$ are covered without overlap as follows.
If $m=3$, no new interval remains after $d=2=m-1$.
If $m=4$, Theorem~\ref{thm:high-dimensional-range} gives
$\mathrm{prod}(4,d)=4/d$ for $2<d\le3$.
If $m\ge5$, the same theorem gives
$\mathrm{prod}(m,d)=m/d$ for $m/2\le d\le m-1$.
In each case Proposition~\ref{prop:gamma-prod-upper} supplies the matching upper bound.
Finally, Theorem~\ref{thm:high-dimensional-range} gives
\[
 \mathrm{prod}(m,d)=1 \qquad(m-1<d\le m).
\]
}
In this last interval $\lceil d\rceil=m$, so the section obstruction equals
$m-\lceil d\rceil+1=1$ and is sharp.

The displayed regimes cover $[1,m]$. In each one the lower construction
matches the relevant upper obstruction, and the resulting values are exactly
those in \eqref{eq:main-product-formula}--\eqref{eq:main-product-piecewise}.
\end{proof}

\begingroup

\section{Local sparsity and projection entropy}\label{mult:sec:local}

For a nonempty finite set $F\subset[n]^m$ and $1\le j\le m$, let
\[
 \pi_j(F):=\{i_j:(i_1,\ldots,i_m)\in F\}
\]
be its $j$-th coordinate projection. We write
\begin{equation}\label{mult:eq:local-parameters}
 M(F):=|F|,\qquad
 v_j(F):=|\pi_j(F)|,\qquad
 P(F):=\prod_{j=1}^m v_j(F).
\end{equation}
The entropy associated with the coordinate projections is
\begin{equation}\label{mult:eq:entropy-def}
 H_n(F):=
 1+\sum_{j=1}^m v_j(F)\log\frac{en}{v_j(F)}.
\end{equation}
The logarithm is natural. Notice that $1\le v_j(F)\le M(F)$ and
$v_j(F)\le n$.

\begin{definition}[Local sparsity]\label{mult:def:local-sparsity}
Let $p>2$ and $\rho>0$. A set $\Lambda\subset[n]^m$ is
\emph{$(p,\rho)$-locally sparse} if
\begin{equation}\label{mult:eq:local-sparsity}
 M(F)H_n(F)
 \le n^\rho P(F)^{2/p}
\end{equation}
for every nonempty subset $F\subset\Lambda$.
\end{definition}

The following proposition is the main combinatorial step.

\begin{proposition}\label{mult:prop:random-support}
Let $m\ge3$, $d\in[1,m-2]$, and
\begin{equation}\label{mult:eq:p-range-local}
 2<p<\min\left\{m-\lceil d\rceil+1,\frac{2m}{d+1}\right\}.
\end{equation}
For every $\rho>0$ there are constants $c>0$ and $n_0\in\N$ such that, for
all $n\ge n_0$, there exists a $(p,\rho)$-locally sparse set
$\Lambda_n\subset[n]^m$ satisfying
\begin{equation}\label{mult:eq:support-size}
 |\Lambda_n|\ge c n^d.
\end{equation}
\end{proposition}

\begin{proof}
Fix $m,d,p$ as in the statement and $\rho>0$. Set
\begin{equation}\label{mult:eq:r0-theta}
 r_0:=m-\lceil d\rceil+1,
 \qquad
 \theta:=d-\lceil d\rceil+1\in(0,1].
\end{equation}
Thus
\begin{equation}\label{mult:eq:d-m-theta}
 d-m=\theta-r_0.
\end{equation}
The strict inequalities in \eqref{mult:eq:p-range-local} give
\begin{equation}\label{mult:eq:gaps}
 \delta_1:=\frac{2m}{p}-d-1>0,
 \qquad
 \delta_2:=r_0-p>0.
\end{equation}

Choose
\begin{equation}\label{mult:eq:kappa-choice}
 \kappa:=\frac{p\rho}{4}
\end{equation}
and then choose $\tau>0$ so small that
\begin{equation}\label{mult:eq:tau-choice}
 \tau<\frac{\rho}{4}
 \qquad\text{and}\qquad
 m\tau<\frac{\kappa}{4}.
\end{equation}

Let $\Lambda\subset[n]^m$ be obtained by retaining every point of $[n]^m$
independently with probability
\begin{equation}\label{mult:eq:bernoulli-density}
 q:=n^{d-m}.
\end{equation}
Then $\E|\Lambda|=n^d$. We show that the probability that $\Lambda$ fails
\eqref{mult:eq:local-sparsity} tends to zero.

Suppose that $F\subset\Lambda$ is nonempty and violates
\eqref{mult:eq:local-sparsity}. To shorten notation write
\[
 M:=M(F),\qquad v_j:=v_j(F),\qquad
 P:=P(F),\qquad H:=H_n(F),\qquad
 V:=\max_{1\le j\le m}v_j.
\]
Since every projection is nonempty,
\begin{equation}\label{mult:eq:basic-MV}
 1\le v_j\le M,
 \qquad
 V\le M,
 \qquad
 P\le V^m.
\end{equation}
Moreover,
\begin{equation}\label{mult:eq:H-crude}
 H\le C_m V\log n
\end{equation}
for all sufficiently large $n$.

For fixed integers $v_1,\ldots,v_m,M$, the expected number of sets
$F\subset\Lambda$ having these parameters is at most
\begin{align}
 \prod_{j=1}^m\binom{n}{v_j}\binom{P}{M}q^M
 &\le
 \exp\left(\sum_{j=1}^m v_j\log\frac{en}{v_j}\right)
 \left(\frac{eqP}{M}\right)^M.\label{mult:eq:first-moment-master}
\end{align}
We estimate \eqref{mult:eq:first-moment-master} in two regimes.

\medskip
\noindent\emph{Regime I: $M\ge Vn^\tau$.}
Since $F$ violates \eqref{mult:eq:local-sparsity},
\[
 MH>n^\rho P^{2/p}.
\]
Using \eqref{mult:eq:H-crude},
\[
 M>\frac{n^\rho P^{2/p}}{H}
 \ge \frac{n^\rho}{C_mV\log n}P^{2/p}.
\]
Choose $n_0$ so large that
\[
 C_m\log n\le n^{\rho/4}
 \qquad(n\ge n_0).
\]
Then
\begin{equation}\label{mult:eq:M-lower-regime1}
 M\ge n^{3\rho/4}\frac{P^{2/p}}{V}.
\end{equation}
Therefore
\begin{align}
 \frac{M}{qP}
 &\ge
 n^{3\rho/4+m-d}\frac{P^{2/p-1}}{V}\notag\\
 &\ge
 n^{3\rho/4+m-d}V^{2m/p-m-1}\notag\\
 &\ge
 n^{3\rho/4+2m/p-d-1}
 =n^{3\rho/4+\delta_1}.
 \label{mult:eq:ratio-regime1}
\end{align}
Here the second inequality uses $2/p-1<0$ together with $P\le V^m$,
and the third uses $V\le n$ and $2m/p-m-1<0$.

It follows from \eqref{mult:eq:ratio-regime1} that
\[
 \left(\frac{eqP}{M}\right)^M
 \le \exp(-cM\log n)
\]
for some $c>0$ depending only on $m,d,p,\rho$. On the other hand,
\eqref{mult:eq:H-crude} and $M\ge Vn^\tau$ give
\[
 \sum_{j=1}^m v_j\log\frac{en}{v_j}
 \le C_mMn^{-\tau}\log n
 =o(M\log n).
\]
Thus, for all sufficiently large $n$, the contribution of any fixed
parameter class in Regime~I is at most
\begin{equation}\label{mult:eq:regime1-decay}
 n^{-c_1M}
\end{equation}
for some $c_1>0$.

\medskip
\noindent\emph{Regime II: $V\le M<Vn^\tau$.}
Again using the failure of \eqref{mult:eq:local-sparsity},
\eqref{mult:eq:H-crude}, and $M<Vn^\tau$, we have
\[
 n^\rho P^{2/p}<MH
 < C_mV^2n^\tau\log n.
\]
Hence
\begin{align*}
 P
 &<\left(C_m n^{\tau-\rho}\log n\right)^{p/2}V^p\\
 &=C_m^{p/2}(\log n)^{p/2}
   n^{-p(\rho-\tau)/2}V^p.
\end{align*}
Because $\tau<\rho/4$,
\[
 \frac{p(\rho-\tau)}2>\frac{3p\rho}{8}.
\]
With $\kappa=p\rho/4$, the difference between the two powers is
\[
 \frac{3p\rho}{8}-\kappa
 =\frac{p\rho}{8}>0.
\]
Therefore, after increasing $n_0$ so that
\[
 C_m^{p/2}(\log n)^{p/2}\le n^{p\rho/8},
\]
we obtain
\begin{equation}\label{mult:eq:P-small}
 P\le n^{-\kappa}V^p.
\end{equation}

We first note that a violating set in this regime must satisfy
$V\ge n^\tau$. Indeed, if $V<n^\tau$, then
\eqref{mult:eq:P-small} and \eqref{mult:eq:tau-choice} give
\[
 1\le P\le n^{-\kappa+p\tau}<1,
\]
a contradiction.

Define the set of large projections by
\begin{equation}\label{mult:eq:J-def}
 J:=\{j: v_j\ge Vn^{-\tau}\},
 \qquad r:=|J|.
\end{equation}
We claim that
\begin{equation}\label{mult:eq:r-le-p}
 r\le p.
\end{equation}
If $r>p$, then \eqref{mult:eq:P-small} gives
\[
 (Vn^{-\tau})^r\le P\le n^{-\kappa}V^p,
\]
so
\[
 V^{r-p}\le n^{r\tau-\kappa}<1
\]
by \eqref{mult:eq:tau-choice}, whereas $V\ge1$. This proves
\eqref{mult:eq:r-le-p}. Notice also that $r\in\mathbb N$ and
$r_0=m-\lceil d\rceil+1\in\mathbb N$; hence
\begin{equation}\label{mult:eq:r-integer-gap}
 r\le p<r_0 \quad\Longrightarrow\quad r\le r_0-1.
\end{equation}

Write $V=n^a$, where $a\in[0,1]$. For $j\in J$ we have
$v_j\le V$ and $v_j\ge Vn^{-\tau}$; hence
\[
 v_j\log\frac{en}{v_j}
 \le V\bigl((1-a+\tau)\log n+1\bigr).
\]
For $j\notin J$, the monotonicity of
$t\mapsto t\log(en/t)$ on $[1,n]$ gives
\[
 v_j\log\frac{en}{v_j}
 \le Vn^{-\tau}\bigl((1-a+\tau)\log n+1\bigr).
\]
Since $M\ge V$, we obtain
\begin{equation}\label{mult:eq:entropy-regime2}
 \sum_{j=1}^m v_j\log\frac{en}{v_j}
 \le \bigl(r(1-a+\tau)+o(1)\bigr)M\log n,
\end{equation}
where the $o(1)$ is uniform in the admissible parameters. Indeed,
$r\le m$, $M\ge V$, and the contribution of the small projections carries the
uniform factor $n^{-\tau}$; the remaining additive terms are $O(1/\log n)$,
uniformly for $a\in[0,1]$.

By \eqref{mult:eq:P-small} and $M\ge V$,
\begin{equation}\label{mult:eq:edge-factor-regime2}
 \frac{eqP}{M}
 \le e\,n^{d-m-\kappa}V^{p-1}.
\end{equation}
Combining \eqref{mult:eq:first-moment-master},
\eqref{mult:eq:entropy-regime2}, and \eqref{mult:eq:edge-factor-regime2}, the logarithm
of the expected number of configurations in a fixed parameter class,
divided by $M\log n$, is at most
\begin{equation}\label{mult:eq:Psi}
 \Psi(r,a)
 :=r+d-m+a(p-r-1)-\kappa+r\tau+o(1).
\end{equation}
Using \eqref{mult:eq:d-m-theta}, write
\[
 \Phi(r,a)
 :=r-r_0+\theta+a(p-r-1),
\]
so that
\[
 \Psi(r,a)=\Phi(r,a)-\kappa+r\tau+o(1).
\]
The estimate \eqref{mult:eq:r-le-p}, together with $p<r_0$, forces the desired
negativity. If $p-r-1\ge0$, then
\begin{align*}
 \Phi(r,a)
 &\le \Phi(r,1)
 =p-r_0+\theta-1
 \le p-r_0<0.
\end{align*}
If $p-r-1<0$, then
\[
 \Phi(r,a)\le\Phi(r,0)=r-r_0+\theta\le0,
\]
because $r$ is an integer and $r\le p<r_0$, hence $r\le r_0-1$.
If equality occurs in the last estimate, then necessarily $\theta=1$ and
$r=r_0-1$. In every case, since $r\le m$ and $m\tau<\kappa/4$,
\[
 -\kappa+r\tau\le -\frac{3\kappa}{4}.
\]
Hence there is a number $\sigma>0$, depending only on
$m,d,p,\rho$, such that for all $a\in[0,1]$ and all admissible integers $r$,
\[
 \Phi(r,a)-\kappa+r\tau\le-\sigma.
\]
After increasing $n_0$ so that the uniform $o(1)$ in
\eqref{mult:eq:Psi} is at most $\sigma/2$, we obtain
\[
 \Psi(r,a)\le-\frac{\sigma}{2}.
\]
Consequently there exists $c_2>0$ such that every fixed parameter class in
Regime~II contributes at most
\begin{equation}\label{mult:eq:regime2-decay}
 n^{-c_2M}.
\end{equation}

We now sum over the parameter classes. Since $v_j\le M$, there are at most
$M^m$ possible vectors $(v_1,\ldots,v_m)$ for a fixed $M$. Necessarily
$M\le n^m$; extending the ensuing sum to all $M\ge1$ only enlarges the upper
bound. Therefore
\eqref{mult:eq:regime1-decay} and \eqref{mult:eq:regime2-decay} give
\[
 \Pp\{\Lambda\text{ is not $(p,\rho)$-locally sparse}\}
 \le \sum_{M\ge1}M^m n^{-cM}=o(1)
\]
for some $c>0$.

Finally, $|\Lambda|$ is binomial with mean $n^d$. A standard Chernoff bound
gives
\[
 \Pp\{|\Lambda|<\tfrac12n^d\}\le e^{-c'n^d}.
\]
Hence, with positive probability and for all sufficiently large $n$,
$\Lambda$ is $(p,\rho)$-locally sparse and
$|\Lambda|\ge n^d/2$. This proves the proposition.
\end{proof}

\begin{remark}\label{mult:rem:two-regimes}
In Regime~I the decay is governed by
$\delta_1=2m/p-d-1$. In Regime~II the number $r$ of coordinate projections
comparable with the largest one satisfies $r\le p<r_0$, where
$r_0=m-\lceil d\rceil+1$. These are the two restrictions appearing in
\eqref{mult:eq:Bayart-upper}.
\end{remark}

\section{Random signs on rectangular traces}\label{mult:sec:signs}

The local estimate from the previous section contains no signs. We next show
that a single choice of Rademacher coefficients controls all rectangular
traces of a fixed support.

For $A_1,\ldots,A_m\subset[n]$ and $\Lambda\subset[n]^m$, put
\begin{equation}\label{mult:eq:rect-trace}
 F_A:=\Lambda\cap(A_1\times\cdots\times A_m).
\end{equation}
Whenever $F_A\ne\varnothing$, replacing $A_j$ by $\pi_j(F_A)$ does not change
the trace:
\begin{equation}\label{mult:eq:effective-rectangle}
 F_A
 =\Lambda\cap\bigl(\pi_1(F_A)\times\cdots\times\pi_m(F_A)\bigr).
\end{equation}
Thus every nonempty rectangular trace may be represented using its
coordinate projections.

\begin{proposition}[Rectangular sign lemma]\label{mult:prop:sign-lemma}
For every $m\ge2$ there exists $C_m>0$ with the following property. For every
$n\in\N$ and every $\Lambda\subset[n]^m$, one can choose signs
$\varepsilon_{\mathbf i}\in\{-1,1\}$, $\mathbf i\in\Lambda$, such that for
every nonempty rectangular trace $F=F_A$,
\begin{equation}\label{mult:eq:rect-sign-estimate}
 \sup_{\substack{|z_i^{(j)}|\le1\\
 i\in\pi_j(F),\ 1\le j\le m}}
 \left|
 \sum_{\mathbf i\in F}\varepsilon_{\mathbf i}
 z^{(1)}_{i_1}\cdots z^{(m)}_{i_m}
 \right|
 \le C_m\sqrt{M(F)H_n(F)}.
\end{equation}
\end{proposition}

\begin{proof}
Choose $(\varepsilon_{\mathbf i})_{\mathbf i\in\Lambda}$ independently with
values $\pm1$ and equal probabilities. We show that
\eqref{mult:eq:rect-sign-estimate} holds simultaneously with positive
probability.

Fix an effective rectangular trace
\[
 F=\Lambda\cap(B_1\times\cdots\times B_m),
 \qquad \pi_j(F)=B_j,
\]
and write $v_j:=|B_j|$ and $M:=|F|$. Let
\[
 S_F(z^{(1)},\ldots,z^{(m)})
 :=\sum_{\mathbf i\in F}\varepsilon_{\mathbf i}
 z^{(1)}_{i_1}\cdots z^{(m)}_{i_m}.
\]
For fixed $z^{(j)}$ with $|z_i^{(j)}|\le1$, Hoeffding's inequality, applied
to the real and imaginary parts, gives
\begin{equation}\label{mult:eq:hoeffding}
 \Pp\{|S_F(z^{(1)},\ldots,z^{(m)})|>t\}
 \le4\exp\left(-c\frac{t^2}{M}\right)
\end{equation}
with an absolute constant $c>0$.

By the maximum modulus principle in each group of variables, the supremum in
\eqref{mult:eq:rect-sign-estimate} can be taken over the torus
$\T^{B_1}\times\cdots\times\T^{B_m}$. Fix
$\delta=(4m)^{-1}$ and choose a $\delta$-net $\mathcal N_j$ of
$\T^{B_j}$ in the sup norm with
\begin{equation}\label{mult:eq:net-size}
 |\mathcal N_j|\le C_m^{v_j}.
\end{equation}
If $K_F$ denotes the left-hand side of \eqref{mult:eq:rect-sign-estimate},
choose, for each $z^{(j)}\in\mathbb T^{B_j}$, a point
$w^{(j)}\in\mathcal N_j$ with
$\|z^{(j)}-w^{(j)}\|_\infty\le\delta$. Multilinearity and telescoping one
coordinate group at a time give
\[
 |S_F(z^{(1)},\ldots,z^{(m)})|
 \le |S_F(w^{(1)},\ldots,w^{(m)})|+m\delta K_F.
\]
Taking the supremum over the torus and using $m\delta=1/4$ yields
$K_F\le M_F+K_F/4$, where
$M_F:=\max_{w^{(j)}\in\mathcal N_j}|S_F(w^{(1)},\ldots,w^{(m)})|$. Hence
$K_F\le(4/3)M_F$, and in particular
\begin{equation}\label{mult:eq:net-reduction}
 K_F\le2
 \max_{z^{(j)}\in\mathcal N_j}
 |S_F(z^{(1)},\ldots,z^{(m)})|.
\end{equation}

For fixed projection sizes $v_1,\ldots,v_m$, set
\[
 S(v_1,\ldots,v_m)
 :=\sum_{j=1}^m v_j\log\frac{en}{v_j}.
\]
The number of possible effective tuples $(B_1,\ldots,B_m)$ is at most
\begin{equation}\label{mult:eq:number-rectangles}
 N_{\mathrm{rect}}
 \le\prod_{j=1}^m\binom{n}{v_j}
 \le e^{S(v_1,\ldots,v_m)}.
\end{equation}
By \eqref{mult:eq:net-size}, for each such tuple the number of choices of net
points is at most
\begin{equation}\label{mult:eq:number-net-points}
 N_{\mathrm{net}}
 \le C_m^{v_1+\cdots+v_m}
 =\exp\bigl((\log C_m)(v_1+\cdots+v_m)\bigr).
\end{equation}
Since $1\le v_j\le n$,
\[
 \log\frac{en}{v_j}\ge1,
\]
and therefore
\begin{equation}\label{mult:eq:v-sum-below-S}
 v_1+\cdots+v_m
 \le S(v_1,\ldots,v_m).
\end{equation}
Thus the entire combinatorial cost of choosing an effective rectangle and
one point of each net is at most
\[
 N_{\mathrm{rect}}N_{\mathrm{net}}
 \le\exp\bigl(C'_m S(v_1,\ldots,v_m)\bigr)
 \le\exp\bigl(C'_m H(v_1,\ldots,v_m)\bigr)
\]
for a constant $C'_m>0$ depending only on $m$.

Set
\[
 H(v_1,\ldots,v_m):=1+S(v_1,\ldots,v_m)
\]
and choose
\[
 t=K_m\sqrt{M H(v_1,\ldots,v_m)}.
\]
By \eqref{mult:eq:net-reduction}, failure of
\eqref{mult:eq:rect-sign-estimate} with constant $2K_m$ implies that for some
net point
\[
 |S_F(z^{(1)},\ldots,z^{(m)})|>K_m\sqrt{MH(v_1,\ldots,v_m)}.
\]
For each fixed net point, \eqref{mult:eq:hoeffding} gives probability at most
\[
 4e^{-cK_m^2H(v_1,\ldots,v_m)}.
\]
Multiplying this tail by the bounds
\eqref{mult:eq:number-rectangles} and \eqref{mult:eq:number-net-points}, we obtain
\begin{align}
 &\Pp\Bigl\{\text{\eqref{mult:eq:rect-sign-estimate} fails for some effective
 rectangle with these sizes}\Bigr\}\notag\\
 &\qquad\le
 4\exp\Bigl(
 -cK_m^2H(v_1,\ldots,v_m)
 +C'_mH(v_1,\ldots,v_m)
 \Bigr)\notag\\
 &\qquad=
 4\exp\Bigl(-\bigl(cK_m^2-C'_m\bigr)
 H(v_1,\ldots,v_m)\Bigr).
 \label{mult:eq:sign-union-fixed-sizes}
\end{align}
Choose $K_m$ so large that
\[
 B_m:=cK_m^2-C'_m>2m+4+\log4.
\]
Since $H(v_1,\ldots,v_m)\ge1$, the factor $4$ can be absorbed:
\[
 4e^{-B_mH}
 \le e^{-(B_m-\log4)H}.
\]
Put
\[
 A_m:=B_m-\log4>2m+4.
\]
Then
\begin{equation}\label{mult:eq:sign-fixed-decay}
 \Pp\Bigl\{\text{failure for some effective rectangle with the fixed
 sizes }(v_1,\ldots,v_m)\Bigr\}
 \le e^{-A_m H(v_1,\ldots,v_m)}.
\end{equation}

It remains to sum over the size tuples. For $1\le v\le n$, the function
$v\mapsto v\log(en/v)$ is increasing, hence every nonempty trace satisfies
\[
 H(v_1,\ldots,v_m)\ge 1+m\log(en).
\]
There are at most $n^m$ tuples $(v_1,\ldots,v_m)$. Therefore, by
\eqref{mult:eq:sign-fixed-decay},
\begin{align*}
 &\Pp\Bigl\{\text{\eqref{mult:eq:rect-sign-estimate} fails for at least one
 nonempty rectangular trace}\Bigr\}\\
 &\qquad\le
 n^m\exp\bigl(-A_m(1+m\log(en))\bigr)\\
 &\qquad= e^{-A_m(1+m)}\,n^{m-A_m m}.
\end{align*}
Because $A_m>2m+2$, this quantity is strictly smaller than $1$ for every
$n\ge1$. Hence there is a deterministic choice of signs for which
\eqref{mult:eq:rect-sign-estimate} holds simultaneously for all nonempty
rectangular traces.
\end{proof}

\section{From local sparsity to multilinear norm bounds}\label{mult:sec:dyadic}

We now pass from the combinatorial estimate
\eqref{mult:eq:local-sparsity} to the norm of an $m$-linear form. The only
additional ingredient is a dyadic decomposition by decreasing coordinate
size.

\begin{proposition}\label{mult:prop:local-to-norm}
Let $m\ge2$, $p>2$, $\rho>0$, and let
$\Lambda\subset[n]^m$ be $(p,\rho)$-locally sparse. Then there exist signs
$\varepsilon_{\mathbf i}\in\{-1,1\}$ such that the $m$-linear form
\begin{equation}\label{mult:eq:T-Lambda}
 T_\Lambda(x^{(1)},\ldots,x^{(m)})
 :=\sum_{\mathbf i\in\Lambda}\varepsilon_{\mathbf i}
   x^{(1)}_{i_1}\cdots x^{(m)}_{i_m}
\end{equation}
satisfies
\begin{equation}\label{mult:eq:norm-from-local}
 \|T_\Lambda\|_{\mathcal L(^m\ell_p^n)}
 \le C_{m,p}\, n^{\rho/2}(1+\log n)^m.
\end{equation}
\end{proposition}

\begin{proof}
Choose the signs given by Proposition~\ref{mult:prop:sign-lemma}. Fix
$x^{(j)}\in B_{\ell_p^n}$, $1\le j\le m$. For each $j$, rearrange the
coordinates of $x^{(j)}$ in nonincreasing order of modulus and partition the
corresponding indices into consecutive rank blocks
\[
 A_{j,0},A_{j,1},\ldots,A_{j,L_j},
\]
where $|A_{j,0}|=1$, and for $k\ge1$ the block $A_{j,k}$ contains ranks
$2^{k-1}+1$ through $\min\{2^k,n\}$. Thus
$L_j\le\lceil\log_2 n\rceil$, and for $i\in A_{j,k}$,
\begin{equation}\label{mult:eq:rank-bound}
 |x_i^{(j)}|
 \le C_p |A_{j,k}|^{-1/p}.
\end{equation}
Indeed, if the $r$-th largest coordinate had modulus larger than
$r^{-1/p}$, the $\ell_p$ norm would exceed $1$.

Fix a multi-index of blocks $\mathbf k=(k_1,\ldots,k_m)$ and set
\[
 F_{\mathbf k}
 :=\Lambda\cap
 (A_{1,k_1}\times\cdots\times A_{m,k_m}).
\]
If $F_{\mathbf k}=\varnothing$, its contribution is zero. Otherwise write
\[
 v_j:=|\pi_j(F_{\mathbf k})|,
 \qquad P:=\prod_{j=1}^m v_j,
 \qquad M:=|F_{\mathbf k}|.
\]
Since $v_j\le |A_{j,k_j}|$, \eqref{mult:eq:rank-bound} and the rectangular sign
estimate give
\begin{align}
 &\left|
 \sum_{\mathbf i\in F_{\mathbf k}}
 \varepsilon_{\mathbf i}
 x^{(1)}_{i_1}\cdots x^{(m)}_{i_m}
 \right|\notag\\
 &\qquad\le
 C_{m,p}\prod_{j=1}^m|A_{j,k_j}|^{-1/p}
 \sqrt{M H_n(F_{\mathbf k})}\notag\\
 &\qquad\le
 C_{m,p}P^{-1/p}\sqrt{M H_n(F_{\mathbf k})}.
 \label{mult:eq:block-bound-before-local}
\end{align}
The local sparsity of $\Lambda$ now yields
\[
 M H_n(F_{\mathbf k})\le n^\rho P^{2/p},
\]
and therefore
\begin{equation}\label{mult:eq:block-final}
 \left|
 \sum_{\mathbf i\in F_{\mathbf k}}
 \varepsilon_{\mathbf i}
 x^{(1)}_{i_1}\cdots x^{(m)}_{i_m}
 \right|
 \le C_{m,p}n^{\rho/2}.
\end{equation}

There are at most
$(1+\lceil\log_2n\rceil)^m$ choices of $\mathbf k$. Summing
\eqref{mult:eq:block-final} over the blocks gives
\eqref{mult:eq:norm-from-local}.
\end{proof}

\begin{remark}\label{mult:rem:rho-epsilon}
The factor $(1+\log n)^m$ does not affect the value of the invariant in
Definition~\ref{mult:def:gamma-mult}. Given $\eta>0$, one first chooses
$\rho<\eta$ and then absorbs the logarithmic factor into $n^{\eta-\rho/2}$.
No endpoint estimate at
$p=\min\{m-\lceil d\rceil+1,2m/(d+1)\}$ is required, since
$\gamma_{\mathrm{mult}}$ is defined as a supremum.
\end{remark}

\section{The sharp invariant}\label{mult:sec:main-proof}

We can now prove Theorem~\ref{mult:thm:B}.

\begin{proof}[Proof of Theorem~\ref{mult:thm:B}]
Bayart's Proposition~5.3 gives the upper bound
\eqref{mult:eq:Bayart-upper}, so only the reverse inequality is required.

If $d\in(m-2,m]$, the desired equality is already contained in
\cite[Corollary~5.4]{BayartJEMS}: the value is $2$ for
$d\in(m-2,m-1]$ and $1$ for $d\in(m-1,m]$. These are exactly the values of
the right-hand side of \eqref{mult:eq:main-formula}.

Assume now that $d\in[1,m-2]$. Set
\[
 G(m,d):=
 \min\left\{m-\lceil d\rceil+1,\frac{2m}{d+1}\right\}.
\]
Then $G(m,d)>2$. Fix any
\[
 2<p<G(m,d)
\]
and any $\eta>0$. Choose $\rho>0$ so small that $\rho<\eta$.
By Proposition~\ref{mult:prop:random-support}, for all sufficiently large $n$
there exists a $(p,\rho)$-locally sparse set
$\Lambda_n\subset[n]^m$ with
\[
 |\Lambda_n|\ge c n^d.
\]
Proposition~\ref{mult:prop:local-to-norm} provides signs
$\varepsilon_{\mathbf i}\in\{-1,1\}$ such that
\[
 \|T_n\|_{\mathcal L(^m\ell_p^n)}
 \le C_{m,p}n^{\rho/2}(1+\log n)^m
 \le C_{m,p,\eta}n^\eta.
\]
After modifying the constant to cover the finitely many smaller values of
$n$, Definition~\ref{mult:def:gamma-mult} gives
$p\in\Gamma_{\mathrm{mult}}(m,d)$. Since this holds for every
$2<p<G(m,d)$,
\[
 \gamma_{\mathrm{mult}}(m,d)\ge G(m,d).
\]
Together with \eqref{mult:eq:Bayart-upper}, this proves
\eqref{mult:eq:main-formula}.
\end{proof}

\begin{proof}[Proof of Corollary~\ref{mult:cor:noninteger}]
By Theorem~\ref{mult:thm:B},
\begin{equation}\label{mult:eq:four-two-eight-thirds}
 \gamma_{\mathrm{mult}}(4,2)
 =\min\left\{4-2+1,\frac{2\cdot4}{2+1}\right\}
 =\min\left\{3,\frac83\right\}
 =\frac83\notin\mathbb Z.
\end{equation}
Thus Bayart's Question~5.7 has a negative answer.
\end{proof}

The two extremal problems from \cite[Problems~5.1 and~5.2]{BayartJEMS} can
now be displayed side by side. By Theorem~\ref{thm:A},
\[
 \mathrm{prod}(m,d)
 =\min\left\{\frac md,m-\lceil d\rceil+1\right\},
\]
while Theorem~\ref{mult:thm:B} gives
\[
 \gamma_{\mathrm{mult}}(m,d)
 =\min\left\{\frac{2m}{d+1},m-\lceil d\rceil+1\right\}.
\]
For $d>1$ one has
\[
 \frac{2m}{d+1}>\frac md.
\]
Consequently, whenever
\begin{equation}\label{mult:eq:strict-gap-condition}
 \frac md<m-\lceil d\rceil+1,
\end{equation}
the two invariants are strictly separated:
\begin{equation}\label{mult:eq:strict-gap}
 \mathrm{prod}(m,d)
 <\gamma_{\mathrm{mult}}(m,d).
\end{equation}
Thus cancellation changes the optimal exponent precisely in the range where
the dimensional obstruction is active before the section obstruction.

\begin{remark}\label{mult:rem:mechanism-final}
Proposition~\ref{mult:prop:random-support} shows that the two terms in Bayart's
upper bound are the only polynomial obstructions to keeping
\[
 \frac{M(F)H_n(F)}{P(F)^{2/p}}
\]
subpolynomial on a support with $n^d$ points.
\end{remark}

\endgroup

\section{Sparse supports for Hardy--Littlewood inequalities}\label{sec:sharp-HL-sparse}

For a normed space $X$, write $B_X:=\{x\in X:\|x\|_X\le1\}$.

\subsection{Real and complex scalar fields}

The critical exponent is independent of the scalar field.

{
\begin{lemma}\label{lem:complexification-HL}
Let $\Lambda\subset\mathbb{N}^m$ be infinite and let $\mathbf p\in[1,\infty]^m$. Then
\[
 \mathrm{HL}_{\mathbb{R}}(\Lambda;\mathbf p)
 \le \mathrm{HL}_{\mathbb C}(\Lambda;\mathbf p).
\]
More precisely, every exponent admissible over $\mathbb C$ is admissible over $\mathbb{R}$, with the admissibility constant multiplied by at most $2^m$.
\end{lemma}
}

\begin{proof}
Let
\[
 T:Z_{p_1}(\mathbb{R})\times\cdots\times Z_{p_m}(\mathbb{R})\longrightarrow\mathbb{R}
\]
be a continuous real $m$-linear form. Write $z_j=x_j+iy_j$ with
$x_j,y_j\in Z_{p_j}(\mathbb{R})$, and define the canonical complexification
\[
 \widetilde T(z_1,\ldots,z_m)
 :=\sum_{\varepsilon\in\{0,1\}^m}
 i^{|\varepsilon|}
 T\bigl(u_1^{(\varepsilon_1)},\ldots,u_m^{(\varepsilon_m)}\bigr),
\]
where $u_j^{(0)}=x_j$ and $u_j^{(1)}=y_j$. Since
\[
 \|x_j\|_{p_j}\le \|z_j\|_{p_j},
 \qquad
 \|y_j\|_{p_j}\le \|z_j\|_{p_j},
\]
with the same statement for $p_j=\infty$, we obtain
\[
 \|\widetilde T\|\le 2^m\|T\|.
\]
Moreover, the canonical basis vectors are real, so
\[
 \widetilde T(e_{i_1},\ldots,e_{i_m})
 =T(e_{i_1},\ldots,e_{i_m}).
\]
Hence, if $s$ is admissible over $\mathbb C$ with constant $C_s$, then
\[
 \|T\|_{\Lambda,s}
 =\|\widetilde T\|_{\Lambda,s}
 \le C_s\|\widetilde T\|
 \le 2^m C_s\|T\|.
\]
Thus $s$ is also admissible over $\mathbb{R}$.
\end{proof}

{
\begin{proposition}[Scalar-field invariance of the critical exponent]\label{prop:HL-scalar-invariance}
Let $\Lambda\subset\mathbb{N}^m$ be infinite and let
$\mathbf p\in[1,\infty]^m$. Then
\[
 \mathrm{HL}_{\mathbb{R}}(\Lambda;\mathbf p)
 =
 \mathrm{HL}_{\mathbb C}(\Lambda;\mathbf p).
\]
\end{proposition}

\begin{proof}
Lemma~\ref{lem:complexification-HL} gives
$\mathrm{HL}_{\mathbb{R}}(\Lambda;\mathbf p)
\le \mathrm{HL}_{\mathbb C}(\Lambda;\mathbf p)$.
For the converse, let $s$ be admissible over $\mathbb{R}$ with constant $C_s$, and let
\[
 T:Z_{p_1}(\mathbb C)\times\cdots\times Z_{p_m}(\mathbb C)\longrightarrow\mathbb C
\]
be continuous and $m$-linear. On real vectors write
$T=T_1+iT_2$, where $T_1$ and $T_2$ are real $m$-linear forms. Since
$\|T_1\|,\|T_2\|\le\|T\|$, Minkowski's inequality gives
\[
 \|T\|_{\Lambda,s}
 \le \|T_1\|_{\Lambda,s}+\|T_2\|_{\Lambda,s}
 \le C_s(\|T_1\|+\|T_2\|)
 \le 2C_s\|T\|.
\]
Thus every exponent admissible over $\mathbb{R}$ is admissible over $\mathbb C$, and the two infima coincide.
\end{proof}
}

\subsection{Finite-block estimates}

For a finite set $E\subset[N]^m$, define the $m$-linear form
\begin{equation}\label{eq:def-AE}
 \mathcal A_E:(\mathbb K^N)^m\longrightarrow\mathbb K,\qquad
 \mathcal A_E(y^{(1)},\ldots,y^{(m)})
 :=\sum_{\mathbf i\in E}\prod_{j=1}^m y^{(j)}_{i_j}.
\end{equation}

Recall that, for $E\subset[N]^m$, $\Delta(E)$ denotes the largest cardinality of a one-coordinate fiber of $E$.

\begin{lemma}\label{lem:HL-simplex}
Let $N\in\mathbb{N}$, $1\le d\le m$, $C_0,C_1>0$, and let
$E\subset[N]^m$ satisfy
\[
 |E|\le C_0N^d,\qquad \Delta(E)\le C_1N^{d-1}.
\]
If $q_1,\ldots,q_m\in[1,\infty]$ and $\sum_j1/q_j\le1$, then
\begin{equation}\label{eq:HL-simplex}
 \sup_{\substack{y^{(j)}\in B_{\ell_{q_j}^N}\\1\le j\le m}}
 |\mathcal A_E(y^{(1)},\ldots,y^{(m)})|
 \le C N^{d-\sum_j1/q_j},
\end{equation}
where $C$ depends only on $m,C_0,C_1$.
\end{lemma}

\begin{proof}
At the all-$\ell_\infty$ endpoint,
\[
 \sup_{\substack{y^{(j)}\in B_{\ell_\infty^N}\\1\le j\le m}}
 |\mathcal A_E(y^{(1)},\ldots,y^{(m)})|\le |E|\le C_0N^d.
\]
If the $k$th factor is $\ell_1^N$ and the others are $\ell_\infty^N$, then
{
grouping the summation according to the $k$th coordinate gives
\begin{align*}
 |\mathcal A_E(y^{(1)},\ldots,y^{(m)})|
 &\le \sum_{r=1}^N |y_r^{(k)}|
 \sum_{\substack{\mathbf i\in E\\ i_k=r}}
 \prod_{j\ne k}|y_{i_j}^{(j)}|\\
 &\le \sum_{r=1}^N |y_r^{(k)}|\,|E_{k,r}|
 \prod_{j\ne k}\|y^{(j)}\|_\infty\\
 &\le \Delta(E)\|y^{(k)}\|_1
 \prod_{j\ne k}\|y^{(j)}\|_\infty.
\end{align*}
Consequently,}
\[
 |\mathcal A_E(y^{(1)},\ldots,y^{(m)})|
 \le \Delta(E)\|y^{(k)}\|_1\prod_{j\ne k}\|y^{(j)}\|_\infty
 \le C_1N^{d-1}\|y^{(k)}\|_1\prod_{j\ne k}\|y^{(j)}\|_\infty.
\]
The reciprocal-exponent vertices are
$(0,\ldots,0),e_1,\ldots,e_m$. Set
\[
 \theta_j:=\frac1{q_j}\quad(1\le j\le m),
 \qquad
 \theta_0:=1-\sum_{j=1}^m\theta_j.
\]
The assumption $\sum_j1/q_j\le1$ gives $\theta_0\ge0$, and
\[
 \theta_0+\theta_1+\cdots+\theta_m=1.
\]
Thus $(1/q_1,\ldots,1/q_m)$ is the convex combination
\[
 \theta_0(0,\ldots,0)+\sum_{j=1}^m\theta_j e_j.
\]
Multilinear complex interpolation between the all-$\ell_\infty$ estimate
and the $m$ estimates with one $\ell_1$ factor gives the power
\begin{align*}
 \theta_0d+\sum_{j=1}^m\theta_j(d-1)
 &=d\left(\theta_0+\sum_{j=1}^m\theta_j\right)
   -\sum_{j=1}^m\theta_j\\
 &=d-\sum_{j=1}^m\theta_j\\
&=d-\sum_{j=1}^m\frac1{q_j}.
\end{align*}
{More precisely, the $j$th domain space is obtained by
interpolating $\ell_\infty^N$ with $\ell_1^N$ with parameter
$\theta_j$, while the remaining weight $\theta_0$ corresponds to the
all-$\ell_\infty$ vertex.  The cases $q_j=\infty$ correspond to $\theta_j=0$. The estimate is
field-independent: for arbitrary scalar vectors,
\[
 |\mathcal A_E(y^{(1)},\ldots,y^{(m)})|
 \le \mathcal A_E(|y^{(1)}|,\ldots,|y^{(m)}|),
\]
and taking coordinatewise moduli preserves all the norms involved.}
{The interpolated operator norm is at most
$C_0^{\theta_0}C_1^{1-\theta_0}$, and hence at most
$\max\{C_0,C_1\}$. Together with the displayed interpolation identity $\theta_0d+\sum_{j=1}^m\theta_j(d-1)=d-\sum_{j=1}^m1/q_j$, this proves
\eqref{eq:HL-simplex}.}
\end{proof}

\begin{corollary}\label{cor:HL-quadratic}
Under the hypotheses of Lemma~\ref{lem:HL-simplex}, let
$\mathbf p\in[1,\infty]^m$ satisfy $|1/\mathbf p|\le1/2$. Then
\begin{equation}\label{eq:HL-quadratic}
 \sup_{\substack{x^{(j)}\in B_{\ell_{p_j}^N}\\1\le j\le m}}
 \left(\sum_{\mathbf i\in E}\prod_{j=1}^m|x^{(j)}_{i_j}|^2\right)^{1/2}
 \le C N^{d/2-|1/\mathbf p|}.
\end{equation}
\end{corollary}

\begin{proof}
The hypothesis $|1/\mathbf p|\le1/2$ implies $p_j\ge2$ for every $j$. Put
\[
 q_j:=\frac{p_j}{2}\quad\text{if }p_j<\infty,
 \qquad q_j:=\infty\quad\text{if }p_j=\infty,
 \qquad y^{(j)}_r:=|x^{(j)}_r|^2.
\]
Then $q_j\in[1,\infty]$. If $p_j<\infty$, then
\begin{align*}
 \|y^{(j)}\|_{\ell_{q_j}^N}^{q_j}
 &=\sum_{r=1}^N |y^{(j)}_r|^{q_j}\\
 &=\sum_{r=1}^N |x^{(j)}_r|^{2q_j}\\
 &=\sum_{r=1}^N |x^{(j)}_r|^{p_j}
 =\|x^{(j)}\|_{\ell_{p_j}^N}^{p_j}
 \le1.
\end{align*}
If $p_j=\infty$, then
\[
 \|y^{(j)}\|_{\ell_\infty^N}
 =\max_{1\le r\le N}|x^{(j)}_r|^2
 =\|x^{(j)}\|_{\ell_\infty^N}^2
 \le1.
\]
Hence $\|y^{(j)}\|_{\ell_{q_j}^N}\le1$ for every $j$. With the convention
$1/\infty=0$, we also have
\[
 \sum_{j=1}^m\frac1{q_j}
 =\sum_{j=1}^m\frac2{p_j}
 =2\left|\frac1{\mathbf p}\right|
 \le1.
\]
Lemma~\ref{lem:HL-simplex} therefore gives
\[
 {
 \mathcal A_E(y^{(1)},\ldots,y^{(m)})
 =}
 \sum_{\mathbf i\in E}\prod_{j=1}^m|x^{(j)}_{i_j}|^2
 \le C N^{d-2|1/\mathbf p|},
\]
{Taking square roots and, if necessary, replacing
$C^{1/2}$ by $C$ proves the claim.}
\end{proof}

{
\begin{lemma}[Signed finite-block estimate]\label{lem:signed-HL-block}
Let $N\in\mathbb{N}$, $1\le d\le m$, and $C_0,C_1>0$. Suppose that
$E\subset[N]^m$ satisfies
\[
 |E|\le C_0N^d,
 \qquad
 \Delta(E)\le C_1N^{d-1}.
\]
If $\mathbf p\in[1,\infty]^m$ and $|1/\mathbf p|\le1/2$, then there are
signs $\varepsilon_{\mathbf i}\in\{-1,1\}$, $\mathbf i\in E$, such that
the $m$-linear form
\[
 T_E(x^{(1)},\ldots,x^{(m)})
 :=\sum_{\mathbf i\in E}\varepsilon_{\mathbf i}
 \prod_{j=1}^m x^{(j)}_{i_j}
\]
satisfies
\begin{equation}\label{eq:signed-HL-block}
 \|T_E\|
 \le C_{m,\mathbf p,C_0,C_1}
 N^{(d+1)/2-|1/\mathbf p|}
\end{equation}
over both $\mathbb C$ and $\mathbb{R}$.
\end{lemma}

\begin{proof}
We first work over $\mathbb C$. On
\[
 \Omega_E:=\{-1,1\}^{E}
\]
equip the coordinate signs $(\varepsilon_{\mathbf i})_{\mathbf i\in E}$
with their product probability measure.  In the notation of
\cite[Proposition~4.1]{BayartJEMS}, take
\[
 X_j=\ell_{p_j}^N(\mathbb C),\qquad
 x_{\mathbf i}^{*(j)}=e_{i_j}^*,\qquad
 s=2,\qquad \beta=\frac12.
\]
Here $e_i^*$ denotes the $i$th coordinate functional.
Equivalently, one may extend the random family to $[N]^m$ by setting it
equal to zero outside $E$.  For fixed
$x=(x^{(1)},\ldots,x^{(m)})$ in the product of the unit balls, set
\[
 Y_x(\varepsilon)
 :=\sum_{\mathbf i\in E}\varepsilon_{\mathbf i}
 \prod_{j=1}^m x^{(j)}_{i_j}.
\]
By \cite[Proposition~4.1]{BayartJEMS}, on an event of probability at least
$1/2$,
\begin{equation}\label{eq:Bayart-random-intermediate}
 \sup_x|Y_x(\varepsilon)|
 \le C_mN^{1/2}\sup_x\|Y_x\|_{\psi_2},
 \qquad \psi_2(t)=e^{t^2}-1.
\end{equation}
The factor $N^{1/2}$ is the factor $N^{1/s}$ in that proposition with
$s=2$.

The standard subgaussian estimate for a Rademacher sum gives, for scalar
coefficients $(a_{\mathbf i})_{\mathbf i\in E}$,
\[
 \left\|\sum_{\mathbf i\in E}
 \varepsilon_{\mathbf i}a_{\mathbf i}\right\|_{\psi_2}
 \le C\left(\sum_{\mathbf i\in E}|a_{\mathbf i}|^2\right)^{1/2}.
\]
For complex coefficients this follows by applying the real estimate to
the real and imaginary parts; only the absolute constant changes.  With
$a_{\mathbf i}=\prod_jx^{(j)}_{i_j}$, Corollary~\ref{cor:HL-quadratic}
therefore implies
\[
 \sup_x\|Y_x\|_{\psi_2}
 \le C_{m,\mathbf p,C_0,C_1}
 N^{d/2-|1/\mathbf p|}.
\]
Combining this estimate with \eqref{eq:Bayart-random-intermediate}, we
obtain, with positive probability,
\[
 \sup_x|Y_x(\varepsilon)|
 \le C_{m,\mathbf p,C_0,C_1}
 N^{(d+1)/2-|1/\mathbf p|}.
\]
Hence at least one deterministic choice of signs satisfies
\eqref{eq:signed-HL-block}.  The same coefficients define a real form,
and restriction from the complex unit balls to the real unit balls can
only decrease its norm.
\end{proof}
}

\section{Proof of Theorem~C}\label{sec:proof-theorem-C}

\begin{proof}[Proof of Theorem~\ref{thm:C}]
\proofstep{1}{Construction and exact dimension for $1<d<m$}
Assume first $1<d<m$. Choose $N_\nu\uparrow\infty$ and blocks
$E_\nu\subset[N_\nu]^m$ from Lemma~\ref{lem:sparse-random-blocks-general}. For
each coordinate $j$, reindex the copy of $[N_\nu]$ used in the $\nu$th block onto
a finite set $I_{j,\nu}\subset\mathbb{N}$, with the sets $I_{j,\nu}$ pairwise disjoint in $\nu$ for each fixed
coordinate $j$. {For instance, these sets may be chosen as
successive intervals of lengths $N_\nu$.} Let $\widetilde E_\nu$ be the reindexed copy and put
\[
 \Lambda_0:=\bigcup_{\nu=1}^\infty\widetilde E_\nu,
 \qquad
 {\mathrm{Diag}:=\{(n,\ldots,n):n\in\mathbb{N}\},\qquad
 \Lambda:=\Lambda_0\cup\mathrm{Diag}.}
\]
{For each $\nu$, take $A_j=I_{j,\nu}$ and $n=N_\nu$ in the
definition of the counting function.  Then
\[
 \Lambda(N_\nu)\ge |\widetilde E_\nu|=|E_\nu|
 \ge \frac12N_\nu^d.
\]
Consequently,
\[
 \limsup_{\nu\to\infty}
 \frac{\log\Lambda(N_\nu)}{\log N_\nu}\ge d,
\]
and hence $\dim(\Lambda)\ge d$.} Conversely, for
$A_j\subset\mathbb{N}$ with $|A_j|\le n$, set
$n_\nu=\max_j|A_j\cap I_{j,\nu}|$. For each $j$, the coordinate sets $I_{j,\nu}$ are disjoint in $\nu$;
and therefore
\[
 \sum_\nu n_\nu\le \sum_\nu\sum_{j=1}^m |A_j\cap I_{j,\nu}|
 =\sum_{j=1}^m\sum_\nu |A_j\cap I_{j,\nu}|\le mn.
\] Therefore
\eqref{eq:random-block-dimension-general} gives
\[
 |\Lambda_0\cap(A_1\times\cdots\times A_m)|
 \le K_0\sum_\nu n_\nu^d
 {\le K_0\left(\sum_\nu n_\nu\right)^d}
 \le K_0m^dn^d.
\]
{Here only finitely many $n_\nu$ are nonzero because the
sets $A_j$ are finite, and the middle inequality follows from $d\ge1$.}
{The estimate obtained from \eqref{eq:random-block-dimension-general} is an estimate for $\Lambda_0$. Since
${|\mathrm{Diag}\cap(A_1\times\cdots\times A_m)|}\le n$, we obtain
\[
 \Lambda(n)\le K_0m^dn^d+n\le (K_0m^d+1)n^d.
\]
Together with the lower bound \eqref{eq:random-block-size-general}, this gives $\dim(\Lambda)=d$ and shows that the dimension is exact in the sense of Definition~\ref{def:comb-dim}.}

\proofstep{2}{Random signs and the block norm}
{Fix $\mathbf p$ with $|1/\mathbf p|\le1/2$.  Since
\eqref{eq:random-block-size-general} and
\eqref{eq:random-block-fibers} give
\[
 |E_\nu|\le2N_\nu^d,
 \qquad
 \Delta(E_\nu)\le K_0N_\nu^{d-1},
\]
Lemma~\ref{lem:signed-HL-block}, with $C_0=2$ and $C_1=K_0$,
provides signs}
$\varepsilon_{\mathbf i}\in\{-1,1\}$ for which
\[
 T_\nu:\ell_{p_1}^{N_\nu}\times\cdots\times\ell_{p_m}^{N_\nu}
 \longrightarrow\mathbb K,
 \qquad
 T_\nu(x^{(1)},\ldots,x^{(m)})
 :=\sum_{\mathbf i\in E_\nu}\varepsilon_{\mathbf i}
 \prod_{j=1}^m x^{(j)}_{i_j}
\]
satisfies
\begin{equation}\label{eq:HL-random-sign-norm}
 \|T_\nu\|\le C_{m,\mathbf p}
 N_\nu^{(d+1)/2-|1/\mathbf p|}.
\end{equation}
{For each coordinate, the restriction map
\[
 R_{j,\nu}:Z_{p_j}(\mathbb K)\longrightarrow
 \ell_{p_j}(I_{j,\nu})
\]
has norm one, and extension by zero from
$\ell_{p_j}(I_{j,\nu})$ into $Z_{p_j}(\mathbb K)$ is isometric.  Thus
reindexing and extending the form do not change its norm.  We may therefore
regard $T_\nu$ as an $m$-linear form on}
$Z_{p_1}(\mathbb K)\times\cdots\times Z_{p_m}(\mathbb K)$ supported on the $\nu$th block of
$\Lambda$. {The support $\Lambda$ depends only on $m$ and $d$ and is fixed independently
of $\mathbf p$. Once $\mathbf p$ is fixed, the testing signs may be selected
separately on each block.}

\proofstep{3}{Extraction of the sharp exponent}
Let $\sigma\ge1$ be admissible for $(\Lambda,\mathbf p)$, and denote by
$C_\sigma$ a constant for which the admissibility estimate
\eqref{eq:HL-admissible-def} holds. Since the nonzero coefficients of
$T_\nu$ have modulus one, Definition~\ref{def:HL-exponent}, the lower size
bound, and \eqref{eq:HL-random-sign-norm} give
\[
 {2^{-1/\sigma}N_\nu^{d/\sigma}
 \le C_{\sigma}C_{m,\mathbf p}
 N_\nu^{(d+1)/2-|1/\mathbf p|}.}
\]
{Equivalently,
\[
 \left(\frac d\sigma-\frac{d+1}{2}
 +\left|\frac1{\mathbf p}\right|\right)\log N_\nu
 \le \log(C_\sigma C_{m,\mathbf p})+\frac{\log2}{\sigma}.
\]
Dividing by $\log N_\nu$ and letting $\nu\to\infty$ yields}
\[
 \frac1{\sigma}\le\frac{d+1}{2d}-\frac1d\left|\frac1{\mathbf p}\right|.
\]
{Since this holds for every admissible $\sigma$, taking the
infimum over all such exponents gives
\[
 \mathrm{HL}_{\mathbb K}(\Lambda;\mathbf p)
 \ge
 \left(
 \frac{d+1}{2d}-\frac1d\left|\frac1{\mathbf p}\right|
 \right)^{-1}.
\]
Equivalently,}
\begin{equation}\label{eq:HL-upper-reciprocal-unified}
 \mathrm{HL}_{\mathbb K}(\Lambda;\mathbf p)^{-1}
 \le\frac{d+1}{2d}-\frac1d\left|\frac1{\mathbf p}\right|.
\end{equation}
For $\mathbb K=\mathbb C$, the reverse inequality is \eqref{eq:Bayart-HL-lower-unified} when $|1/\mathbf p|<1/2$, and follows from \cite[Theorem~1.1(b)]{BayartJEMS} on the boundary $|1/\mathbf p|=1/2$. For $\mathbb K=\mathbb{R}$, Proposition~\ref{prop:HL-scalar-invariance} transfers the same critical exponent from the complex to the real scalar field. Hence equality holds over both scalar fields for $1<d<m$.

{
\proofstep{4}{The range $1/2<|1/\mathbf p|<1$}
Assume $1<d<m$ and fix $\mathbf p$ with
$1/2<|1/\mathbf p|<1$. Put
\[
 r:=\frac1{1-|1/\mathbf p|}.
\]
By \cite[Theorem~1.1(b)]{BayartJEMS}, $r$ is admissible over $\mathbb C$ for every infinite support, and Proposition~\ref{prop:HL-scalar-invariance} transfers the same critical exponent to $\mathbb{R}$. Hence
\[
 \mathrm{HL}_{\mathbb K}(\Lambda;\mathbf p)^{-1}
 \ge 1-\left|\frac1{\mathbf p}\right|.
\]
For the reverse inequality, let $a=(a_n)$ be finitely supported and consider the diagonal form
\[
 T_a(x^{(1)},\ldots,x^{(m)})
 :=\sum_n a_n x_n^{(1)}\cdots x_n^{(m)}.
\]
{Since
\[
 \frac1r+\sum_{j=1}^m\frac1{p_j}=1,
\]
generalized H\"older's inequality gives
\begin{align*}
 |T_a(x^{(1)},\ldots,x^{(m)})|
 &\le \sum_n|a_n|\prod_{j=1}^m|x_n^{(j)}|\\
 &\le \|a\|_r\prod_{j=1}^m\|x^{(j)}\|_{p_j}.
\end{align*}
Thus $\|T_a\|\le\|a\|_r$.} To obtain equality, let
$r'=1/|1/\mathbf p|$ be the conjugate exponent of $r$, and choose a finitely
supported $b\in\ell_{r'}$ with $\|b\|_{r'}=1$ and
$\sum_n a_nb_n=\|a\|_r$. {Such a vector exists by the
finite-dimensional duality between $\ell_r$ and $\ell_{r'}$; in the
complex case its phases are chosen so that the displayed pairing is real
and nonnegative.} Since $|1/\mathbf p|>0$, choose an index $j_0$ with $p_{j_0}<\infty$. For $p_j<\infty$ and $j\ne j_0$, set
\[
 x_n^{(j)}:=|b_n|^{r'/p_j},
\]
and define
\[
 x_n^{(j_0)}:=
 \begin{cases}
 \dfrac{b_n}{|b_n|}|b_n|^{r'/p_{j_0}},& b_n\ne0,\\
 0,& b_n=0.
 \end{cases}
\]
For $p_j=\infty$, take $x_n^{(j)}=1$ on $\operatorname{supp}b$ and $0$ elsewhere.
{If $p_j<\infty$, then
\[
 \|x^{(j)}\|_{p_j}^{p_j}
 =\sum_n|b_n|^{r'}=1,
\]
whereas, if $p_j=\infty$, the vector just defined is finitely supported
and has supremum norm one.  Hence {$x^{(j)}\in B_{Z_{p_j}(\mathbb K)}$} for every
$j$.} Moreover, since
$r'\sum_j1/p_j=1$,
\[
 \prod_{j=1}^m x_n^{(j)}=b_n.
\]
Consequently $\|T_a\|\ge\|a\|_r$, and hence
\[
 \|T_a\|=\|a\|_r.
\]
The construction is finitely supported, so the coordinates corresponding to
$p_j=\infty$ belong to $c_0$.
If $\sigma$ is admissible for $(\Lambda,\mathbf p)$, then $\mathrm{Diag}\subset\Lambda$ yields
\[
 \|a\|_{\ell_\sigma}
 \le C_\sigma\|T_a\|
 =C_\sigma\|a\|_r.
\]
Taking $a_1=\cdots=a_N=1$ and the remaining coefficients zero gives
$N^{1/\sigma}\le C_\sigma N^{1/r}$. Letting $N\to\infty$ shows
\[
 \frac1\sigma\le\frac1r
 =1-\left|\frac1{\mathbf p}\right|.
\]
Therefore
\[
 \mathrm{HL}_{\mathbb K}(\Lambda;\mathbf p)^{-1}
 =1-\left|\frac1{\mathbf p}\right|.
\]
}

\proofstep{5}{The endpoint dimensions $d=1$ and $d=m$}
For $d=1$, take
\[
 E_N:=\{(r,\ldots,r):1\le r\le N\},
\]
and for $d=m$ take $E_N:=[N]^m$. Choose a sequence
$N_\nu\uparrow\infty$. For every coordinate $j$ and every $\nu$, reindex the
$j$-th copy of $[N_\nu]$ onto a finite set $I_{j,\nu}\subset\mathbb{N}$, with
the sets $I_{j,\nu}$ pairwise disjoint as $\nu$ varies. Let
$\widetilde E_\nu$ denote the corresponding reindexed copy of $E_{N_\nu}$ and
set
\[
 \Lambda_0:=\bigcup_{\nu\ge1}\widetilde E_\nu,
 \qquad {\mathrm{Diag}:=\{(n,\ldots,n):n\in\mathbb{N}\},\qquad
 \Lambda:=\Lambda_0\cup\mathrm{Diag}.}
\]
In both cases
\[
 |E_N|=N^d,\qquad \Delta(E_N)\le N^{d-1}.
\]
{Indeed, for $d=1$ every one-coordinate fiber of $E_N$
contains at most one point, while for $d=m$ every such fiber contains
exactly $N^{m-1}$ points.  Moreover, if $A_j\subset[N]$ and
$|A_j|\le n$, then in the diagonal case
\[
 |E_N\cap(A_1\times\cdots\times A_m)|
 \le \min_j|A_j|\le n,
\]
whereas in the full-product case
\[
 |E_N\cap(A_1\times\cdots\times A_m)|
 =\prod_{j=1}^m|A_j|\le n^m.
\]
Thus, in either case,}
\[
 |E_N\cap(A_1\times\cdots\times A_m)|\le n^d.
\]
{Taking $A_j=I_{j,\nu}$ and $n=N_\nu$ gives
\[
 \Lambda(N_\nu)\ge|\widetilde E_\nu|=N_\nu^d,
\]
so the block-size lower bound yields $\dim(\Lambda)\ge d$.} If
$A_j\subset\mathbb{N}$ and $|A_j|\le n$, put
$n_\nu:=\max_j|A_j\cap I_{j,\nu}|$. For each fixed $j$, the sets $I_{j,\nu}$ are disjoint in $\nu$, and hence
\[
 \sum_\nu n_\nu\le \sum_\nu\sum_{j=1}^m |A_j\cap I_{j,\nu}|
 =\sum_{j=1}^m\sum_\nu |A_j\cap I_{j,\nu}|\le mn.
\]
{Using the elementary block rectangle estimate established
above for each $E_N$, rather than Lemma~\ref{lem:sparse-random-blocks-general},
whose hypotheses require $1<d<m$, we obtain}
\[
 \Lambda_0(n)\le \sum_\nu n_\nu^d
 \le \left(\sum_\nu n_\nu\right)^d
 \le (mn)^d.
\]
{The estimate $\Lambda_0(n)\le(mn)^d$ just proved concerns $\Lambda_0$ only. Since the diagonal contributes at most $n$ points to any such box and $d\ge1$,
\[
 \Lambda(n)\le (mn)^d+n\le(m^d+1)n^d.
\]
Thus $\Lambda$ has exact combinatorial dimension $d$.}

{Fix $\mathbf p$ with $|1/\mathbf p|\le1/2$.  Since
$|E_N|=N^d$ and $\Delta(E_N)\le N^{d-1}$,
Lemma~\ref{lem:signed-HL-block}, now with $C_0=C_1=1$, provides real
signs on the block such that the resulting $m$-linear form $T_N$
satisfies, over both scalar fields,}
\[
 \|T_N\|\le C_{m,\mathbf p}
 N^{(d+1)/2-|1/\mathbf p|}.
\]
{The support $E_N$ is fixed; for each $\mathbf p$ the signs
may be chosen separately.} If $\sigma$ is admissible for $(\Lambda,\mathbf p)$ over $\mathbb K$, then the
$N^d$ coefficients of modulus one on the block give
\[
 N^{d/\sigma}
 \le C_\sigma C_{m,\mathbf p}
 N^{(d+1)/2-|1/\mathbf p|}.
\]
Letting $N\to\infty$ yields \eqref{eq:HL-upper-reciprocal-unified}. {Over $\mathbb C$, the reverse inequality follows from \cite[Theorem~1.1(a)]{BayartJEMS} when $|1/\mathbf p|<1/2$ and from \cite[Theorem~1.1(b)]{BayartJEMS} when $|1/\mathbf p|=1/2$;} over $\mathbb{R}$, Proposition~\ref{prop:HL-scalar-invariance} transfers the same critical exponent from the complex to the real scalar field. Thus the first line of \eqref{eq:main-HL-unified} also holds for $d=1$ and $d=m$ over both scalar fields.

{Finally, if $1/2<|1/\mathbf p|<1$, the diagonal $\mathrm{Diag}\subset\Lambda$ and the argument of Step~4 give
\[
 \mathrm{HL}_{\mathbb K}(\Lambda;\mathbf p)^{-1}
 =1-\left|\frac1{\mathbf p}\right|.
\]
This proves the second line of \eqref{eq:main-HL-unified} and completes the proof.}
\end{proof}

\section*{Acknowledgements}

E. Teixeira gratefully acknowledges the support of the Grayce B. Kerr Chair at Oklahoma State University.

\section*{Funding}

D. Pellegrino was partially supported by CNPq through Grants 406457/2023-9, 305807/2025-0, and 403964/2024-5. A. Raposo was partially supported by CNPq through Grants 406457/2023-9, 403964/2024-5, and 302341/2025-0.

\section*{Competing interests}

The authors declare that they have no competing interests.

\section*{Declaration of generative AI and AI-assisted technologies}

During the preparation of this manuscript, the authors used OpenAI's ChatGPT as a generative-AI tool for exploratory calculations, consistency checks, organization of arguments, and drafting and editorial assistance. All AI-assisted material retained in the manuscript was independently reviewed and verified by the authors. The manuscript was written and edited by the authors, who take full responsibility for the content of the work.

\end{document}